\documentclass[a4, frensh,  11pt]{amsart}
\usepackage{amssymb,amscd,amsthm,verbatim}
\usepackage{euscript}

\usepackage{amsmath}
\usepackage[all,  poly, knot, 2cell, arc, arrow, frame, rotate]{xy}
\usepackage{qsymbols}

\usepackage[dvips]{graphics}

\newtheorem{proposition}{Proposition}

\newtheorem{corollary}{Corollaire}
\newtheorem{affirmation}{Affirmation}

\newtheorem{lemma}{Lemme}

\newtheorem{definition}{D\'efinition}

\title[Expression de la diff\'erentielle $d_3$]
{Expression de la diff\'erentielle $d_3$ \\
de la suite spectrale de Hochschild-Serre \\ en cohomologie born\'ee
r\'eelle}

\author[A. Bouarich]
{A. Bouarich}

\address{ Universit\'e Sultan Moulay Slimane  \\ Facult\'e des Sciences et T\'echniques\\
 B.P. 523, Beni Mellal\\
 Maroc/Morocco.}

\email{bouarich1@yahoo.fr or bouarich@fstbm.ac.ma}


\def \hfl#1{\smash{\mathop{\hbox to 7mm{\leftarrowfill}}
\limits^{\scriptstyle#1}}} \def \rfl#1#2{\smash{\mathop{\hbox to
7mm{\rightarrowfill}} \limits^{#1}_{\scriptstyle#2}}}

\begin{document}
\maketitle


\begin{abstract}
For discrete groups, we construct two bounded cohomology classes with
coefficients in the second space of the reduced real $\ell_1$-homology.
Precisely,  we associate to any discrete group $G$ a bounded cohomology class
of degree two noted $\frak{g}_2\in H_b^2(G, \overline{H}_2^{\ell_1}(G, \mathbb
R))$. For $G$ and $\Pi$ groups and  $\theta : \Pi\rightarrow Out(G)$ any
homomorphism we associate  a bounded cohomology class of degree three noted
$[\theta]\in H_b^3(\Pi, \overline{H}_2^{\ell_1}(G, \mathbb R))$. When the outer
homomorphism $\theta : \Pi\rightarrow Out(G)$ induces an extension of $G$ by
$\Pi$ we show that the class $\frak{g}_2$ is $\Pi$-invariant and that  the
differential $d_3$ of Hochschild-Serre spectral sequence sends the class
$\frak{g}_2$ on the class $[\theta]$  : $d_3(\frak{g}_2)=[\theta]$. Moreover,
we show  that for any integer $n\geq 0$ the differential $d_3 : E_3^{n,
2}\rightarrow E_3^{n+3, 0}$ of Hochschild-Serre spectral sequence in real
bounded cohomology is given as a cup-product by the class $[\theta]$.
\end{abstract}



\section{Introduction}

\subsection{Motivation}

Soit $G$ un groupe discret et  $C_b^n(G; \mathbb R)$ l'espace vectoriel r\'eel
des $n$-cocha\^\i nes born\'ees non homog\`enes,
 $c : G^n\rightarrow \mathbb R$. La diff\'erentielle de degr\'e $n\geq 0$ d'une
 $n$-cocha\^\i ne $c$ est d\'efinie par,
\begin{enumerate}
\item Pour $n\geq 1$ et pour tout $(g_0, g_1, \dots, g_n)\in G^{n+1}$ on
    pose,
\begin{eqnarray}
 d_nc(g_0, g_1, \dots, g_n) &=& c(g_1, \dots, g_n)  +
 \sum_{i=1}^{i=n}(-1)^ic(g_0, g_1, \dots, g_{i-1}g_i, g_{i+1}, \dots, g_n)
 \nonumber\\
&+& (-1)^{n-1}c(g_0, g_1, \dots, g_{n-1})\nonumber
\end{eqnarray}

\item $d_0 : C_b^0(G, \mathbb R) \rightarrow \mathbb R$ est l'application
    nulle.
\end{enumerate}

L'homologie du complexe diff\'erentiel $(C_b^*(G; \mathbb R), d_*)$ s'appelle
la cohomologie born\'ee r\'eelle du groupe $G$ au sens de Gromov (cf. \cite{I})
et est not\'ee  $H_b^*(G, \mathbb R)$.

Dans \cite{Bou} et \cite{Bou1}, en adaptant au contexte de cohomologie born\'ee
r\'eelle  les m\'ethodes de construction du deuxi\`eme et du troisi\`eme groupe
de cohomologie ordinaire d'un groupe $G$
 \`a la donn\'ee d'une extension de
groupes discrets $1 \rightarrow
G\stackrel{i}{\longrightarrow}\Gamma\stackrel{\sigma}{\longrightarrow}
\Pi\rightarrow 1$ nous avons associ\'e la suite exacte \`a quatre termes,
\begin{eqnarray}
0\longrightarrow H_b^2(\Pi, \mathbb R)\stackrel{\sigma_b}{\longrightarrow}H_b^2(\Gamma, \mathbb R)
\stackrel{i_b}{\longrightarrow} H_b^2(G, \mathbb R)^\Pi\stackrel{\delta}{\longrightarrow} H_b^3(\Pi, \mathbb R)
\end{eqnarray}
dans laquelle l'homomorphisme $H_b^2(G, \mathbb
R)^\Pi\stackrel{\delta}{\longrightarrow} H_b^3(\Pi, \mathbb R)$ s'appelle
op\'erateur de transgression, et o\`u $H_b^2(G, \mathbb R)^\Pi$ d\'esigne le
sous-espace des classes de cohomologie born\'ee r\'eelle de degr\'e deux
invariantes par l'action du groupe $\Pi$ qui est induite par la
repr\'esentation ext\'erieure $\theta : \Pi\longrightarrow Out(G)$ associ\'ee
\`a l'extension  $1 \rightarrow
G\stackrel{i}{\longrightarrow}\Gamma\stackrel{\sigma}{\longrightarrow}
\Pi\rightarrow 1$.

La suite exacte $(1)$ sugg\`ere qu'il existe une th\'eorie des suites spectrales en cohomologie born\'ee.
En effet, A. Noskov \cite{N1} (Voir aussi N. Monod et M. Burger \cite{Mo1}) a prouv\'e qu'on peut associer \`a une extension  de groupes discrets $1 \rightarrow
G\stackrel{i}{\longrightarrow}\Gamma\stackrel{\sigma}{\longrightarrow}
\Pi\rightarrow 1$ une suite spectrale de Hochschild-Serre en cohomologie born\'ee r\'eelle
$(E_{_{r}}^{^{p, q}}, d_r)$ qui converge vers la cohomologie born\'ee r\'eelle du groupe $\Gamma$. Cependant,  pour expliciter le second terme $E_2^{p, q}$,  A. Noskov    a stipul\'e dans \cite{N1}   que les espaces de  cohomologie born\'ee $H_b^q(G, \mathbb R)$ soient des espaces de Banach (voir aussi \cite{Mo1} cf. pr. 4.2.2 p. 264), or cette hypoth\`ese n'est pas toujours remplie  quand la dimension de l'espace vectoriel r\'eelle $H_b^q(G, \mathbb R)$ est infinie \cite{So}.

Dans \cite{Bou3}, pour contourner l'hypoth\`ese demand\'ee par \cite{N1} et \cite{Mo1},  nous nous sommes plac\'e dans la cat\'egorie des espaces vectoriels r\'eels semi-norm\'es pour prouver que tout complexe diff\'erentiel $(K^*, d_*)$ qui est muni d'une filtration positive d\'ecroisante r\'eguli\`ere induit une suite spectrale convergente $(E_r^{*,*}, d_r^{*,*})$ dont les termes sont des espaces vectoriels semi-norm\'es identifi\'es \`a une bijection lin\'eaire continue pr\`es (cf. 3.2.1).  Ainsi, par exemple,  \`a une extension de groupes discrets  $1 \longrightarrow G\stackrel{i}{\longrightarrow}\Gamma\stackrel{\sigma}{\longrightarrow} \Pi\longrightarrow 1$ nous pouvons  associer une suite spectrale de Hochschild-Serre $(E_{_{r}}^{^{p, q}},
d_r)$ dont les termes sont des espaces vectoriels semi-norm\'es et qui converge  vers la cohomologie born\'ee r\'eelle $H_b^{p+q}(\Gamma, \mathbb R)$. De plus, il existe une bijection canonique continue, non  n\'ec\'essairement bicontinue, qui est d\'efinie sur le second terme $E_2^{p, q}$  \`a valeurs dans l'espace vectoriel  r\'eel semi-norm\'e $H_b^p(\Pi, H_b^q(G, \mathbb R))$ de la cohomologie born\'ee avec coefficients ;  ceci  m\^eme si  l'espace vectoriel semi-norm\'e $H_b^q(G, \mathbb R)$ n'est pas s\'epar\'e (cf. 3.2.2).

En cons\'equence de ce r\'esultat nous pouvons   associer \`a toute extension
de groupes discrets  $1 \longrightarrow G\stackrel{i}{\longrightarrow}\Gamma
\stackrel{\sigma}{\longrightarrow} \Pi\longrightarrow 1$ la suite exacte \`a
cinq termes (cf. \cite{N1}, \cite{Mo1} et \cite{Bou3}) :
\begin{eqnarray}
0\longrightarrow H_b^2(\Pi, \mathbb R)\stackrel{\sigma_b}{\longrightarrow}H_b^2(\Gamma, \mathbb R)
\stackrel{i_b}{\longrightarrow} H_b^2(G, \mathbb R)^\Pi\stackrel{d_3}{\longrightarrow} H_b^3(\Pi,
\mathbb R)\stackrel{\sigma_b}{\longrightarrow} H_b^3(\Gamma, \mathbb R)
\end{eqnarray}
 qui diff\`ere de la suite exacte $(1)$  par le
terme suppl\'ementaire $H_b^3(\Gamma, \mathbb R)$ et  au lieu de l'op\'erateur
de transgression $\delta : H_b^2(G, \mathbb R)^\Pi\longrightarrow H_b^3(\Pi,
\mathbb R)$ nous avons la diff\'erentielle $d_{_3} : E_{_3}^{^{0, 2}}=H_b^2(G,
\mathbb R)^\Pi\longrightarrow E_{_3}^{^{3, 0}}=H_b^3(\Pi, \mathbb R)$. Ainsi,
suite \`a ces remarques, on se propose dans ce  travail de comparer les deux
op\'erateurs  $\delta : E_{_3}^{^{0, 2}}\longrightarrow E_{_3}^{^{3, 0}}$ et
$d_{_3} : E_{_3}^{^{0, 2}}\longrightarrow E_{_3}^{^{3, 0}}$ en suivant le plan
que nous d\'ecrirons dans le prochain paragraphe.

\subsection{Pr\'esentation des r\'esultats}

Dans la section 2, nous  \'etudions la notion d'homologie $\ell_1$ d'un groupe
discret $G$ \`a coefficients dans un $G$-module de Banach $V$. Les espaces
d'homologie $\ell_1$ seront not\'es $H_*^{\ell_1}(G, V)$ tandis que les espaces
 de Banach d'homologie $\ell_1$-r\'eduite  seront not\'es  $\overline{H}_*^{\ell_1}(G, V)$.

Afin de rendre le contenu de l'article auto suffisant,  nous allons consacrer la section 3  \`a un bref rappel sur quelques \'el\'ements de la cohomologie born\'ee utiles pour ce travail. Plus pr\'ecis\'ement, nous  rappelons la notion de la cohomologie
born\'ee d'un groupe discret \`a coefficients dans un module de Banach $V$ (cf.
\cite{Bou3}, \cite{I} et \cite{N}), nous d\'ecrivons la construction des termes d'une  suite sp\'ectrale $(E_r^{p, q}, d_r^{p, q})$  et nous expliquerons  aussi la
relation entre les quasi-morphismes et les $2$-cocycles born\'es r\'eels. Ensuite, nous d\'emontrerons notre premier
r\'esultat principal~:

\newtheorem*{thmprA}{Th\'eor\`eme principal A}\begin{thmprA}
Pour tout groupe discret $G$ il existe une unique classe de cohomologie
born\'ee \`a coefficients triviaux not\'ee,
 $\mathbf{g}_{_{2}}\in H_b^2(G, \overline{H}_2^{\ell_1}(G, \mathbb R))$,  qui poss\`ede les deux propri\'et\'es
  suivantes :
\begin{enumerate}
\item $\mathbf{g}_{_{2}}$ est nulle si et seulement si le second groupe de
    cohomologie born\'ee $H_b^2(G, \mathbb R)$ est nul.

\item  Pour toute classe de cohomologie born\'ee r\'eelle $x\in H_b^2(G,
    \mathbb R)$  on a la relation,
$$x\cup \mathbf{g}_{_{2}} = x
$$
o\`u le cup-produit $\cup$ est d\'efini par l'entrelacement naturel
(dualit\'e) entre les espaces de Banach $H_b^2(G, \mathbb R)$ et
$\overline{H}_2^{\ell_1}(G, \mathbb R)$.
\end{enumerate}
\end{thmprA}

Dans la section 4, \`a partir d'une repr\'esentation ext\'erieure   $\theta :
\Pi\rightarrow Out(G)$ nous construisons une classe de cohomologie born\'ee
de degr\'e trois not\'ee $[\theta]\in H_b^3(\Pi, \overline{H}_2^{\ell_1}(G,
\mathbb R))$,  o\`u l'action du groupe $\Pi$ sur $\overline{H}_2^{\ell_1}(G,
\mathbb R)$ est celle induite par la repr\'esentation ext\'erieure $\theta$.

Pour  construire la classe de cohomologie born\'ee avec coefficients,
$[\theta]\in H_b^3(\Pi, \overline{H}_2^{\ell_1}(G,  \mathbb R))$,  nous avons
adapt\'e au contexte de la cohomologie born\'ee avec coefficients les
m\'ethodes permettant la construction d'une  classe de cohomologie ordinaire de
degr\'e trois  \`a partir de  la repr\'esentation ext\'erieure $\theta :
\Pi\rightarrow Out(G)$ (cf. \cite{Ma} et \cite{Bro}).

Dans la section 5, en  supposant que la repr\'esentation ext\'erieure $\theta :
\Pi\rightarrow Out(G)$ est induite par une extension de groupes discrets $1
\rightarrow
G\stackrel{i}{\longrightarrow}\Gamma\stackrel{\sigma}{\longrightarrow}
\Pi\rightarrow 1$  nous  d\'emontrons que la classe de cohomologie born\'ee
$\mathbf{g}_{_{2}}\in H_b^2(G, \overline{H}_2^{\ell_1}(G, \mathbb R))$ est
$\Pi$-invariante. Il r\'esulte de cette invariance que $\mathbf{g}_{_{2}}$
d\'efinit un \'el\'ement de $E_{_3}^{0,2}= H_b^2(G, \overline{H}_2^{\ell_1}(G,
\mathbb R))^\Pi$. Par ailleurs, en remarquant que la classe de cohomologie $[\theta]$ d\'efinit  un
\'el\'ement du terme $E_{_3}^{3, 0}$ nous d\'emontrons  le th\'eor\`eme suivant~:

\newtheorem*{thmprC}{Th\'eor\`eme principal B}\begin{thmprC}
La diff\'erentielle $d_{_3} : E_{_3}^{0,2}\longrightarrow E_{_3}^{3, 0}$ de la
suite spectrale de Hochschild-Serre associ\'ee \`a l'extension de groupes
discrets $1 \longrightarrow
G\stackrel{i}{\longrightarrow}\Gamma\stackrel{\sigma}{\longrightarrow}
\Pi\longrightarrow 1$ en cohomologie born\'ee \`a coefficients dans le
$\Pi$-module de Banach $\overline{H}_2^{\ell_1}(G, \mathbb R)$   envoie la
classe $\mathbf{g}_{_{2}}$  sur la classe $[\theta]$.
\end{thmprC}

Ensuite, gr\^ace au r\'esultat du th\'eor\`eme principal B nous  d\'emontrons
le th\'eor\`eme   suivant qui donne l'expression explicite de la
diff\'erentielle $d_3$ en fonction de la classe de cohomologie born\'ee
$[\theta]$.

\newtheorem*{thmprD}{Th\'eor\`eme principal C}\begin{thmprD}
Soit $\theta : \Pi\longrightarrow Out(G)$ une repr\'esentation ext\'erieure
induite par une extension de groupes discrets $1 \rightarrow
G\stackrel{i}{\longrightarrow}\Gamma\stackrel{\sigma}{\longrightarrow}
\Pi\rightarrow 1$
  et soit $[\theta]\in H_b^3(\Pi,
\overline{H}_2^{\ell_1}(G, \mathbb R))$ la classe de cohomologie born\'ee
associ\'ee \`a   $\theta$.  Alors, pour tout entier $n\geq 0$ la
diff\'erentielle $d_{_{3}} : E_{_{3}}^{{n, 2}}\rightarrow E_{_{3}}^{{n+3,
0}}$ de la suite spectrale de Hochschild-Serre en cohomologie born\'ee r\'eelle
est donn\'ee par l'expression : $$d_{_{3}}(x) = (-1)^n x\cup [\theta], \qquad
\forall x\in E_{_{3}}^{{n, 2}}.$$
\end{thmprD}

En utilisant l'expression explicite de l'op\'erateur  de transgression $\delta
: H_b^2(G, \mathbb R)^\Pi\rightarrow H_b^3(\Pi, \mathbb R)$,   rappel\'ee
ci-dessous \`a la suite du  corollaire 4 de la section 4.3,  nous d\'eduisons
le~:

\newtheorem*{corol1}{Corollaire A}\begin{corol1}
L'op\'erateur de transgression $\delta : H_b^2(G, \mathbb R)^\Pi\rightarrow
H_b^3(\Pi, \mathbb R)$ associ\'e \`a  l'extension  $1 \rightarrow
G\stackrel{i}{\longrightarrow}\Gamma\stackrel{\sigma}{\longrightarrow}
\Pi\rightarrow 1$    est \'egal \`a la diff\'erentielle $d_{_{3}} :
E_{_{3}}^{{0, 2}}\rightarrow E_{_{3}}^{{3, 0}}$.
\end{corol1}

Le r\'esultat du th\'eor\`eme principal C nous permet aussi   de d\'eduire le~:

\newtheorem*{corol2}{Corollaire B}\begin{corol2}
Si la classe de cohomologie born\'ee $[\theta]\in H_b^3(\Pi,
\overline{H}_2^{\ell_1}(G, \mathbb R))$ est nulle, alors l'op\'erateur
$0\rightarrow H_b^3(\Pi, \mathbb R)\stackrel{\sigma_b}{\longrightarrow}
H_b^3(\Gamma, \mathbb R)$ est injectif.
\end{corol2}

Enfin, notons que le r\'esultat du corollaire B nous sugg\`ere les deux
questions suivantes :

\noindent{\bf Question 1 :}  La classe  $[\theta]$ est-elle une obstruction \`a
l'exactitude \`a gauche du foncteur de cohomologie born\'ee r\'eelle $H_b^3(-,
\mathbb R)$ ? C'est-\`a-dire, si un homomorphisme  surjectif $\sigma :
\Gamma\rightarrow \Pi$ induit un op\'erateur injectif
$H_b^3(\Pi, \mathbb R)\stackrel{\sigma_b}{\longrightarrow}
H_b^3(\Gamma, \mathbb R)$ ; la classe $[\theta]\in H_b^3(\Pi,
\overline{H}_2^{\ell_1}(G, \mathbb R))$ est-elle nulle ?

\noindent{\bf Question 2 :}  La classe  $[\theta]\in H_b^3(\Pi,
\overline{H}_2^{\ell_1}(G, \mathbb R))$ est-elle triviale lorsque
$\theta(\Pi)\subset Out(G)$ est moyennable ?  En effet, dans
\cite{Bou2}, nous avons d\'emontr\'e que si l'image de la repr\'esentation
ext\'erieure  $\theta : \Pi\rightarrow Out(G)$ est moyennable alors pour
tout entier $n\geq 0$ l'op\'erateur $\sigma_b : H_b^n(\Pi, \mathbb
R)\stackrel{}{\rightarrow} H_b^n(\Gamma, \mathbb R)$ est injectif.

\section{Homologie $\ell_1$ d'un groupe discret}

\subsection{$G$-modules de Banach relativement projectifs} Soient $G$ un
groupe discret et  $E$ un espace de Banach.  On dira que  $E$ est un $G$-module
de Banach s'il est muni d'une action du groupe $G$ telle que chaque \'el\'ement
$g\in G$ induit un op\'erateur lin\'eaire born\'e $g : E\rightarrow E$ de norme
$\parallel g\parallel\leq 1$. On notera par $g.v$ l'action de $g$ sur un
\'el\'ement $v$ de $E$.

Les $G$-modules de Banach constituent une cat\'egorie dont les $G$-morphismes
sont tous les op\'erateurs lin\'eaires continus  $f : E\rightarrow F$   qui
sont $G$-\'equivariants,
$$f(g\cdot x) = g\cdot f(x), \qquad \forall x\in E, \forall g\in G.$$

\`A un $G$-module de Banach $E$ on associe un sous-espace vectoriel  de
vecteurs  $G$-invariants $E^G:=\{v\in E \ ;  g\cdot v=v \ \forall g\in G\}$  et
un espace quotient de Banach  de vecteurs $G$-coinvariants $E_{_{G}} :=
\displaystyle{E\over \overline{E}(G)}$ ; o\`u $\overline{E}(G)$ d\'esigne
l'adh\'erence du sous-espace vectoriel de $E$ engendr\'e par tous les vecteurs
$g\cdot v - v $ avec $g\in G$ et $v\in E$.

Notons que si on se donne deux $G$-modules de Banach  $E$ et $F$ on  d\'efinit
une structure de $G$-module de Banach sur  leur  produit tensoriel projectif
compl\'et\'e $E\widehat{\otimes}F$ (cf. \cite{Gro}) en posant~:
$$g\cdot (x\otimes y) := g\cdot x\otimes g^{-1}\cdot y, \qquad   \forall g\in G,
x\in E, y\in F$$
L'espace de Banach des vecteurs $G$-coinvariants du $G$-module de Banach
$E\widehat{\otimes}F$ sera d\'esign\'e par l'expression,
$E\widehat{\otimes}_{_{G}}F :=(E\widehat{\otimes}F)_{_{G}}$.

Soient $E$ et $X$ deux $G$-modules de Banach. On dira qu'un  $G$-morphisme
surjectif $p: E\rightarrow X$  est  admissible s'il existe un op\'erateur
lin\'eaire continu $r \in{\mathcal L}(X, E)$,  non n\'ecessairement
$G$-\'equivariant,   tel que $p\circ r=id_{_{X}}$. De m\^eme, on dira qu'un
$G$-module de Banach $V$ est relativement projectif si pour tout $G$-morphisme
surjectif admissible $\begin{xy} (0,0)*+{E}="a"; (15, 0)*+{X}="b";
(25,0)*+{0}="o" \ar@{->}^{p} "a" ; "b" <3pt> \ar@{<-}_{r} "a"; "b" <-3pt>
\ar@{->} "b"; "o"\end{xy}$ et pour tout $G$-morphisme $\alpha : V\rightarrow X$
il existe au moins un $G$-morphisme  $\beta : V\rightarrow E$ tel que $p\circ
\beta=\alpha$,
$$\begin{xy}
(15, 15)*+{V}="v"; (0,0)*+{E}="a"; (15, 0)*+{X}="b"; (25,0)*+{0.}="o"
\ar@{->}^{p} "a" ; "b" <3pt>
\ar@{<-}_{r} "a"; "b" <-3pt>
\ar@{->} "b"; "o"
\ar@{->}^{\alpha} "v"; "b"
\ar@{<--}^{\beta} "a"; "v"
\end{xy}$$

\'Etant donn\'e un groupe discret $G$, on d\'esigne par $C_n^{\ell_1}(G,
\mathbb R)$ le compl\'et\'e de l'espace vectoriel des $n$-cha\^\i nes r\'eelles
$C_n(G, \mathbb R):=\mathbb R[G^n]$ qui est engendr\'e par les \'el\'ements de
l'ensemble $G^n$ et est muni par la norme $\ell^1$,
\begin{eqnarray}
\textrm{ si } \quad z = \sum_{i=1}^{i=m}a_i(g_1^i, \cdots,
g_n^i)\in\mathbb R[G^n] \quad \textrm{ on pose } \quad \parallel
z\parallel_1 = \sum_{i=1}^{i=m}\mid a_i\mid \in\mathbb R^+.
\end{eqnarray}

Sur l'espace de Banach $C_n^{\ell_1}(G, \mathbb R)$ nous avons une structure de
$G$-module de Banach naturelle d\'efinie par la $G$-action suivante, $$g\cdot
(g_1, \cdots, g_n) :=(gg_1, \cdots, gg_n), \quad \forall g, g_1, \cdots, g_n\in
G$$

\begin{lemma} Pour tout groupe discret $G$ et pour tout  $G$-module
de Banach $V$,   le produit tensoriel projectif compl\'et\'e $C_n^{\ell_1}(G,
\mathbb R)\widehat{\otimes}V$ est un $G$-module de Banach relativement
projectif.
\end{lemma}

\begin{proof}[D\'emonstration] Puisque pour tout couple d'entiers naturels $m\geq 1$ et $n\geq 1$ les espaces de
Banach $C_m^{\ell_1}(G, \mathbb R)\widehat{\otimes} C_n^{\ell_1}(G, \mathbb R)$
et $C_{m+n}^{\ell_1}(G, \mathbb R)$ sont canoniquement  isomorphismes, il
suffit  de d\'emontrer le lemme pour $n=1$.

Soit  $\begin{xy} (0,0)*+{E}="a"; (15, 0)*+{X}="b"; (25,0)*+{0}="o"
\ar@{->}^{p} "a" ; "b" <3pt> \ar@{<-}_{r} "a"; "b" <-3pt> \ar@{->} "b";
"o"\end{xy}$ un $G$-morphisme surjectif admissible. Pour tout  $G$-morphisme
$\alpha : C_1^{\ell_1}(G, \mathbb R)\widehat{\otimes}V \rightarrow X$ et
pour tous $g\in G$ et $v\in V$ posons
$$\beta(g\otimes v) = g\cdot r(g^{-1}\cdot \alpha(g\otimes v)).$$

Les op\'erateurs  $\alpha$ et $p$ \'etant  $G$-\'equivariants, l'op\'erateur
continu $\beta : C_1^{\ell_1}(G, \mathbb R)\widehat{\otimes}V\rightarrow E$
v\'erifie l'identit\'e   $p\circ \beta= \alpha$. De plus, comme pour tous les
\'el\'ements $g$ et $h\in G$ et pour tout vecteur $v\in V$ on a,
\begin{eqnarray}
h\cdot \beta(h^{-1}\cdot (g\otimes v)) &=& h\cdot \beta(h^{-1}g\otimes  h\cdot v) \nonumber \\
&=& h \cdot [h^{-1}g \cdot r(g^{-1}h \cdot \alpha(h^{-1}g\otimes  h\cdot v))] \nonumber \\
&=& g\cdot r(g^{-1}h \cdot \alpha(h^{-1}\cdot(g\otimes v))) \nonumber \\
&=& g\cdot r(g^{-1}\cdot \alpha(g\otimes v))\nonumber \\
&=&\beta(g\otimes v)\nonumber   \end{eqnarray} on en d\'eduit que l'op\'erateur
continu $\beta : C_1^{\ell_1}(G, \mathbb R)\widehat{\otimes}V\rightarrow E$
est $G$-\'equivariant. Par suite, le $G$-module de Banach $C_1^{\ell_1}(G,
\mathbb R)\widehat{\otimes}V$ est relativement projectif.
\end{proof}

\subsection{R\'esolutions relativement projectives}

Soit $V$ un $G$-module de Banach. On appelle  $G$-r\'esolution homologique de
$V$ dans la cat\'egorie relative des $G$-modules de Banach la donn\'ee d'un
complexe diff\'erentiel $(K_*, d_*)$,
$$\begin{xy}
(-5,0)*+{}="w"; (10, 0)*+{K_3}="v"; (25, 0)*+{K_2}="k2"; (40, 0)*+{K_1}="k1";
(55, 0)*+{K_0}="k0"; (70, 0)*+{V}="e" ; (85, 0)*+{0}="t";
\ar@{.>} "w"; "v" \ar@{->}^{d_3} "v"; "k2"
\ar@{->}^{d_2} "k2"; "k1" \ar@{->}^{d_1} "k1";
"k0" \ar@{->}^{d_0=\varepsilon} "k0"; "e"
\ar@{->} "e"; "t"
 \end{xy}$$
dont les fl\`eches sont exactes (i.e. $Im d_{n+1}=Ker d_{n}$) et dont les
termes $K_n$ sont des $G$-modules de Banach. On rappelle que le complexe
diff\'erentiel $(K_*, d_*)$ poss\`ede une homotopie contractante s'il existe
une suite d'op\'erateurs continus $s_n : K_n\rightarrow K_{n+1}$ tels que
$\forall n\in \mathbb N, s_{n-1}\circ d_n+d_{n+1}\circ s_n=id_{K_n}$,
$d_{1}\circ s_0=id_{K_0}$ et tel que la norme $\parallel s_n\parallel\leq 1$,
$$\begin{xy}
(-5,0)*+{}="w"; (10, 0)*+{K_3}="v"; (25, 0)*+{K_2}="k2"; (40, 0)*+{K_1}="k1";
(55, 0)*+{K_0}="k0"; (70, 0)*+{V}="e" ; (85, 0)*+{0}="t";
\ar@{.>} "w"; "v" <3pt> \ar@{<.} "w"; "v" <-3pt>
\ar@{->}^{d_3} "v"; "k2" <3pt> \ar@{<-}_{s_2} "v"; "k2" <-3pt>
\ar@{->}^{d_2} "k2"; "k1" <3pt> \ar@{<-}_{s_1} "k2"; "k1" <-3pt>
\ar@{->}^{d_1} "k1"; "k0" <3pt> \ar@{<-}_{s_0} "k1"; "k0" <-3pt>
\ar@{->}^{d_0=\varepsilon} "k0"; "e"
\ar@{->} "e"; "t"
 \end{xy}.$$ Quand une $G$-r\'esolution
homologique $(K_*, d_*)$ d'un $G$-module de Banach $V$ poss\`ede une homotopie
contractante $s_* : K_*\rightarrow K_{*+1}$ on dira que $(K_*, d_*, s_*)$ est
une  r\'esolution homologique forte du  $G$-module de Banach $V$.

Dans ce qui va suivre on va se servir du lemme 1 pour associer \`a chaque
$G$-module de Banach une r\'esolution relativement projective forte.

D'abord, notons que lorsque le groupe $G$ agit trivialement sur  $\mathbb R$,
vue comme espace de Banach,  le lemme 1 permet de voir que pour tout entier
$n\geq 0$ l'espace de Banach des  $n$-cha\^\i nes non normalis\'ees
$$C_{n, 0}^{\ell_1}(G, \mathbb R):= C_{n+1}^{\ell_1}(G, \mathbb R)$$
est un $G$-module de Banach relativement projectif. En fait, si  on
consid\`ere la suite d'op\'erateurs diff\'erentiels continus de degr\'e $-1$,
$$\partial_n =\displaystyle{\sum_{i=0}^{i=n+1}}
(-1)^id_i : C_n^{\ell_1}(G, \mathbb R)\rightarrow
C_{n-1}^{\ell_1}(G, \mathbb R) \quad \textrm{o\`u} \quad d_i(g_0,
g_1, \cdots, g_n) = (g_0, \cdots, \widehat{g_i}, \cdots, g_n)$$
on obtient une $G$-r\'esolution forte $(C_{*, 0}^{\ell_1}(G, \mathbb R),
\partial_*)$ de $\mathbb R$. De plus, si pour tout entier $n\geq 0$ on pose
$s_n(g_0, \cdots, g_n) = (1, g_0, \cdots, g_n)$ on d\'efinit ainsi une
homotopie contractante qui rend le complexe diff\'erentiel des $G$-modules de
Banach $(C_{*, 0}^{\ell_1}(G, \mathbb R),
\partial_*, s_*)$  une r\'esolution homologique forte de l'espace
de Banach $\mathbb R$  qui s'appelle la   bar r\'esolution du groupe $G$ par
les cha\^\i nes  non normalis\'ees.

Dans le cas d'un  $G$-module de Banach  $V$ quelconque,  une  r\'esolution
homologique  forte peut \^etre construite de la mani\`ere suivante.

Soient $E$, $F$ et $V$ des espaces de Banach et  $q : F\rightarrow Q$ un
op\'erateur lin\'eaire continu  surjectif. D'apr\`es le th\'eor\`eme de
l'application ouverte (cf. \cite{Hb} et \cite{Yo}),  l'op\'erateur lin\'eaire
$q\widehat{\otimes}V : E\widehat {\otimes}V\rightarrow Q\widehat{\otimes}V$ qui
est obtenu  en  tensorisant l'op\'erateur $q$ par l'espace de Banach $V$ est
lui m\^eme continu et surjectif. De m\^eme,  le th\'eor\`eme l'application
ouverte permet de d\'eduire l'identification : $Ker(q\widehat{\otimes}V)\simeq
Ker(q)\widehat{\otimes}V$ (cf. prop. 3 de \cite{Gro} page 38). En cons\'equence
de ces remarques, on conclut  que le foncteur de tensorisation
$-\widehat{\otimes}V$ pr\'eserve les suites exactes dans la cat\'egorie des
espaces de Banach. D'o\`u la~:

\begin{proposition}
Soient $G$ un groupe discret et $V$ un $G$-module de Banach. Alors,  le
complexe diff\'erentiel $(C_{*, 0}^{\ell_1}(G, \mathbb R)\widehat{\otimes}V,
\partial_*, s_*)$ qui est obtenu en tensorisant la bar r\'esolution
$(C_{*,0}^{\ell_1}(G, \mathbb R), \partial_n*, s_*)$ du groupe $G$ par le
$G$-module de Banach $V$ est une r\'esolution forte relativement projective de
$V$.
\end{proposition}

Le lemme suivant nous  montre que toutes les r\'esolutions fortes relativement
projectives d'un $G$-module de Banach donn\'e  sont homotopiquement
\'equivalentes. Sa preuve sera omise,  car,  elle est analogue au cas classique
(cf. \cite{Gu} et \cite{Ma}).

\begin{lemma}
Soient $(U_*, d_*, s_*)$ une r\'esolution homologique forte de $G$-modules de
Banach et $(V_*, d_*)$ un complexe diff\'erentiel homologique de $G$-modules de
Banach relativement projectifs. Alors, tout $G$-morphisme  continu $u :
V_{_{-1}}\rightarrow U_{_{-1}}$  se prolonge en un morphisme $u_* :
V_*\rightarrow U_*$ de complexes diff\'erentiels de $G$-modules de Banach
unique \`a homotopie pr\`es.
\end{lemma}

\subsection{D\'efinition et propri\'et\'es  de l'homologie $\ell_1$}

\begin{definition}
Soient $V$ un $G$-module de Banach et $(P_*, d_*)$ une r\'esolution homologique
forte relativement projective de $V$. Les groupes d'homologie du sous-complexe
diff\'erentiel des vecteurs $G$-coinvariants $(P_*)_{_{G}}$ s'appellent espaces
de $\ell_1$-homologie  du groupe $G$ \`a coefficients dans le $G$-module de
Banach $V$ et se notent  $H_*^{\ell_1}(G, V):=H_*((P_*)_{_{G}})$.
\end{definition}

D'apr\`es le  lemme 2 on sait  que toutes les r\'esolutions fortes relativement
projectives d'un $G$-module de Banach $V$ sont homotopiquement \'equivalentes.
Donc, les espaces d'homologie $H_*^{\ell_1}(G, V)$ ne d\'ependent pas de la
r\'esolution relativement projective choisie. En particulier, si on prend  la
bar r\'esolution $(C_{*,0}^{\ell_1}(G, \mathbb R)\widehat{\otimes}V,
\partial_*, s_*)$ associ\'ee au groupe $G$  on obtient un isomorphisme
canonique  :
\begin{eqnarray}
H_n^{\ell_1}(G, V) \simeq H_n(C_{*, 0}^{\ell_1}(G, \mathbb
R)\widehat{\otimes}_{_{G}}V, \partial_*)
\end{eqnarray}
o\`u $\partial_n$ d\'esigne l'op\'erateur diff\'erentiel du complexe des
n-cha\^\i nes { normalis\'ees } de type $\ell_1$,
$$C_{n+1, 0}^{\ell_1}(G, \mathbb
R)\widehat{\otimes}_{_{G}}V\simeq C_n^{\ell_1}(G, \mathbb
R)\widehat{\otimes}V$$ dont  l'expression explicite est donn\'ee pour tout
$[g_1\mid \cdots\mid g_n]\otimes v\in C_n^{\ell_1}(G, \mathbb
R)\widehat{\otimes}V$ par :
\begin{eqnarray*}\partial_n([g_1\mid \cdots\mid g_n]\otimes v) &=&
[g_2\mid \cdots\mid g_n]\otimes g_1\cdot v
-[g_1g_2\mid  \cdots\mid g_n]\otimes v \\ &+& \cdots +
(-1)^n[g_1\mid \cdots\mid g_{n-1}]\otimes v\end{eqnarray*}

Maintenant, gr\^ace \`a l'isomorphisme (4) on voit ais\'ement que le groupe
 $H_0^{\ell_1}(G, V) = V_{_{G}}$. De m\^eme, si on suppose
que le groupe  $G$ agit trivialement sur l'espace de Banach $V$ on en d\'eduit
que le groupe $H_1^{\ell_1}(G, V)=0$ (cf. \cite{Mi}). En effet, si pour tous
les \'el\'ements $g\in G$ et $v\in V$ on pose,
\begin{eqnarray}
\mathbf{m}(g, v) = \sum_{n\geq 1}{1\over 2^n}[g^{2^{n-1}}\mid
g^{2^{n-1}}]\otimes v\in C_2^{\ell_1}(G, \mathbb R)\widehat{\otimes}V
\end{eqnarray}
on obtient une  $2$-cha\^\i ne normalis\'ee dont le  bord est \'egal \`a
$\partial_2(\mathbf{m}(g, v)) = [g]\otimes v\in B_1^{\ell_1}(G, V).$

Rappelons aussi que puisque l'image d'un op\'erateur lin\'eaire continu n'est
pas en g\'en\'eral ferm\'ee  ceci implique  que la restriction  de la norme
$\ell_1$ sur l'espace des $\ell_1$-cycles $Z_n^{\ell_1}(G, V)$ peut
d\'eg\'en\'erer sur le  quotient, $\displaystyle\frac{Z_n^{\ell_1}(G,
V)}{B_n^{\ell_1}(G, V)}= H_n^{\ell_1}(G, V)$. En effet, S. Soma a construit des
exemples de  vari\'et\'es de dimension trois $M^3$ pour  lesquelles la
semi-norme $\parallel\cdot\parallel_1$ induite  sur l'espace $H_2^{\ell_1}(M,
\mathbb R)$ d\'eg\'en\`ere (cf. \cite{So}).

\begin{definition}
Le groupe quotient de l'espace $H_n^{\ell_1}(G, V)$ par le noyau
 $Ker(\parallel\cdot\parallel_1)$ s'appelle groupe
d'homologie $\ell_1$-r\'eduite du groupe $G$ \`a coefficients dans le
$G$-module de Banach $V$  et se note $\overline{H}_n^{\ell_1}(G, V)$.
\end{definition}

Pour finir cette section nous donnerons deux r\'esultats classiques qui
concernent l'action d'un groupe discret sur ses propres espaces de
$\ell_1$-homologie.

\begin{proposition}
Soit $V$ un $G$-module de Banach. Alors, pour tout entier $n\geq 0$ et pour
tout \'el\'ement $g\in G$ l'automorphisme int\'erieur $i_g : G\rightarrow G$
induit l'application identique sur l'espace d'homologie $H_n^{\ell_1}(G, V)$
(resp. $\overline{H}_n^{\ell_1}(G, V)$).
\end{proposition}

\begin{proof}[D\'emonstration] Soient $V$ un $G$-module de Banach et $(P_n*, \partial_*, s_*)$ une r\'esolution homo-logique forte relativement projective de $V$. Pour chaque $g_0\in G$  et $n\in\mathbb N$   posons  pour tout vecteur $v_n\in P_n$, $U_n(v_n) = g_0\cdot v_n$. Ainsi, puisque les morphismes $U_n : P_n\rightarrow P_n$ v\'erifient les deux relations,
$$U_n(g\cdot v_n) = i_{_{g_0}}(g)\cdot U_n(v_n) \quad \textrm{et}\quad U_n(v_n) = v_n + (g_0\cdot v_n-v_n), \quad \forall v_n\in P_n, \forall g\in G$$ on  en d\'eduit que  $U_n$ induit l'identit\'e sur
l'espace des vecteurs $G$-coinvariants $(P_n)_{_{G}}$. Par cons\'equent,
l'automorphisme int\'erieur  $i_{_{g_0}} : G\rightarrow G$ induit l'application
identique sur les groupes d'homologie  $H_n((P_*)_{_{G}})=H_n^{\ell_1}(G, V)$.
\end{proof}

\begin{corollary}
Soit $\Gamma$ un groupe discret et $G$ sous-groupe normal de $\Gamma$. Alors,
la conjugaison dans le groupe $\Gamma$ induit une action  naturelle du groupe
quotient $\displaystyle\frac{\Gamma}{G}$ sur les espaces d'homologie
$H_*^{\ell_1}(G, V)$ (resp. $\overline{H}_*^{\ell_1}(G, V)$).
\end{corollary}

\section{Cohomologie born\'ee d'un groupe discret}

\subsection{D\'efinition et propri\'et\'es}

\'Etant donn\'e un $G$-module de Banach $V$, pour tout entier $n\geq 0$ on
d\'esigne  par $C_b^n(G, V)$ l'espace de Banach des $n$-cocha\^\i nes  non
homog\`enes born\'ees $f : G^n\rightarrow V$. C'est-\`a-dire, il existe un
r\'eel $k>0$ tel que la norme $\ell_\infty$,
$$\parallel f\parallel_\infty = \sup\{\parallel f(g_1,\cdots, g_n)\parallel /
g_1, \cdots, g_n\in G\}\leq k.$$

Notons que la norme $\|\cdot\|_\infty$  permet de voir que l'espace $C_b^n(G,
V)$ est isomorphe \`a l'espace de Banach des op\'erateurs lin\'eaires continus
$\mathcal{L}(\mathbb R[G^n], V)$.

Sur le complexe des cocha\^\i nes born\'ees non homog\`enes $C_b^*(G, V)$ on
d\'efinit une diff\'erentielle $d_n : C_b^n(G, V)\rightarrow C_b^{n+1}(G, V)$
par l'expression,
\begin{eqnarray}
d_n(f)(g_1,\cdots ,g_{n{\scriptstyle+}1})&=& g_1.f(g_2,\cdots ,
   g_{n{\scriptstyle+}1}) \nonumber \\
   &+& \sum_{i=1}^{i=n}(-1)^i f(g_1,\cdots ,g_ig_{i{\scriptstyle+}1}, \cdots,
g_{n{\scriptstyle+}1}){\scriptstyle+}(-1)^{n{\scriptstyle+}1}f(g_1, \cdots,
g_n) \nonumber
\end{eqnarray}

Les groupes quotient $H_b^n(G, V):=\displaystyle\frac{Ker(d_n)}{Im(d_{n-1})}$
s'appellent espaces de cohomologie born\'ee du groupe $G$ \`a valeurs dans le
$G$-module de Banach $V$. Notons qu'on a  $H_b^0(G, V)=V^{G}$ et que
$$H_b^1(G, V)=
\displaystyle{ \{f\in C_b^1(G, V)\mid f(g_1g_2)=g_1.f(g_2)+f(g_1),
 \forall g_1, g_2\in G\} \over \{f\in C_b^1(G, V) \mid \exists v\in V, \forall
 g\in G, f(g)=g.v-v\}}.$$

En particulier, quand le groupe $G$ agit trivialement sur l'espace vectoriel
r\'eel $V$ on voit qu'on a $H_b^1(G, V){\scriptstyle=}0$ puisque  $\{0\}$ est
le seul sous-groupe  additif born\'e dans  $V$. Notons aussi que la norme
$\|\cdot\|_\infty$ d\'efinie sur les espaces $C_b^n(G, V)$ induit une
semi-norme sur les espaces de cohomologie born\'ee $H_b^n(G, V)$. Quand le
groupe $G$ agit trivialement sur $V=\mathbb R$,  N. Ivanov \`a d\'emontr\'e
que $H_b^2(G, \mathbb R)$ est un espace de Banach (cf. \cite{Iv}).

Les groupes de cohomologie born\'ee $H_b^n(G, V)$ peuvent \^etre d\'efinis \`a
partir du sous-complexe diff\'erentiel des coha\^ines  homog\`enes
$G$-invariantes,
$$(C_{b, 0}^n(G, V))^G:=(\mathcal{L}(\mathbb R[G^{n+1}], V))^G
\simeq C_b^n(G, V)$$ o\`u la diff\'erentielle d'une $n$-cocha\^\i ne homog\`ene
$f : G^{n+1}\rightarrow V$ est donn\'ee par l'expression,
\begin{eqnarray}
\delta_n(f)(g_0,g_1,\cdots ,g_{n+1})=\sum_{i=0}^{i=n+1}(-1)^if(g_0,\cdots
,{\widehat g_i},\cdots ,
g_{n+1}), \nonumber
\end{eqnarray}
o\`u $\widehat g_i$ veut dire omettre la composante $g_i$.

Pour finir ce paragraphe on donnera quelques r\'esultats utiles pour la suite
de ce travail dont les d\'emonstrations se trouvent dans l'article  \cite{I} ou
\cite{Bou3}.

\begin{proposition}
Soient $G$ un groupe discret et $V$ un $G$-module de Banach. Pour tout
\'el\'ement
 $g_0$ fix\'e dans le groupe $G$ on note  $i_{g_0} : G \rightarrow G$
 l'automorphisme int\'erieur qui lui est associ\'e. Alors l'op\'erateur $({i_{g_0}})_b :
H_b^n(G, V) \rightarrow H_b^n(G, V)$ est \'egal \`a l'identit\'e.
\end{proposition}

\begin{corollary}
Soit $\Gamma$ un groupe discret. Pour tout sous-groupe normal
$G\subseteq\Gamma$ la conjugaison dans le groupe $\Gamma$ induit une action
isom\'etrique du groupe quotient $\displaystyle\frac{\Gamma}{G}$ sur les
espaces semi-norm\'es $H_b^n(G, V)$.
 \end{corollary}

La proposition 3 et son  corollaire 2 nous permettent de d\'eduire la
proposition suivante :

\begin{proposition}
Soit $\Gamma$ un groupe discret. Pour tout sous-groupe normal $G\subseteq
\Gamma$ et pour tout entier $n\geq 1$ l'injection canonique $i : G \rightarrow
\Gamma$ induit un op\'erateur $i_b : H_b^n(\Gamma, V) \rightarrow H_b^n(G, V)$
qui prend ses valeurs dans le sous-espace des vecteurs
$\displaystyle\frac{\Gamma}{G}$-invariants. C'est-\`a-dire on a
$Im(i_b)\subseteq H_b^n(G, V)^{^{\Gamma/G}}$.
 \end{proposition}

\subsection{Suites spectrales en cohomologie born\'ee} Dans ce paragraphe,  nous donnerons  un bref rappel sur la th\'eorie des suites specrales d\'evelopp\'ee dans la cat\'egorie des complexes diff\'erentiels semi-norm\'es. Ensuite, avec la donn\'ee d'une extension de groupes discrets $1\rightarrow G\stackrel{i}{\rightarrow}\Gamma\stackrel{\sigma}{\rightarrow}\Pi\rightarrow 1$,
 nous d\'ecrirons le complexe diff\'erentiel  double qui induira la suite spectrale de Hochschilde-Serre  en cohomologie born\'ee $(E_r^{p, q}, d_r^{p, q})$ dont la diff\'erentielle $d_3^{n,0} : E_3^{n, 2}\longrightarrow E_r^{n+3, 0}$ sera explicit\'ee  dans la section 5.

\subsubsection{Construction des suites spectrales}
Soit $(K^*, d_*)$ un complexe diff\'erentiel  de degr\'e $+1$ (i.e. cohomologique) dont le terme g\'en\'eral $K^n$ est un espace vectoriel r\'eel  semi-norm\'e et  $d_n : K^n\rightarrow K^{n+1}$ est un op\'erateur lin\'eaire continu tel que $d_{n+1}\circ d_n=0$.

On dira  que le complexe diff\'erentiel $(K^*, d_*)$ poss\`ede une
filtration positive d\'ecroissante  si pour tout entier $n\geq 0$, il
existe une famille de sous-espaces vectoriels  $F^pK^n\subset
F^{p-1}K^n$ tels que $d_n(F^pK^n)\subset F^{p}K^{n+1}$  avec $F^pK^*=K^*$ si $p\leq 0$. Les termes $F^pK^n$ seront donc munis par la topologie induite.

Il est clair que  les inclusions de complexes diff\'erentiels $F^pK^*\stackrel{j}{\hookrightarrow}K^*$  induisent des op\'erateurs lin\'eaires continus, $j^p : H^n(F^pK^*, d_*)
\rightarrow H^n(K^*, d_*)$, et que  les images
\begin{eqnarray*}
F^pH^n(K^*, d_*) :=Im(j^p) \subseteq   H^n(K^*, d_*)  \quad \textrm{ avec } \quad F^0H^n(K^*, d_*)=H^n(K^*, d_*)
\end{eqnarray*}
d\'efinissent une filtration positive d\'ecroissante de l'espace vectoriel $H^n(K^*, d_*)$,
$$ \cdots \subseteq F^{p+1}H^n(K^*, d_*)\subseteq F^pH^n(K^*, d_*)\subseteq\cdots \subseteq F^1H^n(K^*, d_*)\subseteq  F^0H^n(K^*, d_*)$$
dont le nombre des termes n'est pas fini en g\'en\'eral.

En fait, si on suppose que la filtration positive d\'ecroissante  $F^pK^*$ est r\'eguli\`ere i.e.~:
\begin{eqnarray*}
\forall q\geq 0, \quad \exists n(q)\geq 0,  \quad \forall p>u(q) \qquad \Longrightarrow \qquad
F^pK^q =0
\end{eqnarray*}
il en r\'esulte que pour chaque entier $n\geq 0$ fix\'e  la famille   $\{F^pK^n\ ; p\in\mathbb N\}$  est finie et que par cons\'equent  la famille d\'ecroissante de sous-espaces vectoriels $\{F^pH^n(K^*, d_*)\ ; p\in\mathbb N\}$  est  finie.

Consid\'erons un  complexe diff\'erentiel   $(K^*, d_*)$ d'espaces vectoriels r\'eels semi-norm\'es    et  supposons qu'il est muni d'une filtartion  positive d\'ecroissante r\'eguli\`ere not\'ee $F^pK^*$. Pour tout entier $r\geq 0$  on d\'efinit  les
sous-espaces vectoriels topologiques, \begin{itemize}
\item $Z_r^{p, q}=\{x\in F^pK^{p+q}\mid d_{_{p+q}}x\in
F^{p+r}K^{p+q+1}\}$ ; \item $B_r^{p, q}=\{x\in F^pK^{p+q}\mid \exists
y\in F^{p-r}K^{p+q-1}, d_{_{p+q-1}}y=x\}$.
\end{itemize}

Observons  que par construction des sous-espaces $Z_r^{p, q}$ et $B_r^{p, q}$ on a  les   inclusions,
$$d_{_{p+q}}(Z_r^{p, q})\subset Z_r^{p+r, q-r+1}\quad \mbox{et}
\quad d_{_{p+q}}(Z_{r-1}^{p+1,q-1}+B_{r-1}^{p, q})\subset
Z_{r-1}^{p+1+r,q-r} + B_{r-1}^{p+r, q-r+1}$$
Donc, si on munit l'espace vectoriel quotient $E_r^{p, q}=\displaystyle{Z_r^{p, q}\over Z_{r-1}^{p+1,
q-1}+ B_{r-1}^{p,q}}$ par la semi-norme de la topologique quotient on d\'eduit que  la diff\'erentielle $d_n : K^n
\rightarrow K^{n+1}$ induit une famille d'op\'erateurs lin\'eaires continus~:
$$\begin{array}{cccc}
d_r^{p, q} : & E_r^{p, q}& \longrightarrow & E_r^{p+r, q-r+1}\\
&  [x]& \longmapsto & [d_{p+q}(x)]\end{array}$$ de norme
$\parallel d_r^{p, q}\parallel \leq\parallel d_{_{p+q}}\parallel$ et
tels que $d_r^{p+r, q-r+1}\circ d_r^{p, q}=0$. En cons\'equence,  pour tout entier $r\geq 0$ le couple $(E_r^{*, *},
d_{r}^{*, *})$ est un complexe diff\'erentiel d'espaces vectoriels r\'eels
semi-norm\'es bi-gradu\'es.

Au   complexe diff\'erentiel filt\'e  $(K^*, d_*)$  on associe \'egalement  les familles de sous-espaces vectoriels topologiques suivants :
\begin{itemize}
\item $Z_{_{\infty}}^{p, q} = \{x\in F^pK^{p+q} / d_{_{p+q}}x=0\}$
; \item $B_{_{\infty}}^{p, q} = \{x\in F^pK^{p+q} /\exists y\in
K^{p+q-1}, d_{_{p+q-1}}y=x\}$ ; \item   $E_{_{\infty}}^{p,
q}=\displaystyle{Z_{_{\infty}}^{p, q}\over B_{_{\infty}}^{p, q} +
Z_{_{\infty}}^{p+1, q-1}}$ que l'on munit de la structure
topologique quotient.
\end{itemize}

Avec les notations ci-dessus on obtient donc les inclusions suivantes~:
$$ B_0^{p, q}\subseteq \cdots\subseteq B_r^{p, q}\subseteq B_{r+1}^{p, q} \subseteq \cdots\subseteq B_\infty^{p, q}\subseteq Z_\infty^{p, q}\subseteq  \cdots\subseteq  Z_{r+1}^{p, q}\subseteq Z_r^{p, q}\cdots \subseteq Z_0^{p, q}    $$
tel que pour $r\geq p$ on a~: $$B_\infty^{p, q}= \displaystyle \bigcup_{r\geq 0}B_r^{p, q}=B_r^{p, q}=B_{r+1}^{p, q}=\cdots =B_\infty^{p, q}$$  et par  r\'egularit\'e de la filtration $F^pK^*$ si on pose   $r_0=u(p+q+1)-p$ on v\'erifie que l'intersection~:
$$\displaystyle\bigcap_{s\geq 0}Z_s^{p, q} = Z_{r_0}^{p, q} = Z_{r_0+1}^{p, q} = \cdots = Z_\infty^{p, q}$$

En effet, gr\^ace \`a la r\'egularit\'e de la filtration $F^pK^*$  on d\'eduit  que la suite des sous-espaces vectoriels r\'eels $\{E_r^{p, q} \ ; \forall  r\in \mathbb N\}$ est stationaire. C'est-\`a-dire, \`a patir d'un certain rang  assez grand $r\geq 0$ on a,
$$E_{r}^{p, q} = E_{r+1}^{p,q}=\cdots= E_\infty^{p, q}$$

Suite \`a cette propri\'et\'e on dira que la famille de complexes diff\'erentiels bi-gradu\'es   $(E_r^{*, *}, d_r^{*,*})$  converge vers l'aboutissement $E_\infty^{*, *}$ et on d\'esignera ce fait par le symbole~:
$$ E_r^{p, q} \quad \Longrightarrow \quad E_\infty^{p, q} $$

Enfin, observons que puisque le sous-espace vectoriel  $Z_{_{\infty}}^{p, q}$ est contenu dans le sous-espace vectoriel des cocycles $\textrm{Ker}(d_{p+q})$  on obtient   une  surjection canonique continue, $$Z_{_{\infty}}^{p, q}\rightarrow
\displaystyle{F^pH^{p+q}(K^*, d_*)\over F^{p+1}H^{p+q}(K^*,
d_*)}$$ dont le  noyau  est \'egal \`a la somme $B_{_{\infty}}^{p, q} +
Z_{_{\infty}}^{p+1, q-1}$. Par cons\'equent,  comme la filtration $F^pK^*$ est    r\'eguli\`ere  on d\'eduit que pour tout entier $n\geq 0$  la famille  des  bijections  lin\'eaires  continues induites,  $\{E_{_{\infty}}^{p, n-p} :=\displaystyle\frac{Z_{_{\infty}}^{p, n-p}}{B_{_{\infty}}^{p, n-p} +
Z_{_{\infty}}^{p+1, n-p-1}} \rightarrow
\displaystyle{F^pH^{n}(K^*, d_*)\over F^{p+1}H^{n}(K^*, d_*)}\ ; \forall p\in\mathbb N\}$, est finie.

Avec les notations ci-dessus, pour tout complexe diff\'erentiel $(K^*, d_*)$  qui est muni d'une filtration positive  d\'ecroissante et r\'eguli\`ere $F^pK^*$  on a les affirmations suivantes (cf.  \cite{Bou3})~:
\begin{enumerate}
\item
Il existe une bijection canonique continue d\'efinie sur le premier terme $E_1^{p,
q}$ dans l'espace d'homologie relative
$H^{p+q}(\displaystyle{F^pK^*/F^{p+1}K^{*}})$  qui envoie la
diff\'erentielle $d_1$ sur l'op\'erateur de bord (i.e. connexion) associ\'e au triplet
$(F^pK^*, F^{p+1}K^*, F^{p+2}K^*)$.

\item Pour chaque entier  $r\geq 0$ il existe une bijection
 continue, $E_{_{r+1}}^{p, q}
\stackrel{\sim}{\rightarrow} H^{p, q}(E_{_{r}}^{*, *}, d_r).$ Autrement dit, la famille de complexes diff\'erentiels bi-gradu\'es $(E_r^{*,*}, d_r^{*,*})$ est une suite spectrale dont les termes sont des espaces vectoriels r\'eels semi-norm\'es.

\item Il existe une bijection continue qui envoie la
somme directe topologique  $E_\infty^n = \displaystyle{\bigoplus_{p=0}^{p=n}}
E_\infty^{p, n-p}$ sur la somme directe topologique
$\displaystyle{\bigoplus_{p=0}^{p=n}} \displaystyle{F^pH^n(K^*,
d_*)\over F^{p+1}H^n(K^*, d_*)}\stackrel{\sim}{\rightarrow}H^{p+q}(K^*, d_*)$. En cons\'equence,
la suite spectrale $(E_r^{*, *}, d_r^{*, *})$ converge vers l'espace vectoriel de la cohomologie $H^n(K^*, d_*)$ i.e.~:
$$ E_r^{p, q} \quad \Longrightarrow \quad  E_\infty^{p+q}\stackrel{\sim}{\rightarrow} H^{p+q}(K^*, d_*)$$\end{enumerate}

\subsubsection{La suite  spectrale de Hochschild-Serre} La  donn\'ee d'une extension de groupes discrets $1 \rightarrow G\stackrel{i}{\longrightarrow}\Gamma
\stackrel{\sigma}{\longrightarrow} \Pi\rightarrow 1$ et d'un $\Gamma$-module de Banach $V$ permet de construire  un complexe diff\'erentiel double  dont le terme g\'en\'eral est d\'efini par,
$$ \forall p, q\in\mathbb N, \quad K^{p, q}:=C_b^p(\Pi, U^q) \qquad \textrm{ o\`u }\qquad U^q:= {\mathcal L}_{G}
(\mathbb R[\Gamma^{q+1}], V)$$

Comme au paragraphe 3.1, puisque les \'el\'ements de l'espace vectoriel $K^{p, q}$ sont des cocha\^\i nes born\'ees,   on  d\'efinit une  diff\'erentielle  horizontale  $d_\Pi :
K^{p, q} \rightarrow K^{p+1, q}$ et une  diff\'erentielle  verticale $d_U : K^{p, q} \rightarrow
K^{p, q+1}$ telles que $$d_\Pi d_\Pi=0, \quad d_Ud_U=0 \quad \textrm{  et  } \quad d_\Pi d_U=d_Ud_\Pi$$
De m\^eme, au complexe diff\'erentiel double $(K^{*, *}, d_\Pi, d_U)$ on associe un complexe diff\'erentielle totale $(\textrm{Tot}(K^{*, *}), d_*)$ de degr\'e $+1$ o\`u pour tout entier $n\in\mathbb N$ on pose~:
$$\textrm{Tot}(K^{*, *})^n :=\bigoplus_{p+q=n}K^{p, q} \quad \textrm{ et  }
\quad d_* :=d_\Pi+(-1)^pd_U$$

Le complexe diff\'erentiel total $(\textrm{Tot}(K^{*, *}), d_*)$  poss\`ede deux filtrations positives d\'ecroissantes et r\'eguli\`eres   verticle et horizontale not\'ees respectivement~:
$$F_v^p\textrm{Tot}(K^{*, *}) :=\displaystyle\sum_{i\geq p}K^{i, j} \qquad  \textrm{et}
 \qquad  F_h^q\textrm{Tot}(K^{*, *}) :=\displaystyle\sum_{j\geq q} K^{i, j}$$

Donc, aux  filtartions   $(F_h^p\textrm{Tot}(K^{*, *}), d_*)$ et $(F_v^p\textrm{Tot}(K^{*, *}), d_*)$ on associe deux suites spectrales not\'ees respectivement $E_{r, h}^{*, *}$ et $E_{r, v}^{*, *}$  qui convergent vers la cohomologie du complexe
diff\'erentiel total, $(\textrm{Tot}(K^{*, *}), d_*= d_\Pi+(-1)^pd_U)$ et   v\'erifient les propri\'et\'es suivantes (cf. \cite{Bou3})~:
\begin{enumerate}
\item La suite spectrale $E_{r, h}^{*, *}$  d\'eg\'en\`ere au premier terme (i.e., $E_{r, h}^{p, q}=0, \forall p\geq 1, q\geq 0$) et  son aboutissement $E_{\infty, h}^n=H^n_b(\Gamma, V)$. Donc,  la suite spectrale $E_{r, h}^{p, q}$ converge vers l'espace de la cohomologie born\'ee $H_b^{p+q}(\Gamma, V)$.

\item Il existe une bijection  continue du  terme
$E_{1, v}^{p, q}$ dans l'espace vectoriel  semi-norm\'e $C_b^p(\Pi, H_b^q(G, V))$
des $p$-cocha\^\i nes non homog\`enes born\'ees sur le groupe $\Pi$ \`a
coefficients dans l'espace semi-norm\'e $H_b^q(G, V)$, en plus, la diff\'erentielle $d_{_{v, 1}}$ s'envoie sur la diff\'ererntielle $d_{_{\Pi}}$.

\item Il existe une bijection  continue du terme
$E_{2, v}^{p, q}$ dans $H_b^p(\Pi, H_b^q(G, V))$ espace de cohomologie
born\'ee associ\'e au complexe diff\'erentiel  $(C_b^*(\Pi, H_b^q(G,
V)), d_{_{\Pi}})$.

\item La suite spectrale $(E_{_{r, v}}^{p, q}, d_{_{r, v}})$ converge
vers la cohomologie born\'ee du groupe $\Gamma$ \`a coefficients dans
le $\Gamma$-module $V$ i.e.~:
$$ E_{2, v}^{p, q} \stackrel{\sim}{\rightarrow} H_b^p(\Pi, H_b^q(G, V)) \Longrightarrow
H_b^{p+q}(\Gamma, V)$$
\end{enumerate}

Pour finir ce paragraphe nous d\'emontrons le lemme suivant qui donne des renseigements  sur la structure des termes  $E_3^{n, 0}$ et $E_3^{n, 2}$  utiles pour la preuve des   th\'eor\`emes  B et C.

\begin{lemma} Si $V$ est un  $\Gamma$-module de Banach trivial alors   on a les assertions suivantes~:
\begin{enumerate}
\item Le terme $E_3^{n, 0}=E_2^{n, 0}$, donc il existe une bijection lin\'eaire continue sur le terme $E_3^{n, 0}$  \`a valeurs  dans l'espace de cohomologie born\'ee $H_b^n(\Pi, V)$.
\item Il existe une suite exacte  d'op\'erateurs lin\'eaires continues, $$E_2^{n-2, 3}\stackrel{d_2^{n-2, 3}}{\longrightarrow} E_2^{n, 2}\longrightarrow E_3^{n, 2}\rightarrow 0$$ En cons\'equence, il existe une bijection lin\'eaire continue sur le terme $E_3^{0, 2}$ \`a valeurs dans $H_b^2(G, V)^\Pi$ espace des classes de  cohomologie born\'ee $\Pi$-invariantes.
\end{enumerate}  \end{lemma}

\begin{proof}[D\'emonstration] 1) Rappelons que le  terme $E_3^{p, q}$ est \'egal \`a la cohomologie du complexe diff\'erentiel $(E_2^{*, *}, d_2^{*,*})$ (\`a une bijection continue pr\`es),  donc pour tout entier $n\geq 0$ on a~:
$$ E_3^{n, 0} \stackrel{\sim}{\rightarrow}   \displaystyle\frac{\textrm{Ker}(d_2^{n,0} : E_2^{n, 0}\rightarrow E_2^{n+2, -1})}{\textrm{Im}( d_2^{n-2,1} :  E_2^{n-2, 1}\rightarrow E_2^{n, 0})} \quad \textrm{ et } \quad E_3^{n, 2} \stackrel{\sim}{\rightarrow}
 \displaystyle\frac{\textrm{Ker}(d_2^{n,2} : E_2^{n, 2}\rightarrow E_2^{n+2, 1})}{\textrm{Im}( d_2^{n-2,3} :  E_2^{n-2, 3}\rightarrow E_2^{n, 2})}  $$

Puisque $V$ est un $\Gamma$-module trivial,  l'espace vectoriel $H_b^1(G, V)=\{0\}$ et par suite le terme $E_2^{n-2, 1}\stackrel{\sim}{\rightarrow} H_b^{n-2}(\Pi, H_b^1(G, V))=\{0\}$. De plus,   comme le terme $E_2^{n+2, -1}=\{0\}$
on d\'eduit que le terme $E_3^{n, 0}=E_2^{n, 0}\stackrel{\sim}{\rightarrow}H_b^n(\Pi, V)$.

\noindent 2) Puisque  le terme $E_2^{n+2, 1}\stackrel{\sim}{\rightarrow}H_b^{n+2}(\Pi, H_b^1(G, V))=\{0\}$ cela implique que  la suite d'op\'erateurs lin\'eaires continus   $E_2^{n-2, 3}\stackrel{d_2^{n-2,3}}{\rightarrow} E_2^{n, 2}\rightarrow E_3^{n, 2}\rightarrow 0$ est exacte. Ainsi, si on prend $n=0$  il s'ensuit que le terme $E_2^{n-2, 3}=\{0\}$ et que $E_3^{0, 2}=E_2^{0, 2}\stackrel{\sim}{\rightarrow}H_b^0(\Pi, H_b^2(G, V))$. Donc,  le terme  $E_3^{0, 2}\stackrel{\sim}{\rightarrow}H_b^2(G, V)^\Pi$.
\end{proof}

\subsection{Cup-produit}
Soient $U$, $V$ et $W$  trois  $G$-modules de Banach et $\mu : U\times
V\rightarrow W$ un op\'erateur bilin\'eaire continu $G$-\'equivariant.
Pour tout couple de cocha\^\i nes born\'ees non homog\`enes $f\in C_b^p(G, U)$
et $h\in C_b^q(G, V)$ on d\'efinit leurs  cup-produit $f\cup h\in C_b^{p+q}(G,
W)$ par la formule suivante (cf. \cite{Bro}) :
\begin{eqnarray}
f\cup h([g_1\mid \cdots\mid g_{_{p+q}}])=\mu(f([ g_1\mid \cdots\mid
g_p])\otimes
g_1g_2\cdots g_p\cdot h([ g_{_{p+1}}\mid \cdots\mid g_{_{p+q}}])
\end{eqnarray}

Il est facile de v\'erifier que le cup-produit $\cup : C_b^*(G, U) \times
C_b^*(G, V)\rightarrow C_b^*(G, W)$ est continu, associatif et commute avec la
diff\'erentielle $d_*$ dans la formule suivante,
\begin{eqnarray}
d_{p+q}(f\cup h) = d_p(f)\cup h + (-1)^pf\cup d_q(h)
\end{eqnarray}

Ainsi, en passant en cohomologie, l'expression $(6)$ induit un cup-produit sur
les espaces de cohomologie born\'ee qui sera  not\'e aussi~:
$$
\begin{array}{ccccccc}
\cup : &H_b^p(G, U)\times H_b^q(G, V)& \longrightarrow & H_b^{p+q}(G, W) \\
& ([f], [h])& \longmapsto &[f]\cup [h]:=[f\cup h]
\end{array}
$$

Dans le reste de ce paragraphe, on donnera des exemples de cup-produit utiles
pour la suite de ce travail.

Notons d'abord que gr\^ace \`a la fonctorialit\'e de la cohomologie born\'ee et
de l'homologie $\ell^1$, la donn\'ee d'une repr\'esentation ext\'erieure
$\theta : \Pi\rightarrow Out(G)$ permet de munir les deux espaces $H_b^2(G,
\mathbb R)$ et $\overline{H}_2^{\ell_1}(G, \mathbb R)$ d'une structure de
$\Pi$-module de Banach. Notons aussi que le crochet de dualit\'e $<, > :
H_b^2(G, \mathbb R)\times \overline{H}_2^{\ell_1}(G, \mathbb R) \rightarrow
\mathbb R$ (\'evaluation) r\'ealise une forme bilin\'eaire
$\Pi$-\'equivariante. En effet,  l'action de  $\Pi$ sur $\mathbb R$ \'etant
triviale   on aura pour tous les \'el\'ements $\alpha\in \Pi$, $x\in H_b^2(G,
\mathbb R)$ et $y\in \overline{H}_2^{\ell_1}(G, \mathbb R)$ :
$$<\theta(\alpha)_b(x), \theta(\alpha)_*^{-1}(y)>=<x, \theta(\alpha)_*\circ \theta(\alpha)_*^{-1}(y)>=<x, y>.$$

Par cons\'equent, l'expression $(6)$ induit  un cup-produit sur la cohomologie
born\'ee du groupe $\Pi$ qu'on notera~:
$$\cup : H_b^p(\Pi, \overline{H}_2^{\ell_1}(G, \mathbb R))\times
H_b^q(\Pi, H_b^2(G, \mathbb R)) \rightarrow H_b^{p+q}(\Pi, \mathbb R), \qquad
\forall p, q\in\mathbb N.$$

Ce cup-produit $\cup$  qu'on vient de d\'efinir poss\`ede deux cas particuliers
tr\`es utiles pour la suite de ce travail :

\noindent  1) Si on suppose $\Pi=G$ et que  $\theta=1$ est la repr\'esentation
triviale on obtient alors un cup-produit sur les espaces de cohomologie
born\'ee \`a coefficients triviaux :
$$\cup : H_b^p(G, \overline{H}_2^{\ell_1}(G, \mathbb R))\times
H_b^q(G, H_b^2(G, \mathbb R)) \longrightarrow H_b^{p+q}(G, \mathbb R), \qquad
\forall p, q\in\mathbb N.$$

\noindent 2) Supposons que la repr\'esentation ext\'erieure $\theta : \Pi
\rightarrow Out(G)$ est associ\'ee \`a une extension de groupes discrets
$1\rightarrow G\stackrel{i}{\rightarrow}  \Gamma\stackrel{\sigma}{\rightarrow}
\Pi\rightarrow 1$. Ensuite, consid\'erons les trois complexes diff\'erentiels
doubles $K^{p,q}=C_b^p(\Pi, \mathcal{L}_G(\mathbb R[\Gamma^{q+1}], \mathbb R))
$, $K_\infty^{p,q}=C_b^p(\Pi, \mathcal{L}_G(\mathbb R[\Gamma^{q+1}], H_b^2(G,
\mathbb R))$ et $K_{\ell_1}^{p,q}=C_b^p(\Pi, \mathcal{L}_G(\mathbb
R[\Gamma^{q+1}], \overline{H}_2^{\ell_1}(G, \mathbb R))$ o\`u chacun d'eux est
muni de la filtration positive d\'ecroissante naturelle  suivant le degr\'e $p$ (cf.  3.2.2). Ainsi, avec
ces donn\'ees,   on obtient  trois suites  spectrales de
Hochschild-Serre not\'ees respectivement   $$(E_{_{r}}^{p, q}, d_r), \qquad
(E_{_{r, \infty}}^{p, q},  d_{_{r, \infty}}) \qquad \textrm{ et  } \qquad
(E_{_{r, \ell_1}}^{p, q}, d_{_{r, \ell_1}})$$

Enfin, si on consid\`ere le crochet de dualit\'e $<, > : H_b^2(G, \mathbb
R)\times \overline{H}_2^{\ell_1}(G, \mathbb R) \rightarrow \mathbb R$ on
obtient un cup-produit au niveau des complexes diff\'erentiels doubles, $\cup :
K_\infty^{p,q}\times K_{\ell_1}^{p,q}\rightarrow K^{p,q}$, qui induit par
suite un cup-produit au niveau des suites spectrales,
$$\cup : E_{_{r, \infty}}^{p, q}\times E_{_{r, \ell_1}}^{p', q'}\rightarrow E_{_{r}}^{p+p', q+q'}$$
qui commute avec les trois  diff\'erentielles $d_{_{r, \infty}}$, $d_{_{r,
\ell_1}}$ et $d_{_{r}}$ dans la relation suivante :
\begin{eqnarray}
d_r^{p+p', q+q'}(x_{_{\infty}}^{p, q}\cup x_{_{\ell_1}}^{p', q'}) =
d_{_{r, \infty}}^{p, q}(x_{_{\infty}}^{p, q})\cup x_{_{\ell_1}}^{p', q'} +
(-1)^{p+q}x_{_{\infty}}^{p, q}\cup d_{_{r, \ell_1}}^{p', q'}(x_{_{\ell_1}}^{p',
q'}).\nonumber
\end{eqnarray}

\subsection{Quasi-morphismes et  $2$-cocycles born\'es}
Ce paragraphe est enti\`erement extrait de \cite{Bou} et \cite{Bou1}.

\subsubsection{G\'en\'eralit\'es sur les extensions centrales}

Soient $G$ un groupe discret et $c\in Z^2(G, \mathbb R)$ un 2-cocycle r\'eel
non d\'eg\'en\'er\'e i.e. $c(g, 1)=c(1, g)=0, \forall g\in G$.

Observons que si on munit le produit cart\'esien $\mathbb R\times G$ par la multiplication interne
$$(t, g)\cdot (u, h) := (t + u + c(g, h), gh)$$
on obtient un groupe not\'e,  $\mathbb R\times_c G$,  tel que l'injection canonique $j(t)=(t, 1)$ et la surjection canonique $p(t, g)=g$ deviennent des hommorphismes et que le sous-groupe image $j(\mathbb R)=\textrm{Ker}(p)$ est central dans le groupe $\mathbb R\times_c G$. Notons aussi que  si on remplace $c$ par le $2$-cocycle  $c'=c + df$, avec $f : G\rightarrow \mathbb R$ est une cocha\^\i nes telle que $f(1)=0$, on obtient  un isomorphisme de groupes $F : \mathbb R\times_{c'} G\longrightarrow \mathbb R\times_c G$ d\'efini par $F(t, g) = (t + f(g) , g)$ o\`u $F(t, 1)=(t, 1), \forall t\in\mathbb R$.

En cons\'equence, un $2$-cocycle $c\in Z^2(G, \mathbb R)$ induit une suite exacte courte
$$\xymatrix{ 0\ar@{->}[r]  & \mathbb R\ar@{->}[r]^{j} & \mathbb R\times_c G
\ar@{->}[r]^{p} & G\ar@{->}[r]& 1 }
 $$
qui r\'ealise une extension centrale du groue $G$ par $\mathbb R$ et  qui ne d\'epend que de la classe de cohomologie $[c]\in H^2(G, \mathbb R)$.

Inversement, \'etant donn\'ee une extension centrale $\xymatrix{ 0\ar@{->}[r]  & \mathbb R\ar@{->}[r]^{j} & \overline{G}\ar@{->}[r]^{p} & G\ar@{->}[r]& 1 }$,  en fixant  une section ensembliste $s : G \rightarrow {\overline G}$ de la projection $ p : {\overline G} \rightarrow G$ (i.e. $p\circ s=id$) telle que $s(1)=1$ on v\'erifie que l'expression \begin{eqnarray}
c(g_1, g_2)=s(g_1)s(g_2)s(g_1g_2)^{-1}
 \end{eqnarray}
d\'efinie un  $2$-cocycle non d\'eg\'en\'er\'e $c\in Z^2(G, \mathbb R)$ (cf. \cite{Bro} et \cite{Ma}).

D'autre part, puisque pour tout \'el\'ement ${\bar g}\in {\overline G}$, les deux
\'el\'ements $\bar g$ et $s\circ  p({\bar g})$ se proj\`etent {\it via} la surjection $p$ sur
$p({\bar g})\in G$, il existe un unique \'el\'ement central $\Phi(\bar
g)\in\mathbb R$ tel que :
\begin{eqnarray}
{\bar g}=s\circ p(\bar g)\Phi(\bar g).
\end{eqnarray}

\begin{affirmation} La 1-cocha\^\i ne r\'eelle $\Phi :
{\overline G} \rightarrow \mathbb R$ d\'efinie par (9) v\'erifie les
propri\'etes suivantes o\`u on consid\`ere $\mathbb R\simeq Ker(p)$ comme un
groupe additif~:
\begin{enumerate}
\item $\forall {\bar g_1}, \ {\bar g_2}\in {\overline
    G}, \quad p^{*}(c)({\bar g_1}, {\bar g_2})=\Phi({\bar g_1}{\bar
    g_2})-\Phi(\bar g_1)- \Phi(\bar g_2)$.
\item $\forall t\in \mathbb R, \quad  \Phi(t)=t$.
\item $\forall t\in \mathbb R, \forall {\bar g}\in {\overline G}, \quad
    \Phi({\bar g}t)=\Phi(\bar g)+t$.
\item $\phi\circ s=0$.
\end{enumerate}

En cons\'equence, la correspondance $F(\bar{g}) := (\Phi(\bar g), p(\bar g))$ est un isomorphisme d'extensions centrales de $\overline{G}$ dans $\mathbb R\times_c G$.
\end{affirmation}
\begin{proof}[D\'emonstration] 1) Soient ${\bar g_1}$ et ${\bar g_2}\in{
\overline G}$, des expressions  (8) et  (9) il r\'esulte que :
\begin{eqnarray}
p^{*}(c)({\bar g_1}, {\bar g_2})&= & s\circ p(\bar g_1)s\circ p(\bar
g_2)(s\circ p({\bar g_1}{\bar g_2}))^{-1}\nonumber\\
&=& {\bar g_1}\Phi(\bar g_1)^{-1}{
\bar g_2}\Phi(\bar g_2)^{-1}({\bar g_1}{\bar g_2}\Phi({\bar g_1}{\bar
g_2})^{-1})^{-1}\nonumber\\
&=& \Phi({\bar g_1}{\bar g_2})-\Phi(\bar g_1)-\Phi(\bar g_2). \nonumber
\end{eqnarray}

2) Puisque pour tout $t\in\mathbb R$, $p(t)=1$,  la relation (9)
implique $\Phi(t)=t$.

3) Par construction on a ${\bar g}t=s\circ p({\bar g}t)\Phi({\bar g}t), \forall t\in \mathbb R$ et ${\bar g}\in {\overline G}$. Ainsi, comme $p({\bar g}t)=p({\bar
g})$ on obtient $\Phi({\bar g}t)=\Phi({\bar g})+t$.

4) Puisque pour $g\in G$ on a $s(g)=s\circ p(s(g))\Phi\circ s(g)$ et  $p\circ s(g)=g$, donc $\Phi\circ s(g)=0$. \end{proof}

Suite aux discussions pr\'ec\'edentes on conclut qu'il existe une correspondance bijective entre les classes de cohomologie $x\in H^2(G, \mathbb R)$ et les extensions cetrales du groupe $G$ dont le noyau est isomorphe avec  $\mathbb R$ (cf. \cite{Bro} et \cite{Ma}).

\subsubsection{G\'en\'eralit\'es sur les quasi-morphismes}

\begin{definition}
Soit $G$ un groupe discret, on dira que l'application $\Phi : G\rightarrow
\mathbb R$ est un quasi-morphisme s'il existe un r\'eel $k>0$ tel que pour tous
$g$ et $h$ \'el\'ements de $G$ on a,
$$\mid \Phi(gh) - \Phi(g) - \Phi(h)\mid<k.$$

Si en plus, pour tout entier $n\in\mathbb Z$ on a $\Phi(g^n)=n\Phi(g)$ on dira
que $\Phi$ est un quasi-morphisme homog\`ene.
\end{definition}

Notons que si on applique l'affirmation 1,  on voit que pour tout 2-cocycle born\'e  $c : G\times G \rightarrow
\mathbb R$ la 1-cocha\^ \i ne $\Phi :
{\overline G} \rightarrow \mathbb R$ qui lui est associ\'ee par l'expression
(9) est un quasi-morphisme. Ce quasi-morphisme n'est pas n\'ecessairement homog\`ene,
mais on peut le rendre homog\`ene en posant, \begin{eqnarray} \varphi(\bar
g)=\displaystyle{\lim_{n \rightarrow +\infty}}\displaystyle{1\over n}\Phi({\bar
g}^n).\end{eqnarray}

\begin{affirmation}
 Le quasi-morphisme homog\`ene $\varphi : {\overline G} \rightarrow \mathbb R$ d\'efini par (10)
 v\'erifie les propri\'et\'es suivantes~:
\begin{enumerate}
\item $\forall t\in \mathbb R, \forall {\bar g}\in {\overline G},
    \varphi({\bar g}t)=\varphi(\bar g)+t$, en particulier $\varphi(t)=t$
    pour tout $t\in \mathbb R$.
\item La cocha\^\i ne r\'eelle $b'=\varphi-\Phi$ est born\'ee et $\mathbb
    R$-invariante (i.e. $b'(t\overline{g}) = b'(\overline{g})$).
\end{enumerate}
\end{affirmation}
\begin{proof}[D\'emonstration] 1) Puisque $\mathbb R$ est contenu dans le centre du groupe $\overline
G$, pour tout $n\in \mathbb N$ et pour tout \'el\'ement ${\bar g}\in {\overline
G}$ on a $({\bar g}t)^n={\bar g}^nt^n$. Par application de $\Phi$, puis par
passage \`a la limite on obtient 1).

\noindent 2) Puisque $\Phi$ est un quasi-morphisme il existe  un r\'eel $k>0$ tel que
pour tous les \'el\'ements ${\bar g}$ et ${\bar h}$ de $\overline G$,
$\mid\Phi({\bar g}{\bar h})-\Phi(\bar g)-\Phi(\bar h)\mid \leq k$. Donc,  pour
tout ${\bar g}\in {\overline G}$ et pour tout $n\in \mathbb N$, on a $\mid
\Phi({\bar g}^n)-n\Phi(\bar g)\mid\leq (n-1)k$. Ainsi, par passage \`a
la limite sur $n$ on d\'eduit que $b'=\varphi-\Phi$ est une cocha\^ \i ne born\'ee, qui est
$\mathbb R$-invariante puisque $\Phi$ et $\varphi$ v\'erifient
 la propri\'et\'e 1).  \end{proof}

 La propri\'et\'e 2) de l'affirmation 2 montre que la cocha\^ \i ne born\'ee
$b'=\varphi-\Phi$ induit sur le groupe $G$ une cocha\^ \i ne born\'ee $b : G
\rightarrow \mathbb R$
 telle que $b\circ p=b'$. Ainsi, en posant $$s_0(g)=s(g)b(g)^{-1}, \qquad  \forall g\in G$$
 on obtient une section ensembliste \`a la projection $p$ telle que $s_0(1)=1$.
  De plus, puisque pour tout  ${\bar g}\in {\overline G}$ l'\'el\'ement  $ s_0\circ p(\bar
  g)$ se projetent sur  $p(\bar g)\in G$,  il existe
  donc un  \'el\'ement central $h(\bar g)\in \mathbb R$ tel que~:
 \begin{eqnarray}
   {\bar g}=h(\bar g)s_0\circ p(\bar g).
   \end{eqnarray}

   \begin{affirmation} Avec les notations ci-dessus on a les affirmations
   suivantes~:
   \begin{enumerate}
   \item La cocha\^ \i ne $h : {\overline G}\rightarrow \mathbb R$
       d\'efinie par  (11) est \'egale au quasi-morphisme homog\`ene
       $\varphi : {\overline G}\rightarrow \mathbb R$ qui est associ\'e \`a $\Phi$ par l'expression
       (10).

  \item Le 2-cocycle born\'e $\bar c$ induit par le cobord $\mathbb
      R$-invariant $-d\varphi$ est cohomologue au 2-cocycle born\'e $c$.
      Plus pr\'ecis\'ement, on a ${\bar c}-c=-db$, o\`u $b$ est une cocha\^ \i ne born\'ee.
    \end{enumerate}
    \end{affirmation}
\begin{proof}[D\'emonstration] 1) Si \`a l'\'el\'ement  ${\bar g}\in {\overline G}$ on applique
 les formules (9) et (11) simultan\'ement, on obtient l'expression
$${\bar g}=s\circ p(\bar g)\phi(\bar g)=s_0\circ p(\bar g)h(\bar g)=s
\circ p(\bar g)b(p(\bar g))^{-1}h(\bar g)$$
de laquelle on  d\'eduit que $h(\bar g)-\Phi(\bar g)=b(p(\bar g))=b'(\bar g)$ (cf. aff. 2). C'est-\`a-dire on a $h=\varphi$.

\noindent 2) Un calcul direct sur le d\'efaut des sections $s$ et $s_0$ montre
que ${\bar c}-c=-db$. \end{proof}

Suite aux assertions de l'affirmation 3 on peut maintenant   interpr\'eter une classe de
cohomologie born\'ee r\'eelle $x\in H_b^2(G, \mathbb R)$ par la donn\'ee d'une extension
centrale
$$\xymatrix{ 0\ar@{->}[r]  & \mathbb R\ar@{->}[r]^{j} & \overline{G}\ar@<-3pt>@{->}[r]_{p}
 \ar@<3pt>@{<-}[r]^{s_x}\ar@<3pt>@{->}[d]_{\varphi} & G\ar@{->}[r]& 1\\
 && \mathbb R
}$$ munie d'une section ensembliste $s_x : G
\rightarrow {\overline G}$ de la projection $p$ telle que le quasi-morphisme $\varphi : {\overline G} \rightarrow \mathbb R$ qui lui est  associ\'e par la formule (9) est homog\`ene et avec $\varphi(s_x(g))=0, \forall g\in G$.

L'affirmation suivante se d\'emontre en utilisant le fait que  $\varphi$ est un quasi-morphisme homog\`ene associ\'e \`a la section $s_x$~:
\begin{affirmation} Avec les notations ci-dessus on a les assertions suivantes~:
\begin{enumerate}
\item $\forall g\in G, \forall n\in\mathbb Z, s_x(g^n)=(s_x(g))^n$.
\item La section $s_x : G\rightarrow \overline{G}$ commute avec la
    conjugaison i.e.~:   $$s_x(ghg^{-1})=s_x(g)s_x(h)s_x(g)^{-1}, \quad \forall
    g, h\in G.$$
\item Si on d\'esigne par $Z(G)$ le centre du groupe $G$ alors pour tous
    les \'el\'ements $g\in G$ et $z\in Z(G)$ on a $s_x(z)\in
    Z(\overline{G})$ et  $s_x(gz) = s_x(g)s_x(z)$.
\end{enumerate}

En cons\'equence,  la  restriction de $s_x$ \`a $Z(G)$ (\`a valeurs dans  $Z(\overline{G})$) (reps. $\varphi$ \`a $Z(\overline{G})$) est un homomorphismes tels que pour tout  $\overline{z}\in Z(\overline G)$ on a~:
        $$\overline{z}=s_x\circ p(\overline z) + j\circ \varphi(\overline
        z) \qquad \textrm{ et } \qquad
      \varphi(\overline{z}\ \overline{g}) = \varphi(\overline{z}) +
      \varphi(\overline{g})$$
Autrement dit, en tant que $\mathbb Z$-modules on a $Z(\overline G) = \mathbb R\bigoplus Z(G)$.
 \end{affirmation}

La section $s_x$ \'etudi\'ee ci-dessus n'est pas unique dans l'ensemble des sections de la
 projection $p$ qui induisent par (9) un quasimorphisme homog\`ene. En effet, si on multiplie $s_x$ par un homomorphisme $m : G \rightarrow \mathbb R$ on obtient
 une nouvelle section $s_1=s_x.m$ de $p$ dont le quasi-morphisme homog\`ene associ\'e par (9) est \'egal \`a
$\varphi_1=\varphi-m\circ p$. Cependant, si
on note $c_x$ le 2-cocycle born\'e obtenu par la formule (8) \`a partir de la section $s_x$ on voit  alors
que le 2-cocycle born\'e associ\'e \`a la section  $s_1$ par la formule (8) reste \'egal \`a $c_x$.

La proposition suivante r\'esume les propri\'etes
du 2-cocycle born\'e $c_x$  qui sont en effet essentielles
pour le reste de ce travail.

 \begin{proposition}[cf. \cite{Bou1}]
 Pour toute classe de cohomologie born\'ee
 $x\in H_b^2(G, \mathbb R)$ il existe un 2-cocycle born\'e $c_x$ unique dans la
 classe de cohomologie de $x$ tel que~:
\begin{enumerate}
 \item $c_x$ est le seul 2-cocycle born\'e repr\'esentant $x$ qui v\'erifie la
     relation :
 \begin{eqnarray}
  c_x(g^n, g^m)=0, \qquad \forall m, n\in \mathbb N, \forall g\in G.
 \end{eqnarray}

\noindent De plus le quasi-morphisme qui lui est  associ\'e
 par (9) est homog\`ene. On appellera $c_x$ le 2-cocycle homog\`ene
 repr\'esentant la classe $x$.

\item La classe de cohomologie  $x\in H_b^2(G, \mathbb R)$  est
    nulle si et seulement si le 2-cocycle born\'e homog\`ene $c_x$ est nul.
 \item Pour tout automorphisme $\alpha$ de $G$ on a $\alpha^*(c_x)=c_{
     \alpha_{b(x)}}$, o\`u $\alpha_b : H_b^2(G, \mathbb R) \rightarrow
     H_b^2(G, \mathbb R)$ est l'isom\'etrie induite par $\alpha$.
     En  cons\'equence,  la classe $x$ est fix\'ee par l'isom\'etrie $\alpha_b$ si
     et seulement si le 2-cocycle homog\`ene $c_x$ est invariant par
     $\alpha$ (i.e. $\alpha^*(c_x)=c_x$). En particulier, le 2-cocycle homog\`ene $c_x$ est invariant par
     conjugaison.
 \end{enumerate}
 \end{proposition}

 \begin{proof}[D\'emonstartion]  1) Soient $c\in x$ un $2$-cocycle qui v\'erifie (12) et  $b : G\rightarrow \mathbb R$ une cocha\^\i ne born\'ee telle que $c-c_x = db$. Puisque $c$ et $c_x$ v\'erifient la condition  (12) il s'ensuit que pour tout $g\in G$ on a l'\'egalit\'e $db(g, g)=0$ qui implique que $b(g^{2^n})=2^nb(g), \forall n\in \mathbb N$. Donc, $b=0$ et par suite $c=c_x$.

 \noindent 2) Supposons que la classe $x$ est nulle. Donc, il existe une cocha\^\i ne born\'ee r\'eelle $b$ telle que $c_x=db$. Ainsi, comme dans la preuve de 1)  la condition (12) implique que $c_x=0$. La r\'eciproque est \'evidente.

 \noindent 3) Puisque le 2-cocycle  $c_x$ est homog\`ene il en r\'esulte que pour tout automorphisme $\alpha$ du groupe $G$ le $2$-cocycle $\alpha^*(c_x)\in \alpha_b(x)$ v\'erifie (12). Donc, d'apr\`es  1) $c_{\alpha_b(x)}=\alpha^*(c_x)$.
 \end{proof}

\subsection{Construction de la classe de cohomologie born\'ee de degr\'e deux $\mathbf{g}_{_{2}}$}

Soit $G$ un groupe discret que l'on fait agir trivialement sur le second espace
de l'homologie $\ell^1$-r\'eduite $\overline{H}_2^{\ell_1}(G, \mathbb R)$. Soit
$\mathbf{m}_{_{2}} : G^2\rightarrow C_2^{\ell_1}(G, \mathbb R)$ la cocha\^\i ne
d\'efinie par l'expression :
\begin{eqnarray}
\mathbf{m}_{_{2}}(g, h) = (g, h) - \mathbf{m}(g) + \mathbf{m}(gh) - \mathbf{m}(h)
= (g, h) - \mathbf{m}\circ\partial_2(g, h)
\end{eqnarray}
o\`u $\mathbf{m} : C_1^{\ell_1}(G, \mathbb R)\rightarrow C_2^{\ell_1}(G,
\mathbb R)$ d\'esigne l'op\'erateur born\'e d\'efini ci-dessus par la formule
$(5)$ (cf.  \cite{Mi}). La cocha\^\i ne $\mathbf{m}_{_{2}}$ est born\'ee parce
que d'apr\`es l'expression $(5)$  pour tout $g\in G$ la norme
$\|\mathbf{m}(g)\|_1=1$ implique que pour tous $ g$ et  $h\in G$ la norme
$\|\mathbf{m}_{_{2}}(g, h)\|_1\leq 4$.

\begin{affirmation} Pour tous  $g$ et $h\in G$  on a  les propri\'et\'es suivantes :
\begin{enumerate}
\item Pour toute cocha\^\i ne born\'ee $b : G\rightarrow \mathbb R$, on a : $<db, \mathbf{m}_{_{2}}(g, h)>=0$.
\item Si $c : G^2\rightarrow \mathbb R$ est un $2$-cocycle born\'e et si
    $c_x: G^2\rightarrow \mathbb R$ d\'esigne l'unique $2$-cocycle born\'e
    homog\`ene tel que  $x=[c]=[c_x]\in H_b^2(G, \mathbb R)$ alors,
$$<c, \mathbf{m}_{_{2}}(g, h)>=<c_x, \mathbf{m}_{_{2}}(g, h)>=c_x(g, h).$$
\end{enumerate}
\end{affirmation}

\begin{proof}[D\'emonstration] Premi\`erement, remarquons que d'apr\`es l'expression $(5)$ qui d\'efinit la
co-cha\^\i ne $\mathbf m : G\rightarrow
C_2^{\ell_1}(G, \mathbb R)$ pour toute cocha\^\i ne born\'ee $b$  on
 a, $<db, \mathbf{m}(g)>=b(g)$,  ceci entra\^\i ne
$<db, \mathbf{m}_{_{2}}>=0$. Ainsi, puisqu'il existe une cocha\^\i ne born\'ee
$b' : G\rightarrow \mathbb R$ telle que  $c-c_x=db'$ on voit donc que $<c,
\mathbf{m}_{_{2}}>=<c_x, \mathbf{m}_{_{2}}>$.

Enfin, en utilisant le fait que le $2$-cocycle $c_x$ v\'erifie la condition
d'homog\'en\'eit\'e  on en d\'eduit que $<c_x, \mathbf{m}_{_{2}}(g, h)>=c_x(g,
h)$.
\end{proof}

\begin{affirmation} La cocha\^\i ne born\'ee $\mathbf{m}_{_{2}} : G^2\rightarrow C_2^{\ell_1}(G, \mathbb R)$
induit un $2$-cocycle born\'e homog\`ene $\overline{\mathbf{m}}_{_{2}} : G^2\rightarrow \overline{H}_2^{\ell_1}(G,
\mathbb R)$ (i.e. $\overline{\mathbf{m}}_{_{2}}(g^m, g^n)=0, \forall g\in G, \forall m, n\in\mathbb Z$).
\end{affirmation}

\begin{proof}[D\'emonstration] Observons d'abord que puisque pour tout $g\in G$ on a
$\partial_2(\mathbf{m}(g))=g$ il en r\'esulte que la cocha\^\i ne born\'ee
$\mathbf{m}_{_{2}} : G^2\rightarrow C_2^{\ell_1}(G, \mathbb R)$ prend ses
valeurs dans l'espace des $\ell_1$-cycles $Z_2^{\ell_1}(G, \mathbb R)$. Ainsi,
en composant la $2$-cocha\^\i ne $\mathbf{m}_{_{2}}$ avec la surjection
canonique
 $Z_2^{\ell_1}(G, \mathbb R)\rightarrow \overline{H}_2^{\ell_1}(G, \mathbb R)$, on obtient une $2$-cocha\^\i ne
 born\'ee $\overline{\mathbf{m}}_{_{2}} : G^2\rightarrow \overline{H}_2^{\ell_1}(G, \mathbb R)$.

Pour d\'emontrer que le cobord de la cocha\^\i ne $\overline{\mathbf{m}}_{_{2}}
: G^2\rightarrow \overline{H}_2^{\ell_1}(G, \mathbb R)$ est nul nous allons
d\'emontrer que le cobord de la cocha\^\i ne $\mathbf{m}_{_{2}} :
G^2\rightarrow Z_2^{\ell_1}(G, \mathbb R)$ est une $3$-cocha\^\i ne \`a valeurs
dans l'adh\'erence $\overline{B_2^{\ell_1}(G, \mathbb R)}\subseteq
Z_2^{\ell_1}(G, \mathbb R)$.

En effet, si on  d\'esigne par $D_*$ la diff\'erentielle ordinaire du complexe
des cocha\^\i nes born\'ees non homog\`enes  $C_b^*(G,  Z_2^{\ell_1}(G, \mathbb
R))$ o\`u le groupe $G$ agit trivialement sur l'espace $Z_2^{\ell_1}(G, \mathbb
R)$, alors,  \`a partir de la deuxi\`eme assertion de l'affirmation 5 et de la
lin\'earit\'e du crochet de dualit\'e $<, > : H_b^2(G, \mathbb R)\times
\overline{H}_2^{\ell_1}(G, \mathbb R) \rightarrow \mathbb R$ on voit que pour
toute classe de cohomologie  born\'ee $x=[c]=[c_x]\in H_b^2(G, \mathbb R)$ et
pour tous $g$, $h$ et  $k\in G$ on a,
\begin{eqnarray}
<c, D_*(\mathbf{m}_{_{2}})(g, h, k)>&=&<c, \mathbf{m}_{_{2}}(h, k)> - <c, \mathbf{m}_{_{2}}(gh, k)> \nonumber \\
&& + <c, \mathbf{m}_{_{2}}(g, hk)> - <c, \mathbf{m}_{_{2}}(g, h)>\nonumber \\
&=& c_x(h, k) - c_x(gh, k) + c_x(g, hk) - c_x(g, h)\nonumber \\
&=& dc_x(g, h, k)=0.\nonumber \end{eqnarray}

Ainsi, de la derni\`ere expression on d\'eduit  que le cobord
$D_*(\mathbf{m}_{_{2}}) : G^3\rightarrow Z_2^{\ell_1}(G, \mathbb R)$ prend ses
valeurs dans l'adh\'erence $\overline{B_2^{\ell_1}(G, \mathbb R)}\subseteq
Z_2^{\ell_1}(G, \mathbb R)$. Parce que si on suppose qu'il existe des
\'el\'ements  $g_0, h_0$ et $k_0\in G$ tels que,
$D(\mathbf{m}_{_{2}})(g_0,h_0,k_0)\not\in\overline{B_2^{\ell_1}(G, \mathbb
R)}\subset C_2^{\ell_1}(G, \mathbb R)$, le th\'eor\`eme de s\'eparation de
Hahn-Banach (cf. \cite{Hb}, \cite{Yo}) nous permet de trouver une forme
lin\'eaire continue non nulle $c_0 : C_2^{\ell_1}(G, \mathbb R)\rightarrow
\mathbb R$ qui s'annule sur le sous-espace ferm\'e $\overline{B_2^{\ell_1}(G,
\mathbb R)}$ et telle que $<c_0, D_*(\mathbf{m}_{_{2}})(g_0,h_0,k_0)>=1$.

Maintenant, puisque  la forme lin\'eaire continue $c_0 : C_2^{\ell_1}(G,
\mathbb R)\rightarrow \mathbb R$ est nulle sur l'espace des $2$-bords born\'es
$B_2^{\ell_1}(G, \mathbb R)\subseteq \overline{B_2^{\ell_1}(G, \mathbb R)}$ on
en d\'eduit que  $c_0 : C_2^{\ell_1}(G, \mathbb R)\rightarrow \mathbb R$
d\'efinit un $2$-cocycle born\'e $c_0\in Z_b^2(G, \mathbb R)$ et donc,   pour
tout triplet   $(g, h, k)\in G^3$ on aura l'\'egalit\'e  $$<c_0,
D_*(\mathbf{m}_{_{2}})(g,h,k)>=  0$$ qui contredit le fait que $<c_0,
D_*(\mathbf{m}_{_{2}})(g_0, h_0, k_0)>=1$.

Par cons\'equent, pour tous $g$, $h$ et $k\in G$ le cobord
$D_*(\mathbf{m}_{_{2}})(g,h,k)\in \overline{B_2^{\ell_1}(G, \mathbb R)}$.

Finalement, en passant en homologie $\ell_1$-r\'eduite on conclut que
$\overline{\mathbf{m}}_{_{2}} : G^2\rightarrow \overline{H}_2^{\ell_1}(G,
\mathbb R)$ est un $2$-cocycle born\'e. De plus, il est homog\`ene parce que
pour tous $g\in G$ et $n, m\in \mathbb Z$ la relation  $<c_x,
\mathbf{m}_{_{2}}(g^n, g^m)>=c_x(g^n, g^m)=0$ implique
$\overline{\mathbf{m}}_{_{2}}(g^n, g^m)=0$.\end{proof}

Nous avons maintenant tous les \'el\'ements n\'ecessaires  pour \'enoncer et
d\'emontrer le th\'eor\`eme principal A.

\newtheorem*{PrA}{Th\'eor\`eme principal A}\begin{PrA}
La classe de cohomologie born\'ee $\mathbf{g}_{_{2}}:=
[\overline{\mathbf{m}}_{_{2}}]\in H_b^2(G, \overline{H}_2^{\ell_1}(G, \mathbb
R))$ v\'erifie les deux propri\'et\'es suivantes :

\begin{enumerate}
\item $\mathbf{g}_{_{2}}$ est nulle si et seulement si le second groupe de
    cohomologie born\'ee $H_b^2(G, \mathbb R)=0$.
\item $\mathbf{g}_{_{2}}$ est la seule classe de cohomologie born\'ee
    \'el\'ement de l'espace $H_b^2(G, \overline{H}_2^{\ell_1}(G, \mathbb
    R))$ qui v\'erifie la relation,
$$x\cup \mathbf{g}_{_{2}} = x, \qquad  \forall x\in H_b^2(G, \mathbb R) $$
o\`u le cup-produit $\cup$ est d\'efini par l'entrelacement naturel
(dualit\'e) entre les espaces de Banach $H_b^2(G, \mathbb R)$ et
$\overline{H}_2^{\ell_1}(G, \mathbb R)$.
\end{enumerate}
\end{PrA}

\begin{proof}[D\'emonstration] 1) La premi\`ere proposition du th\'eor\`eme est une cons\'equence de la
formule $<c_x, \mathbf{m}_{_{2}}(g, h)>=c_x(g, h)$ (cf. aff. 5) et du
fait que la classe de cohomologie born\'ee r\'eelle $[c_x]$ est nulle  si et
seulement si le $2$-cocycle born\'e homog\`ene $c_x=0$ (cf. pr. 5).

\noindent 2) Supposons qu'il existe une cocha\^\i ne born\'ee ${\mu} :
G^2\rightarrow{Z}_2^{\ell_1}(G, \mathbb R)$ qui induit une classe de
cohomologie $[\overline{\mu}]\in H_b^2(G, \overline{H}_2^{\ell_1}(G, \mathbb
R))$ telle que, $x\cup [\overline{\mu}] = x, \forall x\in H_b^2(G, \mathbb R).$

Observons que sous cette hypoth\`ese   pour toute classe de cohomologie  $x\in
H_b^2(G, \mathbb R)$ on peut \'ecrire $<c_x, \mu - \mathbf{m}_{_{2}}>=0$.
Ainsi, en appliquant le th\'eor\`eme de s\'eparation de Hahn-Banach comme dans
la preuve de l'affirmation 6, on d\'eduit que la cocha\^\i ne born\'ee $\mu -
\mathbf{m}_{_{2}} : G^2 \rightarrow Z_2^{\ell_1}(G, \mathbb R)$ prend ses
valeurs dans l'adh\'erence $\overline{B_2^{\ell_1}(G, \mathbb R)}$. Autrement
dit on a, $\overline{\mu}=\overline{\mathbf{m}}_{_{2}} : G^2\rightarrow
\overline{H}_2^{\ell_1}(G, \mathbb R).$  \end{proof}

Dans le papier \cite{Bou4}, nous reviendrons sur la construction de  la classe de cohomologie $\mathbf{g}_2$ pour montrer qu'ellle est universelle au sens de Yoneda et nous en expliciterons l'expression sur quelques groupes discrets particuliers.

Dans la section 5, si $\theta : \Pi\rightarrow \textrm{Out}(G)$ d\'esigne la repr\'esentation ext\'erieure associ\'ee \`a une extension de groupes discrets $1\rightarrow G \stackrel{i}{\rightarrow} \Gamma
\stackrel{\sigma}{\rightarrow}\Pi \rightarrow 1$ (cf. voir 4.1) nous allons d\'emontrer que la classe de cohomologie $\mathbf{g}_2$ est invariante par rapport \`a l'action du groupe $\Pi$  sur l'espace $H_b^2(G, \overline{H}_2^{\ell_1}(G, \mathbb R))$ qui  est induite par la repr\'esentation $\theta$ (cf. cor. 6).

\section{Construction de la  classe de cohomologie born\'ee de degr\'e trois $[\theta]$}

Dans cette section, nous nous proposons d'associer \`a toute repr\'esentation
ext\'erieure $\theta : \Pi\rightarrow Out(G)$ une classe de cohomologie
born\'ee de degr\'e trois du groupe $\Pi$  \`a valeurs dans l'espace des
classes d'homologie $\ell_1$-r\'eduite $\overline{H}_2^{\ell_1}(G, \mathbb R)$
que l'on suppose muni de la structure de $\Pi$-module de Banach induite par
l'homomorphisme   $\theta$. Cette classe de cohomologie sera not\'ee
$[\theta]\in H_b^3(\Pi, \overline{H}_2^{\ell_1}(G, \mathbb R))$.

Pour construire la classe de cohomologie born\'ee $[\theta]\in
H_b^3(\Pi,\overline{H}_2^{\ell_1}(G, \mathbb R))$ nous allons nous {inspirer}
des techniques introduites en th\'eorie de l'obstruction qui associe \`a la
donn\'ee d'une repr\'esentation ext\'erieure $\theta : \Pi\rightarrow Out(G)$
une classe de cohomologie ordinaire de degr\'e trois \'el\'ement du groupe
$H^3(\Pi, Z(G))$ (cf. \cite{Ma} et \cite{Bro}).

\subsection{Classes de cohomologie  d'obstruction}  Dans ce paragraphe, nous donnerons l'essentiel des id\'ees de la th\'eorie de l'obstruction qui permettent d'associer \`a un homomorphisme $\theta : \Pi\rightarrow \textrm{Out}(G)$ une classe de cohomologie ordinaire de degr\'e trois \'el\'ement du groupe
$H^3(\Pi, Z(G))$ (cf. \cite{Ma} et \cite{Bro}).

\subsubsection{L'approche dir\`ecte} Soient $1\rightarrow G \stackrel{i}{\rightarrow} \Gamma
\stackrel{\sigma}{\rightarrow}\Pi \rightarrow 1 $ une extension de groupes
discrets et $ s : \Pi \rightarrow G$ une section ensembliste de la projection
$\sigma$ telle que $\sigma\circ s=id_{\Pi}$ et $s(1)=1$. Pour tous
$\alpha\in\Pi$ et $g\in G$ on pose $\Psi_s(\alpha)(g)=
s(\alpha)i(g)s(\alpha)^{-1}$.

 En g\'en\'eral, l'application $\Psi_s : \Pi\rightarrow \mathrm{Aut}(G)$ n'est pas
un homomorphisme. En effet, si on note $f_s : \Pi\times \Pi \rightarrow G$ le
d\'efaut de la section $s$ \`a \^etre un morphisme de groupes,
\begin{eqnarray*}
s(\alpha)s(\beta)= i(f_s(\alpha, \beta))s(\alpha\beta), \quad \forall \alpha, \beta\in \Pi
\end{eqnarray*}
on v\'erifie que pour tout $g\in G$ :
$$\Psi_s(\alpha)\circ\Psi_s(\beta)(g)= i_{f_s(\alpha, \beta)}\circ\Psi_s(\alpha\beta)(g) $$
o\`u $i_x$  d\'esigne l'automorphisme int\'erieur de $G$ d\'efini par $i_x(g) =xgx^{-1}, \forall x, g\in G$.

Notons aussi que si  $s_1$ est une deuxi\`eme section ensembliste de $\sigma$
alors il existe une application $ h : \Pi \rightarrow G$ telle que
$s_1(\alpha)=h(\alpha)s(\alpha), \forall \alpha\in \Pi$. Ainsi, par
d\'efinition de  l'application $\Psi_{s_1}$, on aura la relation
$\Psi_{s_1}(\alpha)(g)=h(\alpha)\Psi_s(\alpha)(g)h(\alpha)^{-1}$ qui montre que
$\Psi_s$ induit un homomorphisme $\theta : \Pi \rightarrow Out(G)$ ne
d\'ependant que de la suite exacte donn\'ee.

La cocha\^\i ne non ab\'elienne $f_s$ v\'erifie la condition de
2-cocycle non ab\'elien sur le groupe $\Pi$ \`a valeurs dans le groupe
$G$ donn\'ee par :
 \begin{eqnarray}
 \Psi_s(\alpha)(f_s(\beta, \gamma))f_s(\alpha, \beta\gamma)=f_s(\alpha,
 \beta)f_s(\alpha\beta, \gamma)
 \end{eqnarray}

En effet, c'est gr\^ace \`a l'expression (14) qu'on pourra identifier le groupe
$\Gamma$ avec le produit   cart\'esien tordu $\Pi \times_{f_s}G$ munit de la
loi de composition interne suivante (cf.  \cite{Bro} et \cite{Ma})~:
 \begin{eqnarray*}
 (\alpha, g)\cdot(\alpha_1, g_1)=(\alpha\alpha_1,g\Psi_s(
   \alpha)(g_1)f(\alpha,\alpha_1))
   \end{eqnarray*}

On v\'erifie aussi que le couple $(1,1)$ est \'el\'ement neutre dans
$\Pi\times_{f_s}G$ et que $i(g)=(1,g)$ et $\sigma(\alpha, g)=\alpha$ sont
respectivement l'injection canonique de $G$ dans $\Pi\times_{f_s}G$ et la
projection canonique de $\Pi\times_{f_s}G$ sur $\Pi$.

\subsubsection{L'approche inverse} Consid\'erons  une repr\'esentation $\theta : \Pi \rightarrow
\mathrm{Out}(G)$ et un rel\`evement ensembliste, $\Psi : \Pi \rightarrow
\mathrm{Aut}(G)$,  de $\theta$. Puisque pour tous les \'el\'ements $\alpha$ et $\beta$ de $\Pi$ les deux
automorphismes $\Psi(\alpha)\circ\Psi(\beta)$ et $\Psi(\alpha\beta)$
repr\'esentent le m\^eme automorphisme ext\'erieur
$\theta(\alpha)\circ\theta(\beta)=\theta(\alpha\beta)$ il existe donc  un
\'element $f(\alpha, \beta)\in G$ tel que~:
\begin{eqnarray*}
i_{f(\alpha, \beta)}\circ\Psi(\alpha\beta)=\Psi(
\alpha)\circ\Psi(\beta)
\end{eqnarray*}

Dans la suite nous appelerons  le couple $(\Psi, f)$  {\it noyau abstrait} de
l'homomorphisme $\theta : \Pi\rightarrow Out(G)$ (cf. \cite{Bro} et \cite{Ma}).

 En g\'en\'eral, la cocha\^\i ne $f : \Pi\times\Pi \rightarrow G$ ne v\'erifie pas
la relation (14)  des 2-cocycles non ab\'eliens. Cependant, si pour tous les
\'el\'ements $\alpha$, $\beta$ et $\gamma\in \Pi$ on pose :
\begin{eqnarray}
K_{_{\Psi, f}}(\alpha, \beta, \gamma)
=\Psi(\alpha)(f(\beta, \gamma))f(\alpha,\beta\gamma)f(\alpha\beta,
 \gamma)^{-1}f(\alpha, \beta)^{-1}
 \end{eqnarray}
on v\'erifie  que l'application $K_{_{\Psi,f}} : \Pi^3\rightarrow G$ prend ses
valeurs dans le centre $Z(G)$ du groupe $G$. En outre, en munissant le groupe ab\'elien  $Z(G)$ de
la structure de $\Pi$-module induite par l'homomorphisme $\theta :
\Pi\rightarrow Out(G)$, alors gr\^ace \`a l'associativit\'e du groupe des automorphismes $Aut(G)$
on d\'emontre que la $3$-cocha\^\i ne $K_{_{\Psi,f}} : \Pi^3\rightarrow Z(G)$ est un cocycle et
que la classe de cohomologie $[K_{_{\Psi, f}}]\in H^3(\Pi, Z(G))$ ne d\'epend
que de la repr\'esentation ext\'erieure
 $\theta : \Pi\rightarrow Out(G)$. De plus, on d\'emontre que la classe    $[K_{_{\Psi, f}}]\in H^3(\Pi, Z(G))$ s'annule  si et seulement si l'homomorphisme
 $\theta : \Pi\rightarrow Out(G)$ est associ\'e \`a une
 extension de groupes discrets $1 \longrightarrow G\stackrel{i}{\longrightarrow}
 \Gamma\stackrel{\sigma}{\longrightarrow} \Pi\longrightarrow 1$ (cf. \cite{Bro} et \cite{Ma}).

Dans le reste de cette section, nous allons munir l'espace d'homologie $\ell^1$-r\'eduite $\overline{H}_2^{\ell_1}(G, \mathbb R)$ de la structure de $\Pi$-module induite par l'homomorphisme $\theta : \Pi\rightarrow \textrm{Out}(G)$ et ainsi en remarquant  que si on r\'epartit les quatre facteurs le l'expression (15) comme suit (voir l'expression (23) du paragraphe 4.4)
$$\overline{\theta}_{_{\Phi, f}}(\alpha, \beta, \gamma):= \overline{\mathbf{m}}_2(\Psi(\alpha)(f(\beta, \gamma)), f(\alpha,\beta\gamma)) -\overline{\mathbf{m}}_2(f(\alpha, \beta),  f(\alpha\beta,
 \gamma))      $$
 nous allons v\'erifier que la $3$-cocha\^\i ne born\'ee $\overline{\theta}_{_{\Phi, f}} : \Pi^3 \longrightarrow \overline{H}_2^{\ell_1}(G, \mathbb R)$ est un cocycle (cf. pr. 9) repr\'esantant une  classe de cohomologie not\'ee, $[\theta]\in H_b^3(\Pi,  \overline{H}_2^{\ell_1}(G, \mathbb R))$,  qui ne d\'epend pas du noyau abstrait $(\Psi, f)$ associ\'e \`a l'homomorphisme  $\theta : \Pi\rightarrow \textrm{Out}(G)$ (cf. pr. 10).

L'utilit\'e de la classe de cohomologie  $[\theta]$ pour cet article appara\^\i t dans  la section 5.

Plus pr\'ecis\'ement,  en se donnant une extension de groupes discrets $1 \rightarrow G\stackrel{i}{\rightarrow}\Gamma\stackrel{\sigma}{\rightarrow} \Pi\rightarrow 1$ munie de sa repr\'esentation ext\'erieure $\theta : \Pi\rightarrow \textrm{Out}(G)$, nous allons v\'erifier que la diff\'erentielle $d_3^{0, 2} : E_3^{0, 2}\rightarrow E_3^{3, 0}$, de la suite spectrale de Hochschilde-Serre en cohomologie born\'ee \`a coefficients dans le $\Pi$-module $ \overline{H}_2^{\ell_1}(G, \mathbb R)$,     envoie la classe de cohomologie born\'ee $\mathbf{g}_2$ sur la classe de cohomologie born\'ee $[\theta]$ (cf.  th. B). Ensuite, pour tout entier $n\geq 0$  nous allons d\'emontrer que la diff\'erentielle $d_3^{n, 2} : E_3^{n, 2}\rightarrow E_3^{n+3, 0}$, de la suite spectrale de Hochschilde-Serre en cohomologie born\'ee r\'eelle, est donn\'ee par le cup produit (cf. th. C),
$$d_3^{n, 2}(x)=(-1)^nx\cup [\theta]$$

\subsection{Quasi-actions}
Rappelons que d'apr\`es les discutions du paragraphe 3.4, \'etant donn\'ee une
classe de cohomologie born\'ee  $x\in H_b^2(G, \mathbb R)$ nous lui avons
associ\'e une extension centrale $0 \longrightarrow \mathbb
R\stackrel{j}{\longrightarrow}\overline{G}\stackrel{p}{\longrightarrow} G
\longrightarrow 1$ munie d'un quasi-morphisme homog\`ene $\varphi :
\overline{G}\rightarrow \mathbb R$ et d'une section ensembliste $s_x :
G\rightarrow \overline{G}$ dont le d\'efaut \`a \^etre un homomorphisme est un
$2$-cocycle born\'e homog\`ene $c_x : G^2\rightarrow \mathbb R$ tel que
$p^*(c_x)=-d\varphi$ et  $[c_x]=x\in H_b^2(G, \mathbb R)$.

Le lemme suivant ainsi que la proposition et son corollaire sont extraits de
\cite{Bou} et \cite{Bou1}.

\begin{lemma}
Soit $(\Psi, f)$ un noyau abstrait de l'homomorphisme $\theta : \Pi\rightarrow
\mathrm{Out}(G)$. Pour tous les \'el\'ements $\alpha\in\Pi$ et $\overline{g}\in
\overline{G}$ on pose :
\begin{eqnarray}
\overline{\Psi}(\alpha)(\overline{g}) =
s_x(\Psi(\alpha)(p(\overline{g})))\varphi(\overline{g})
\end{eqnarray}

Alors, pour chaque $\alpha\in \Pi$ l'application $\overline{\Psi}(\alpha) :
\overline{G}\rightarrow \overline{G}$ est une bijection \'egale \`a
l'identit\'e sur la droite centrale $\mathbb R=Ker(p)\subset\overline{G}$ et
qui v\'erifie les propri\'et\'es suivantes :
\begin{enumerate}
\item $\forall \overline{g}\in \overline{G}, \
    p(\overline{\Psi}(\alpha)(\overline{g}))=
    \Psi(\alpha)(p(\overline{g}))$ ;
\item $\forall \overline{g}\in \overline{G},
    \varphi(\overline{\Psi}(\alpha)(\overline{g})) = \varphi(\overline{g})$
    ;
\item $\forall \alpha\in \Pi, \forall \bar{g}\in\overline{G}$ ,
    $\overline{\Psi}(\alpha)\circ i_{_{\bar g}}=
    i_{_{\overline{\Psi}(\alpha)(\bar g)}}\circ \overline{\Psi}(\alpha)$ ;
\item $\forall \alpha, \beta\in \Pi, \ {\overline
    \Psi}(\alpha)\circ{\overline \Psi}(\beta) = i_{F_x(\alpha,
    \beta)}\circ{\overline \Psi}(\alpha\beta)$ o\`u $F_x(\alpha, \beta) =
    s_x\circ f(\alpha, \beta)$.
\item $\forall \alpha\in \Pi, \overline{g}\in \overline{G}, \overline{z}\in
    Z(\overline{G}),
    \overline{\Psi}(\alpha)(\overline{g}\ \overline{z})=\overline{\Psi}(\alpha)(\overline{g})
    \overline{\Psi}(\alpha)(\overline{z})$.
\end{enumerate}
\end{lemma}

\begin{proof}[D\'emonstration] La relation 1) est \'evidente. La relation 2)
r\'esulte du fait que $\varphi\circ s_x=0$. La relation 3) est une
cons\'equence imm\'ediate de la relation : $s_x\circ i_g =
i_{_{s_{_{x}}(g)}}\circ s_x$ tandis que  la relation 4) se d\'eduit de la
formule de composition, $\Psi(\alpha)\circ\Psi(\beta)=i_{_{f(\alpha,
\beta)}}\circ\Psi(\alpha\beta)$.

Enfin, la relation 5) se d\'emontre \`a l'aide des propri\'et\'es de $s_x$ et de $\varphi$
\'enonc\'ees dans l'affirmation 4.
\end{proof}

\begin{proposition}
Pour tous $\overline{g}$ et $\overline{h}$ \'el\'ements du groupe $\overline G$
la bijection ${\overline \Psi}(\alpha) : \overline{G}\rightarrow \overline{G}$
v\'erifie la relation suivante~:
\begin{eqnarray}
{\overline \Psi}(\alpha)(\overline{g})
{\overline \Psi}(\alpha)(\overline{h})
 ={\overline \Psi}(\alpha)(\overline{g}\overline{h})
(\Psi(\alpha))^*(c_x)(p(\overline{g}), p(\overline{h}))
 c_x(p(\overline{g}), p(\overline{h}))^{-1}
 \end{eqnarray}
qui mesure le d\'efaut de $\overline{\Psi}(\alpha)$ \`a \^etre un automorphisme
sur l'extension centrale $\overline G$.

En cons\'equence, l'application ${\overline \Psi}(\alpha) :
\overline{G}\rightarrow \overline{G}$ devient un automorphisme  si et seulement
si la classe de cohomologie born\'ee $x\in H_b^2(G, \mathbb R)$ est invariante
par l'action de l'automorphisme $\Psi(\alpha)\in \mathrm{Aut}(G)$.
\end{proposition}

\begin{proof}[D\'emonstration] En effet, avec les notations ci-dessus on peut \'ecrire :
\begin{eqnarray}
{\overline \Psi}(\alpha)(\overline{g}){\overline \Psi}(\alpha)
(\overline{h})&=&
[s_x(\Psi(\alpha)(p(\overline{g})))
s_x(\Psi(\alpha)(p(\overline{h})))]
[\varphi(\overline{g})\varphi(\overline{h})]\nonumber\\
&=&[s_x(\Psi(\alpha)(p(\overline{g}\overline{h})))
(\Psi(\alpha))^*(c_x)(p(\overline{g}), p(\overline{h}))]
[\varphi(\overline{g}\overline{h})
c_x(p(\overline{g}), p(\overline{h}))^{-1}]\nonumber\\
&=& {\overline \Psi}(\alpha)(\overline{g}\overline{h})
(\Psi(\alpha))^*(c_x)(p(\overline{g}), p(\overline{h}))
c_x(p(\overline{g}), p(\overline{h}))^{-1}\nonumber
\end{eqnarray}
\end{proof}

\begin{corollary} Soient $0 \longrightarrow \mathbb R\stackrel{j}{\longrightarrow}\overline{G} \stackrel{p}
{\longrightarrow} G \longrightarrow 1$ une extension centrale  induite par  une
classe de cohomologie born\'ee $x\in H_b^2(G, \mathbb R)$ ; et $1
\longrightarrow
G\stackrel{i}{\longrightarrow}\Gamma\stackrel{\sigma}{\longrightarrow}
\Pi\longrightarrow 1$
 une extension de groupes avec $\theta : \Pi\rightarrow \mathrm{Out}(G)$ d\'esigne  la repr\'esentation
 ext\'erieure qui lui est associ\'ee. Alors, il existe un homomorphisme $\overline{\theta} : \Pi\rightarrow
 \mathrm{Out}(\overline{G}, \mathbb R)=\{a\in \mathrm{Out}(\overline{G})/ a(t)=t, \forall t\in\mathbb R\}$ qui commute dans le
 diagramme suivant
$$\xymatrix{&\ar@{->}[d]^{\theta}\ar@{-->}[dl]_{\overline{\theta}} \Pi    \\
\mathrm{Out}(\overline{G}, \mathbb R)\ar@{->}[r]^{p^{\#}}&\mathrm{Out}(G)}$$ si et seulement si
la classe  $x\in H_b^2(G, \mathbb R)$ est $\Pi$-invariante. O\`u
$p^{\#} :  \mathrm{Out}(\overline{G}, \mathbb R)\rightarrow \mathrm{Out}(G)$ d\'esigne
l'homomorphisme naturel induit par la surjection $p : \overline{G}\rightarrow
G$.
 \end{corollary}

\subsection{Construction d'une cocha\^\i ne born\'ee de \lq\lq composition\rq\rq}
Dans ce paragraphe, on garde les notations introduites ci-dessus et on
d\'esigne par $\mathrm{Sym}(\overline{G}, \mathbb R)$ (resp. $\mathrm{Aut}(\overline G, \mathbb
R)$) le groupe des bijections (resp. des automorphismes) de l'extension
centrale $\overline{G}$ qui sont \'egales \`a l'identit\'e sur la droite
centrale $\mathbb R=Ker(p)\subset \overline{G}$.

Rappelons que la restriction de $s_x$ \`a $Z(G)$ (\`a valeurs dans $Z(\overline{G})$) est un
homomorphisme (cf. aff. 4). Il en r\'esulte  que  l'application
$\overline{\Psi} : \Pi\rightarrow \mathrm{Sym}(\overline{G}:\mathbb R)$ qui est
d\'efinie par la formule (16) (cf. lemme 4)  induit un homomorphisme $\overline{\Psi} :
\Pi\rightarrow \mathrm{Aut}(Z(\overline{G}), \mathbb R)$ parce que,  pour tous  $\alpha,
\beta\in\Pi$ et $z\in Z(\overline G)$ on a,
$$\overline{\Psi}(\alpha)\circ\overline{\Psi}(\beta)(z) = F_x(\alpha, \beta)
\overline{\Psi}(\alpha\beta)(z)F_x(\alpha, \beta)^{-1}=
\overline{\Psi}(\alpha\beta)(z)
$$

Ainsi, en cons\'equence de ces remarques, dans tout le reste de ce travail
nous allons regarder le centre $Z(\overline G)\subseteq \overline{G}$ comme un
$\Pi$-module  dont la structure est induite par l'homomorphisme
$\overline{\Psi} : \Pi\rightarrow \mathrm{Aut}(Z(\overline{G}),  \mathbb R)$.

Maintenant, consid\'erons  une repr\'esentation ext\'erieure  $\theta :
\Pi\rightarrow \mathrm{Out}(G)$ et fixons un de ses noyaux abstraits $(\Psi, f)$. Avec la
formule (16) on obtient  une application $\overline{\Psi} : \Pi\rightarrow
\mathrm{Sym}(\overline{G}, \mathbb R)$. D'autre part, en imitant la formule (15) qui
d\'efinit le $3$-cocycle d'obstruction $K_{_{\Psi, f}} : \Pi^3\rightarrow
Z(G)$, on pourra d\'efinir une application $K_{_{x, \overline{\Psi}}} :
\Pi^3\rightarrow \overline{G}$ en posant  pour tous $\alpha$, $\beta$ et
$\gamma\in \Pi$~:
\begin{eqnarray}
K_{_{x, \overline{\Psi}}}(\alpha, \beta, \gamma) =
{\overline \Psi}(\alpha)(F_x(\beta, \gamma))F_x(\alpha, \beta\gamma)
F_x(\alpha\beta, \gamma)^{-1}F_x(\alpha, \beta)^{-1}
\end{eqnarray}

En effet, puisque l'image $p(K_{_{x, \overline{\Psi}}}) = K_{_{\Psi, f}}$ prend
ses valeurs dans le centre $Z(G)$ on en d\'eduit que l'expression (18)
d\'efinit une $3$-cocha\^\i ne ab\'elienne  $K_{_{x, \overline{\Psi}}} :
\Pi^3\rightarrow Z(\overline{G})$.  D'autre part, puisque l'on sait que la
section $s_x : Z(G)\rightarrow Z(\overline G)$ et le quasi-morphisme $\varphi :
Z(\overline G)\rightarrow \mathbb R$ sont des homomorphismes reli\'es  par
l'identit\'e (cf. aff. 4),
$$\overline{z} = s_x\circ p(\overline{z}) + j\circ
\varphi(\overline{z}), \qquad  \forall \overline{z}\in Z(\overline G),$$ il en
r\'esulte que les deux  $3$-cocha\^\i nes ab\'eliennes $K_{_{x, \overline{\Psi}}}
: \Pi^3\rightarrow Z(\overline{G})$ et $K_{_{\Psi, f}} : \Pi^3\rightarrow Z(G)$
sont elles  aussi  reli\'ees par l'indentit\'e,
\begin{eqnarray}
K_{_{x, \overline{\Psi}}} = \varphi_*(K_{_{x, \overline{\Psi}}}) +
(s_x)_*(K_{_{\Psi, f}}) \in Z(\overline G)\simeq \mathbb R\oplus Z(G).
\end{eqnarray}

La $3$-cocha\^\i ne image  $\varphi_*(K_{_{x, \overline{\Psi}}}) : \Pi^3
\rightarrow \mathbb R$ sera appel\'ee    cocha\^\i ne de composition. Nous
l'avons nomm\'ee  ainsi  parce que si on suppose que la repr\'esentation
ext\'erieure $\theta : \Pi\rightarrow Out(G)$ est associ\'ee \`a une extension
de groupes $1 \longrightarrow
G\stackrel{i}{\longrightarrow}\Gamma\stackrel{\sigma}{\longrightarrow}
\Pi\longrightarrow 1$,
 en prenant une extension centrale $0 \longrightarrow \mathbb R\stackrel{j}{\longrightarrow}\overline
 G{\rfl{\hfl{s}} {p}}G \longrightarrow 1$ qui repr\'esente la classe de cohomologie born\'ee
  $x\in H_b^2(G, \mathbb R)$ ;
la $2$-extension $0{\hbox to 6mm{\rightarrowfill}} \mathbb R\stackrel{j}{\hbox
to 6mm {\rightarrowfill}}\overline{G}\stackrel{i\circ p}{\hbox to 6mm
{\rightarrowfill}}\Gamma \stackrel{\sigma}{\hbox to 6mm {\rightarrowfill}} \Pi
\stackrel{}{\hbox to 6mm {\rightarrowfill}}1$ permet alors de retrouver la
$3$-cocha\^\i ne $\varphi_*(K{_{x, \overline{\Psi}}}) : \Pi^3 \rightarrow
\mathbb R$ en proc\'edant comme si le groupe $\overline{G}$ \'etait un
$\Pi$-module tordu au sens de Whitehead (cf. \cite{Bro} page 102).

\begin{proposition}
La cocha\^\i ne de composition $c_{_{x, \overline{\Psi}}}:= \varphi_*( K_{_{x,
\overline{\Psi}}}) : \Pi^3\longrightarrow \mathbb R$ est donn\'ee explicitement
par la formule,
\begin{eqnarray}
\qquad c_{_{x, \overline{\Psi}}}(\alpha, \beta, \gamma) = c_x(\Psi(\alpha)(f(\beta,
\gamma), f(\alpha, \beta\gamma)) - c_x(f(\alpha, \beta), f(\alpha\beta,
\gamma)), \ \ \forall \alpha, \beta, \gamma\in\Pi
\end{eqnarray}

En cons\'equence,  $c_{_{x, \overline{\Psi}}}=\varphi_*( K_{_{x,
\overline{\Psi}}}) : \Pi^3\longrightarrow \mathbb R$ est une  $3$-cocha\^\i ne
born\'ee.
\end{proposition}

\begin{proof}[D\'emonstration] Puisque pour tous $\overline{z}\in
Z(\overline{G})$ et $\overline{g}\in \overline{G}$ le quasi-morphisme
homog\`ene $\varphi : \overline{G}\rightarrow \mathbb R$ v\'erifie la relation,
$\varphi(\overline{z}\ \overline{g})=\varphi(\overline{z}) +
\varphi(\overline{g})$ (cf. aff. 4), donc, si on l'applique sur
les deux membres de l'expression $$K_{_{x, \overline{\Psi}}}(\alpha, \beta,
\gamma)F_x(\alpha, \beta) F_x(\alpha\beta, \gamma) = {\overline
\Psi}(\alpha)(F_x(\beta, \gamma))F_x(\alpha, \beta\gamma) $$ d\'eduite de
(18) on obtient :
$$
\varphi({\overline \Psi}(\alpha)(F_x(\beta, \gamma))F_x(\alpha,
\beta\gamma)) =
\varphi(K_{_{x, \overline{\Psi}}}(\alpha, \beta, \gamma)) + \varphi(F_x(\alpha,
\beta)F_x(\alpha\beta, \gamma)). $$

Ensuite, d\'eveloppons  les deux membres ci-dessus   en appliquant la relation
$p^*(c_x)=-d\varphi$, le fait qu'on a $\varphi\circ s_x=0$ implique
$\varphi\circ F_x=0$ et que $\varphi({\overline \Psi}(\alpha)(\bar g))=
\varphi(\bar g), \forall \alpha\in \Pi$ et $\bar{g}\in \overline{G}$.
\begin{eqnarray}
\varphi({\overline \Psi}(\alpha)(F_x(\beta, \gamma))F_x(\alpha,
\beta\gamma))
&=& \varphi({\overline \Psi}(\alpha)(F_x(\beta, \gamma)))+\varphi(F_x(\alpha,
\beta\gamma)) \nonumber \\
&+& c_x({\Psi}(\alpha)(f(\beta, \gamma), f(\alpha,
\beta\gamma))\nonumber \\
&=& c_x(
{\Psi}(\alpha)(f(\beta, \gamma), f(\alpha,
\beta\gamma)) \nonumber
\end{eqnarray}
et
\begin{eqnarray}
\varphi(K_{_{x, \overline{\Psi}}}(\alpha, \beta, \gamma)F_x(\alpha,
\beta)F_x(\alpha\beta, \gamma))
&=& \varphi(K_{_{x, \overline{\Psi}}}(\alpha, \beta, \gamma)) +
\varphi(F_x(\alpha,
\beta)F_x(\alpha\beta, \gamma))  \nonumber \\
&=& \varphi(K_{_{x, \overline{\Psi}}}(\alpha, \beta, \gamma)) +
\varphi(F_x(\alpha,
\beta))+\varphi(F_x(\alpha\beta, \gamma)) \nonumber \\
&+& c_x(f(\alpha,
\beta), f(\alpha\beta, \gamma))
\nonumber\\
&=&
\varphi(K_{_{x, \overline{\Psi}}}(\alpha, \beta, \gamma)) + c_x(f(\alpha,
\beta), f(\alpha\beta, \gamma)).
\nonumber
\end{eqnarray}

Ainsi, en comparant les deux d\'eveloppements on d\'eduit l'expresion
(20).\end{proof}

\begin{proposition}
Le cobord de la $3$-cocha\^\i ne  $K_{_{x, \overline{\Psi}}} : \Pi^3\
\rightarrow Z(\overline{G})$ est donn\'e par l'expression,
\begin{eqnarray}
d(K_{_{x, \overline{\Psi}}})(\alpha, \beta, \gamma, \zeta) =
c_{_{x, \overline{\Psi}}}(\beta, \gamma, \zeta) -
c_{_{\Psi(\alpha)_b(x), \overline{\Psi}}}(\beta, \gamma, \zeta)\in\mathbb
R.
\end{eqnarray}
\end{proposition}

\begin{proof}[D\'emonstration] Pour tous $\alpha, \beta, \gamma$ et $\zeta$
\'el\'ements du groupe $\Pi$ nous poserons :
$$ L =
{\overline \Psi}(\alpha)\Big({\overline \Psi}(\beta)(F_x(\gamma, \zeta))
F_x(\beta, \gamma\zeta)\Big)F_x(\alpha, \beta\gamma\zeta)\in\overline{G}.$$

Ci-dessous, nous d\'eveloppons l'\'el\'ement $L\in\overline{G}$ de deux fa\c
cons :

\noindent 1) Dans le premier d\'eveloppement,   nous  utilisons l'expression
(18) qui d\'efinit la cocha\^\i ne $K_{_{x, \overline{\Psi}}}$. De plus, nous
allons utiliser l'expression (17) de la proposition 6 qui contr\^ole la
d\'eviation de la bijection $\overline{\Psi}(\alpha) :
\overline{G}\rightarrow\overline{G}$ \`a \^etre un automorphisme. De m\^eme,
nous allons appliquer la formule (4) du lemme 4 qui contr\^ole la d\'eviation de
l'application $\overline{\Psi} : \Pi\rightarrow
\mathcal{S}ym(\overline{G}:\mathbb R)$ \`a \^etre un homomorphisme.
\begin{eqnarray}
L &=&
{\overline \Psi}(\alpha)\Big({\overline \Psi}(\beta)(F_x(\gamma, \zeta))
F_x(\beta, \gamma\zeta)\Big)F_x(\alpha, \beta\gamma\zeta)\nonumber
\\
&=&{\overline \Psi}(\alpha)\circ{\overline \Psi}(\beta)(F_x(\gamma, \zeta))
[{\overline \Psi}(\alpha)(F_x(\beta, \gamma\zeta))
     F_x(\alpha, \beta\gamma\zeta)] \nonumber\\
   && c_x(\Psi(\beta)(f(\gamma, \zeta)), f(\beta, \gamma\zeta))
 (\Psi(\alpha))^*(c_{x}(\Psi(\beta)(f(\gamma, \zeta)), f(\beta, \gamma\zeta))^{-1}
  \nonumber\\
 &=& i_{F_x(\alpha, \beta)}\circ{\overline \Psi}(\alpha\beta)(F_x(\gamma, \zeta))
   [F_x(\alpha, \beta)F_x(\alpha\beta, \gamma\zeta)
    K_{_{x, \overline{\Psi}}}(\alpha, \beta, \gamma\zeta)]\nonumber\\
   && c_x(\Psi(\beta)(f(\gamma, \zeta)), f(\beta, \gamma\zeta))
  (\Psi(\alpha))^*(c_{x}(\Psi(\beta)(f(\gamma, \zeta)), f(\beta, \gamma\zeta))^{-1}
  \nonumber\\
 &=& F_x(\alpha, \beta)[{\overline \Psi}(\alpha\beta)(F_x(\gamma, \zeta))
      F_x(\alpha\beta, \gamma\zeta)]
      K_{_{x, \Psi}}(\alpha, \beta, \gamma\zeta)
      \nonumber\\
   && c_x(\Psi(\beta)(f(\gamma, \zeta)), f(\beta, \gamma\zeta))
(\Psi(\alpha)^*(c_{x})(\Psi(\beta)(f(\gamma, \zeta)),
f(\beta, \gamma\zeta))^{-1}
  \nonumber\\
 &=& [F_x(\alpha, \beta)F_x(\alpha\beta, \gamma)F_x(\alpha\beta\gamma, \zeta)]
K_{_{x, \overline{\Psi}}}(\alpha\beta, \gamma, \zeta)
K_{_{x, \overline{\Psi}}}(\alpha, \beta, \gamma\zeta)\nonumber\\
   && c_x(\Psi(\beta)(f(\gamma, \zeta)), f(\beta, \gamma\zeta))
(\Psi(\alpha))^*(c_{x})(\Psi(\beta)(f(\gamma, \zeta)), f(\beta, \gamma\zeta))^{-1}
  \nonumber
\end{eqnarray}

\noindent 2) Dans ce second d\'eveloppement de l'expression  $L$, \`a c\^ot\'e
des expressions (17) et (18) nous allons  utiliser le fait que les
bijections $\overline{\Psi}(\alpha)$ fixent les points de $\mathbb R=Ker(p)$ et
que pour tous $\overline{g}\in \overline{G}$ et $\overline{z}\in
Z(\overline{G})$ on a~: $\overline{\Psi}(\alpha)(\overline{g}\ \overline{z})=
\overline{\Psi}(\alpha)(\overline{g})\overline{\Psi}(\alpha)(\overline{z})$
(cf. lemme 4).

\begin{eqnarray}
L&=& {\overline \Psi}(\alpha)\Big({\overline \Psi}(\beta)(F_x(\gamma, \zeta))
F_x(\beta, \gamma\zeta)\Big)F_x(\alpha, \beta\gamma\zeta)\nonumber\\
&=& {\overline \Psi}(\alpha)\Big(K_{_{x, \overline{\Psi}}}(\beta, \gamma, \zeta)
F_x(\beta, \gamma)F_x(\beta\gamma, \zeta)\Big)F_x(\alpha, \beta\gamma\zeta)
\nonumber\\
&=&{\overline \Psi}(\alpha)\Big(F_x(\beta, \gamma)F_x(\beta\gamma,
 \zeta)\Big)F_x(\alpha, \beta\gamma\zeta)
\overline{\Psi}(\alpha)(K_{_{x,\overline{\Psi}}}(\beta, \gamma, \zeta))
\nonumber\\
&=&{\overline \Psi}(\alpha)
(F_x(\beta, \gamma))[{\overline \Psi}(\alpha)(F_x(\beta\gamma, \zeta))
F_x(\alpha, \beta\gamma\zeta)]
\overline{\Psi}(\alpha)(K_{_{x,\overline{\Psi}}}(\beta, \gamma, \zeta))
\nonumber\\
&& c_x(f(\beta, \gamma), f(\beta\gamma, \zeta))
(\Psi(\alpha))^*(c_{x})(f(\beta, \gamma), f(\beta\gamma, \zeta))^{-1}\nonumber\\
&=&{\overline \Psi}(\alpha)(F_x(\beta, \gamma))
[K_{_{x, \overline{\Psi}}}(\alpha, \beta\gamma, \zeta)F_x(\alpha, \beta\gamma)
F_x(\alpha\beta\gamma, \zeta)]
\overline{\Psi}(\alpha)(K_{_{x, \overline{\Psi}}}(\beta, \gamma, \zeta))
\nonumber\\
&& c_x(f(\beta, \gamma), f(\beta\gamma, \zeta))
(\Psi(\alpha)^*(c_{x})(f(\beta, \gamma), f(\beta\gamma, \zeta))^{-1}
\nonumber\\
&=&[{\overline \Psi}(\alpha)(F_x(\beta, \gamma))F_x(\alpha, \beta\gamma)]
F_x(\alpha\beta\gamma, \zeta)
\overline{\Psi}(\alpha)(K_{_{x, \overline{\Psi}}})(\beta, \gamma, \zeta))
K_{_{x, \overline{\Psi}}}(\alpha, \beta\gamma, \zeta)\nonumber\\
&& c_x(f(\beta, \gamma), f(\beta\gamma, \zeta))
(\Psi(\alpha))^*(c_{x})(f(\beta, \gamma), f(\beta\gamma, \zeta))^{-1}
\nonumber \\
&=& [F_x(\alpha, \beta)F_x(\alpha\beta, \gamma)F_x(\alpha\beta\gamma, \zeta)]
\overline{\Psi}(\alpha)(K_{_{x, \overline{\Psi}}}(\beta, \gamma, \zeta))
K_{_{x, \overline{\Psi}}}(\alpha, \beta, \gamma)
K_{_{x, \overline{\Psi}}}(\alpha, \beta\gamma, \zeta) \nonumber\\
&& c_x(f(\beta, \gamma), f(\beta\gamma, \zeta))
(\Psi(\alpha))^*(c_{x})(f(\beta, \gamma), f(\beta\gamma, \zeta))^{-1}\nonumber
\end{eqnarray}

La formule de diff\'erence (21)  s'obtient maintenant par comparaison des deux
d\'eveloppements de l'\'el\'ement $L$ dans le groupe $\overline{G}$.
\end{proof}

Observons que puisque d'apr\`es la formule (19) on sait que la  cocha\^\i ne
$K_{_{x, \overline{\Psi}}} : \Pi^3\rightarrow Z(\overline{G})$ est \'egale \`a
la somme,
$$K_{_{x, \overline{\Psi}}} = \varphi_*(K_{_{x, \overline{\Psi}}}) +
(s_x)_*(K_{_{\Psi, f}}) \in Z(\overline G)\simeq \mathbb R\oplus Z(G)$$
et puisque la cocha\^\i ne image $s_x^*(K_{_{\Psi, f}}) :
\Pi^3\stackrel{K_{_{\Psi, f}}}{\longrightarrow}
 Z(G)\stackrel{s_x}{\longrightarrow}Z(\overline{G})$ est un $3$-cocycle, l'expression  (21)
  \'etablie dans la proposition 8 permet  de d\'eduire  le corollaire~:

\begin{corollary}
Pour toute classe de cohomologie born\'ee $x\in H_b^2(G, \mathbb R)$ le cobord
de la cocha\^\i ne born\'ee de composition  $\varphi_*(K_{_{x,
\overline{\Psi}}}) = c_{_{x, \overline{\Psi}}} : \Pi^3 \rightarrow \mathbb R$
est \'egal \`a,
\begin{eqnarray}
d(c_{_{x, \overline{\Psi}}})(\alpha, \beta, \gamma, \zeta) =
c_{_{x, \overline{\Psi}}}(\beta, \gamma, \zeta) -
c_{_{\Psi(\alpha)_b(x), \overline{\Psi}}}(\beta, \gamma, \zeta)
\end{eqnarray}
\end{corollary}

Notons qu'en cons\'equence  de la proposition 8 et de son corollaire 4,  si on
suppose que la classe de cohomologie born\'ee  $x\in H_b^2(G, \mathbb R)$ est
$\Pi$-invariante  comme dans  \cite{Bou} et \cite{Bou1},   on retrouve le fait
que $K_{_{x, \overline{\Psi}}} : \Pi^3\ \rightarrow Z(\overline{G})$  est un
$3$-cocycle et que la cocha\^\i ne de composition  $\varphi_*(K_{_{x,
\overline{\Psi}}})=c_{_{x, \overline{\Psi}}} : \Pi^3\rightarrow \mathbb R$ est
un $3$-cocycle born\'e qui repr\'esente par cons\'esquent un \'el\'ement de
l'espace de cohomologie born\'ee $H_b^3(\Pi, \mathbb R)$. Si l'on suppose enfin que  $\theta :
\Pi\rightarrow Out(G)$ provient d'une extension de groupes,  nous avons
d\'emontr\'e  dans \cite{Bou} que l'application lin\'eaire
$$\begin{array}{ccccc}
\delta : &H_b^2(G, \mathbb R)^\Pi&\longrightarrow &H_b^3(\Pi, \mathbb R) \\
&x&  \longmapsto & [\varphi_*(K_{_{x, \overline{\Psi}}})]
\end{array}$$
appel\'ee  op\'erateur de transgression  rend la suite (1) exacte.

\subsection{Un repr\'esentant canonique de la classe de cohomologie
 $[\theta]$}

Dans ce paragraphe, \'etant donn\'ee  une repr\'esentation ext\'erieure $\theta
: \Pi\rightarrow Out(G)$ nous munissons l'espace d'homologie $\ell_1$-r\'eduite
$\overline{H}_2^{\ell_1}(G, \mathbb R)$ de la structure de $\Pi$-module de
Banach induite par $\theta$. Et, pour tout noyau  abstrait $(\Psi, f)$ de la
repr\'esentation ext\'erieure $\theta$  nous nous proposons de  d\'emontrer que
la cocha\^\i ne born\'ee $\theta_{_{\Psi, f}} : \Pi^3\rightarrow
Z_2^{\ell_1}(G, \mathbb R)$ qui est d\'efinie par l'expression,
\begin{eqnarray}
\theta_{_{\Psi, f}}(\alpha, \beta, \gamma) &=
&\mathbf{m}_{_{2}}(\Psi(\alpha)(f(\beta, \gamma)),
f(\alpha, \beta\gamma)) -
\mathbf{m}_{_{2}}(f(\alpha, \beta), f(\alpha\beta, \gamma))
\end{eqnarray}
 induit un $3$-cocycle born\'e $\overline{\theta}_{_{\Psi, f}} : \Pi^3 \buildrel\hbox{$\theta_{_{\Psi, f}}$}\over
 {\hbox to 12mm {\rightarrowfill}} Z_2^{\ell_1}(G, \mathbb R) \longrightarrow \overline{H}_2^{\ell_1}(G, \mathbb R)$
 o\`u $\mathbf{m}_{_{2}} :
G^2\rightarrow Z_2^{\ell_1}(G, \mathbb R)$ d\'esigne la $2$-cocha\^\i ne
d\'efinie  par la formule (13). Nous montrerons dans la section suivante que la
classe de cohomologie induite par ce cocycle ne d\'epend pas du noyau abstrait $(\Psi, f)$.

D'abord, rappelons que puisque pour tout $2$-cobord born\'e $db\in B_b^2(G,
\mathbb R)$
 on a,
$$<db, \mathbf{m}_{_{2}}(g, h)>=0, \qquad \forall g, h\in G$$
(cf. aff. 5) il en r\'esulte que pour tout $2$-cocycle born\'e $c\in
Z_b^2(G, \mathbb R)$ la $3$-cocha\^\i ne born\'ee de composition
$\varphi_*(K_{_{x, \overline{\Psi}}})=c_{_{x, \overline{\Psi}}} :
\Pi^3\rightarrow \mathbb R$ d\'efinie par la formule  (20) peut \^etre
exprim\'ee par le crochet de dualit\'e,
\begin{eqnarray}c_{_{x, \overline{\Psi}}}(\alpha, \beta, \gamma)=<c, \theta_{_{\Psi, f}}(\alpha, \beta, \gamma)>=
<c_x, \theta_{_{\Psi, f}}(\alpha, \beta, \gamma)>
\end{eqnarray}
o\`u  $c_x$ d\'esigne l'unique $2$-cocycle born\'e homog\`ene qui repr\'esente
la classe de cohomologie born\'ee $x=[c]\in H_b^2(G, \mathbb R)$.

\begin{proposition}
La cocha\^\i ne born\'ee $\overline{\theta}_{_{\Psi, f}} : \Pi^3
\buildrel\hbox{$\theta_{_{\Psi, f}}$}\over {\hbox to 12mm {\rightarrowfill}}
Z_2^{\ell_1}(G, \mathbb R) \longrightarrow \overline{H}_2^{\ell_1}(G, \mathbb
R)$ qui est  induite en homologie $\ell_1$-r\'eduite par l'expression (23)
est un $3$-cocycle.
\end{proposition}

\begin{proof}[D\'emonstration]  Notons que puisque la cocha\^\i ne
$\mathbf{m}_{_{2}} : G^2\rightarrow Z_2^{\ell_1}(G, \mathbb R)$ annule le
cobord des cocha\^\i nes  born\'ees  $b : G \rightarrow \mathbb R$ il en
r\'esulte imm\'ediatement que :
$$
<db, {\theta}_{_{\Psi, f}}>=0 \quad \textrm{ et que } \quad <db, d_{_{\Pi}}({\theta}_{_{\Psi,
f}})>=0
$$
o\`u $d_{_{\Pi}}$ d\'esigne la diff\'erentielle d\'efinie sur le complexe des
cocha\^\i nes non homog\`enes d\'efinies sur le groupe $\Pi$   \`a valeurs
dans le $\Pi$-module de Banach $\overline{H}_2^{\ell_1}(G, \mathbb R)$. Par
cons\'equent, pour d\'emontrer  que  la cocha\^\i ne
$\overline{\theta}_{_{\Psi, f}} : \Pi^3\rightarrow\overline{H}_2^{\ell_1}(G,
\mathbb R)$ est un $3$-cocycle,   il suffit de d\'emontrer que pour toute
classe de cohomologie born\'ee  $x=[c_x]\in H_b^2(G, \mathbb R)$ et pour tous
les \'el\'ements $\alpha$, $\beta$, $\gamma$ et  $\zeta \in \Pi$ le crochet de
dualit\'e,
$$<c_x, d_{_{\Pi}}({\theta}_{_{\Psi, f}})(\alpha, \beta,
\gamma, \zeta)>,$$ est nul.

En effet, par d\'efinition de la diff\'erentielle  $d_{_{\Pi}}$   on peut
\'ecrire  :
\begin{eqnarray}
<c_x, d_{_{\Pi}}(\theta_{_{\Psi, f}})(\alpha, \beta, \gamma, \zeta)> &=&
<c_x, \Psi(\alpha)_*(\theta_{_{\Psi, f}}(\beta, \gamma, \zeta))> -
<c_x, \theta_{_{\Psi, f}}(\alpha\beta, \gamma, \zeta)>\nonumber\\
&&+ <c_x, \theta_{_{\Psi, f}}(\alpha, \beta\gamma, \zeta)> - <c_x,
\theta_{_{\Psi, f}}(\alpha, \beta, \gamma\zeta)> \nonumber\\
&&+ <c_x, \theta_{_{\Psi, f}}(\alpha, \beta, \gamma)>\nonumber\\
&=& <\Psi(\alpha)^*(c_x), \theta_{_{\Psi, f}}(\beta, \gamma, \zeta))> -
c_{_{x, \overline{\Psi}}}(\alpha\beta, \gamma, \zeta)\nonumber\\
&&+ c_{_{x, \overline{\Psi}}}(\alpha, \beta\gamma, \zeta) -
c_{_{x, \overline{\Psi}}}(\alpha, \beta, \gamma\zeta) +
c_{_{x, \overline{\Psi}}}(\alpha, \beta, \gamma)\nonumber\\
&=&[<(\Psi(\alpha))^*(c_{x}), \theta_{_{\Psi, f}}(\beta, \gamma, \zeta)>-
c_{_{x, \overline{\Psi}}}(\beta, \gamma, \zeta)] \nonumber\\
&+& [ c_{_{x, \overline{\Psi}}}(\beta, \gamma, \zeta) -
c_{_{x, \overline{\Psi}}}(\alpha\beta, \gamma, \zeta)+
c_{_{x, \overline{\Psi}}}(\alpha, \beta\gamma, \zeta)\nonumber\\
&-& c_{_{x, \overline{\Psi}}}(\alpha, \beta, \gamma\zeta) +
c_{_{x, \overline{\Psi}}}(\alpha, \beta, \gamma)]\nonumber\\
&=&[<(\Psi(\alpha))^*(c_x), \theta_{_{\Psi, f}}(\beta, \gamma, \zeta)> -
c_{_{x, \overline{\Psi}}}(\beta, \gamma, \zeta)] \nonumber \\
&& +
d(c_{_{x, \overline{\Psi}}})(\alpha, \beta, \gamma, \zeta)\nonumber
\end{eqnarray}

Ainsi, si on applique la formule (22) du corollaire 4 qui nous donne :
$$d(c_{_{x, \overline{\Psi}}})(\alpha, \beta, \gamma, \zeta) =
c_{_{x, \overline{\Psi}}}(\beta, \gamma, \zeta)- c_{_{\Psi_b(\alpha)(x),
\overline{\Psi}}}(\beta, \gamma, \zeta), \forall x\in H_b^2(G, \mathbb R)$$ on
d\'eduit finalement que le crochet de dualit\'e $<c_x,
d_{_{\Pi}}(\theta_{_{\Psi, f}})(\alpha, \beta, \gamma, \zeta)>$ est nul et que
la $2$-cha\^\i ne $d_{_{\Pi}}(\theta_{_{\Psi, f}})(\alpha, \beta, \gamma,
\zeta) \in\overline{B_2^{\ell_1}(G, \mathbb R)}$. Donc, si on passe  dans le
$\Pi$-module de Banach d'homologie $\ell_1$-r\'eduite
$\overline{H}_2^{\ell_1}(G, \mathbb R)$ on obtient
$d_{_{\Pi}}(\overline{\theta}_{_{\Psi, f}})(\alpha, \beta, \gamma, \zeta)=0$.
\end{proof}

\subsection{La classe de cohomologie born\'ee $[\overline{\theta}_{_{\Psi, f}}]$
est bien d\'efinie}

Dans ce paragraphe, nous nous proposons de d\'emontrer que  la classe de
cohomologie $[\overline{\theta}_{_{\Psi, f}}]\in H_b^3(\Pi,
\overline{H}_2^{\ell_1}(G, \mathbb R))$ d\'efinie par l'expression (23) \`a
partir du noyau abstrait $(\Psi, f)$  ne d\'epend que de la repr\'esenta-tion
ext\'erieure $\theta : \Pi\rightarrow Out(G)$.

Pour cela, consid\'erons deux noyaux abstraits  $(\Psi, f)$ et $(\Psi', f')$
associ\'es \`a la repr\'esentation ext\'erieure $\theta : \Pi\rightarrow
Out(G)$. Observons que puisque pour tout $\alpha\in \Pi$ les automorphisemes
$\Psi(\alpha)$ et $\Psi'(\alpha)$ repr\'esentent le m\^eme automorphisme
ext\'erieur $\theta(\alpha)\in Out(G)$, il existe  une application $h :
\Pi\rightarrow G$ telle que pour tout $\alpha\in\Pi$, $\Psi'(\alpha) =
i_{_{h(\alpha)}}\circ\Psi(\alpha)$. Ainsi, par d\'efinition des d\'efauts $f$
et $f' : \Pi^2 \rightarrow G$  il existe une cocha\^\i ne ab\'elienne $z :
\Pi^2\rightarrow Z(G)$ qui les relient avec l'application $h : \Pi\rightarrow
G$ dans la relation :
\begin{eqnarray}
f'(\alpha, \beta)h(\alpha\beta)z(\alpha, \beta) = h(\alpha)
\Psi(\alpha)(h(\beta))f(\alpha, \beta),
\qquad \forall \alpha, \beta\in\Pi.
\end{eqnarray}

Rappelons que dans le lemme 4, \`a partir des donn\'ees pr\'ec\'edentes nous
avons construit des familles de bijections $\overline{\Psi} : \Pi\rightarrow
Sym(\overline{G}, \mathbb R)$ et $\overline{\Psi}' : \Pi\rightarrow
Sym(\overline{G}, \mathbb R)$ (cf. formule (16)) dont le d\'efaut  pour
qu'elles soient des homomorphismes est donn\'e respectivement par les
applications, $F_x=s_x\circ f$ et  $F_x'=s_x\circ f' : \Pi^2\rightarrow
\overline{G}$ (cf. l'assertion  4 du lemme 4).

Le lemme suivant joue un r\^ole crucial pour comparer les deux classes de
cohomologie born\'ee  $[\overline{\theta}_{{\Psi, f}}]$ et
$[\overline{\theta}_{{\Psi', f'}}]$ dans le groupe $H_b^3(\Pi,
\overline{H}_2^{\ell_1}(G, \mathbb R))$.

\begin{lemma} Avec les notations ci-dessus on a les propositions suivantes~:
\begin{enumerate}
\item Il existe une application $h_0 : \Pi\rightarrow\overline{G}$ telle
    que pour tout $\alpha\in \Pi$,
    $i_{_{h_0(\alpha)}}\circ\overline{\Psi}(\alpha) =
    \overline{\Psi}'(\alpha)$.
\item Il existe une cocha\^\i ne ab\'elienne $\overline{b} :
    \Pi^2\rightarrow Z(\overline{G})$ telle que pour tous $\alpha$ et
    $\beta\in\Pi$,
$$\overline{b}(\alpha, \beta)F_x'(\alpha, \beta)h_0(\alpha\beta)=
h_0(\alpha)
\overline{\Psi}(\alpha)(h_0(\beta))F_x(\alpha, \beta).$$
\item La cocha\^\i ne r\'eelle  $\varphi_*(\overline{b}) :
    \Pi^2\stackrel{\overline{b}}{\longrightarrow}
    Z(\overline{G})\stackrel{\varphi}{\longrightarrow}\mathbb R$ est
    \'egale au  crochet de dualit\'e $\varphi_*(\overline{b})=<c_x,
    \lambda>$ o\`u $\lambda : \Pi^2\rightarrow Z_2^{\ell_1}(G, \mathbb R)$
    est une cocha\^\i ne born\'ee qui associe \`a chaque couple $(\alpha,
    \beta)\in\Pi^2$ un $2$-cycle qui repr\'esente une classe d'homologie
    $\ell_1$-r\'eduite donn\'ee par l'expression :
\begin{eqnarray}
\overline{\lambda}(\alpha, \beta) &=& \overline{\mathbf{m}}_{_{2}}(h(\alpha), \Psi(\alpha)(h(\beta))) +
\overline{\mathbf{m}}_{_{2}}(\Psi(\alpha)(h(\beta)), f(\alpha, \beta)) \\
&& - \overline{\mathbf{m}}_{_{2}}(f'(\alpha, \beta), h(\alpha\beta))\in
\overline{H}_2^{\ell_1}(G, \mathbb R) \nonumber
\end{eqnarray}
et o\`u $\overline{\mathbf{m}}_{_{2}} : G^2 \rightarrow Z_2^{\ell_1}(G,
\mathbb R)$ d\'esigne la cocha\^\i ne d\'efinie par l'expression (13).
\end{enumerate}
\end{lemma}

\begin{proof}[D\'emonstration] 1) Puisque pour tout $\alpha\in\Pi$ on sait que,
$\Psi'(\alpha)=i_{_{h(\alpha)}}\circ\Psi(\alpha)$, en posant
$h_0(\alpha)=s_x\circ h(\alpha)\in \overline{G}$ la formule $(16$) qui
d\'efinit la famille de bijections $\overline{\Psi}' : \Pi\rightarrow
Sym(\overline{G},\mathbb R)$ permet de voir que  pour tout \'el\'ement
$\overline{g}\in\overline{G}$ on a
\begin{eqnarray}
\overline{\Psi}'(\alpha)(\overline{g}) &=&
s_x(\Psi'(\alpha)(p(\overline{g}))\varphi(\overline{g})\nonumber \\
&=& s_x(i_{_{h(\alpha)}}\circ\Psi(\alpha)(p(\overline{g}))
\varphi(\overline{g})\nonumber\\
&=&s_x\circ p(i_{_{h_0(\alpha)}}\circ\overline{\Psi}(\alpha)(\overline{g}))
\varphi(\overline{g})\nonumber
\end{eqnarray}

En remarquant que l'application $s_x\circ p : \overline{G}
\rightarrow\overline{G}$ commute avec la conjugaison dans le groupe
$\overline{G}$ (cf. aff. 4), on  d\'eduit que $\overline{\Psi}'(\alpha)=i_{_{h_0(\alpha)}}
\circ\overline{\Psi}(\alpha)$.

\noindent 2) Observons que si pour tous les \'el\'ements $\alpha$ et $\beta$ du
groupe $\Pi$ on compose l'expression
$\overline{\Psi}'(\alpha)=i_{_{h_0(\alpha)}} \circ\overline{\Psi}(\alpha)$ avec
$\overline{\Psi}'(\beta)=i_{_{h_0(\beta)}} \circ\overline{\Psi}(\beta)$, on
obtient  par application de la formule (4) du lemme 4 qu'on a
$$\overline{\Psi}'(\alpha)\circ \overline{\Psi}'(\beta)
=i_{_{h_0(\alpha)\overline{\Psi}(\alpha)(h_0(\beta))F_x(\alpha, \beta)}}\circ
\overline{\Psi}(\alpha\beta)$$

D'autre part, puisque on a $\overline{\Psi}'(\alpha)\circ
\overline{\Psi}'(\beta)= i_{_{F_x'(\alpha,
\beta)}}\circ\overline{\Psi}'(\alpha\beta)= i_{_{F_x'(\alpha,
\beta)h_0(\alpha\beta)}}\circ\overline{\Psi}(\alpha\beta)$ on en d\'eduit
l'\'egalit\'e des automorphismes int\'erieurs :
$$\displaystyle{i_{_{h_0(\alpha)\overline{\Psi}(\alpha)(h_0(\beta))F_x(\alpha,
\beta)}} = i_{_{F_x'(\alpha, \beta)h_0(\alpha\beta)}}}$$

Un  automorphisme int\'erieur n'\'etant  d\'efini qu'\`a un \'el\'ement du
centre  pr\`es, il existe donc une cocha\^\i ne ab\'elienne $\overline{b} :
\Pi^2\rightarrow Z(\overline G)$ qui r\'ealise la relation recherch\'ee,
$$\overline{b}(\alpha, \beta)F_x'(\alpha, \beta)h_0(\alpha\beta)=
h_0(\alpha)
\overline{\Psi}(\alpha)(h_0(\beta))F_x(\alpha, \beta), \qquad \forall \alpha, \beta\in\Pi.$$

\noindent 3)  Pour pouvoir \'ecrire la $2$-cocha\^\i ne r\'eelle
$\varphi_*(\overline{b}) : \Pi^2\rightarrow \mathbb R$ au moyen d'un crochet de
dualit\'e,  nous allons d\'evelopper dans le groupe $\overline G$ les deux
membres de la relation \'etablie dans 2) comme suit.

a) Premi\`erement, notons que puisque le $2$-cocycle $c_x : G^2\rightarrow
\mathbb R$ est \'egal au d\'efaut de la section $s_x : G\rightarrow
\overline{G}$ on peut  \'ecrire,
\begin{eqnarray}
\overline{b}(\alpha, \beta)F_x'(\alpha, \beta)h_0(\alpha\beta)&=&
\overline{b}(\alpha, \beta)s_x\circ f'(\alpha, \beta)s_x\circ h(\alpha\beta)
\nonumber \\
&=&
\overline{b}(\alpha, \beta)c_x(f'(\alpha, \beta), h(\alpha\beta))s_x(
f'(\alpha, \beta)h(\alpha\beta))\nonumber
\end{eqnarray}

b) D'autre part, si l'on exprime  la bijection $\overline{\Psi}(\alpha)$ \`a
l'aide de (16)  et si on utilise la propri\'et\'e  $\varphi\circ s_x=0$ on
obtient :
\begin{eqnarray}
h_0(\alpha)\overline{\Psi}(\alpha)(h_0(\beta))F_x(\alpha, \beta)&=&
s_x\circ h(\alpha)s_x(\Psi(\alpha)(h(\beta)))s_x\circ f(\alpha,
\beta)j\circ\varphi\circ s_x(h(\beta))\nonumber \\
&=& c_x(h(\alpha), \Psi(\alpha)(h(\beta)))
s_x(h(\alpha)\Psi(\alpha)(h(\beta)))s_x(f(\alpha, \beta))\nonumber\\
&=&c_x(h(\alpha), \Psi(\alpha)(h(\beta)))c_x(\Psi(\alpha)(h(\beta)),
f(\alpha, \beta))\nonumber\\
&& s_x(h(\alpha)\Psi(\alpha)(h(\beta))f(\alpha, \beta))\nonumber
\end{eqnarray}

Ainsi, si maintenant on applique l'expression (25) qui relie les d\'efauts
respectifs $f$ et $f'$ des rel\`evements ensemblistes $\Psi$ et $\Psi'$ de la
repr\'esentation ext\'erieure  $\theta : \Pi\rightarrow Out(G)$ i.e.,
$$f'(\alpha, \beta)h(\alpha\beta)z(\alpha, \beta) = h(\alpha)
\Psi(\alpha)(h(\beta))f(\alpha, \beta) \quad \textrm{ o\`u } \quad z(\alpha, \beta)\in Z(G),
\quad \forall \alpha, \beta\in\Pi$$
 on pourra alors r\'e\'ecrire l'\'el\'ement $h_0(\alpha)\overline{\Psi}(\alpha)(h_0(\beta))F_x(\alpha, \beta)$ sous
la forme suivante :
\begin{eqnarray}
h_0(\alpha)\overline{\Psi}(\alpha)(h_0(\beta))F_x(\alpha,
\beta)&=& c_x(h(\alpha), \Psi(\alpha)(h(\beta)))c_x(\Psi(\alpha)(h(\beta)),
f(\alpha, \beta))\nonumber\\
&&s_x(f'(\alpha, \beta)h(\alpha\beta)z(\alpha, \beta))\nonumber
\end{eqnarray}

c) Et, puisque pour tous $g\in G$ et $z\in Z(G)$ on sait que
$s_x(gz)=s_x(g)\cdot s_x(z)\in \overline{G}$ (cf. aff. 4) on d\'eduit  de ce qui
pr\'ec\`ede que
\begin{eqnarray}
\overline{b}(\alpha, \beta) &=&c_x(h(\alpha), \Psi(\alpha)(h(\beta))) +
c_x(\Psi(\alpha)(h(\beta)), f(\alpha, \beta)) -
c_x(f'(\alpha, \beta), h(\alpha\beta)) \nonumber \\
&+& s_x(z(\alpha, \beta))\nonumber \\
&=& <c_x, \lambda(\alpha, \beta)> + s_x(z(\alpha, \beta))\nonumber
\end{eqnarray}
o\`u $\lambda(\alpha, \beta) = {\mathbf{m}}_{_{2}}(h(\alpha),
\Psi(\alpha)(h(\beta))) + {\mathbf{m}}_{_{2}}(\Psi(\alpha)(h(\beta)), f(\alpha,
\beta)) - {\mathbf{m}}_{_{2}}(f'(\alpha, \beta), h(\alpha\beta))$.

Enfin, si on applique le quasi-morphisme homog\`ene $\varphi :
\overline{G}\rightarrow \mathbb R$ sur la derni\`ere expression de
l'\'el\'ement $\overline{b}(\alpha, \beta)\in Z(\overline{G})$ on obtient
$\varphi_*(\overline{b})(\alpha, \beta) = <c_x, \lambda(\alpha, \beta)>$ car
$\varphi\circ s_x=0$ et pour tout $t\in \mathbb R, \varphi(t)=t$.
\end{proof}

Pour finir cette section nous allons d\'emontrer la proposition suivante qui
affirme que  la classe de cohomologie born\'ee $[\theta_{_{\Psi, f}}]\in
H_b^3(\Pi, \overline{H}_2^{\ell_1}(G, \mathbb R))$ ne d\'epend pas  du noyau
abstrait choisi $(\Psi, f)$.

\begin{proposition}
Le cobord de la cocha\^\i ne born\'ee $\overline{\lambda} : \Pi^2\rightarrow
\overline{H}_2^{\ell_1}(G, \mathbb R)$ qui est  d\'efinie par l'expression
(26) est  \'egal \`a   $d\overline{\lambda}=\overline{\theta}_{_{\Psi', f'}}
- \overline{\theta}_{_{\Psi, f}}$.

En cons\'equence, la classe de cohomologie born\'ee $[\theta_{_{\Psi, f}}]\in
H_b^3(\Pi, \overline{H}_2^{\ell_1}(G, \mathbb R))$ ne d\'epend que de la
repr\'esentation ext\'erieure $\theta : \Pi \rightarrow Out(G)$.
\end{proposition}

\begin{proof}[D\'emonstration] Pour d\'emontrer que le cobord
$d\overline{\lambda} : \Pi^3\rightarrow \overline{H}_2^{\ell_1}(G, \mathbb R)$
est \'egal \`a la diff\'erence $\overline{\theta}_{_{\Psi', f'}} -
\overline{\theta}_{_{\Psi, f}}$ nous d\'evelopperons ci-dessous le produit
$$L'=
h_0(\alpha)\overline{\Psi}(\alpha)[h_0(\beta)
\overline{\Psi}(\beta)(h_0(\gamma))F_x(\beta, \gamma)]F_x(\alpha, \beta\gamma)
$$ dans  le groupe $\overline{G}$ de deux fa\c cons.

\noindent 1) Dans le premier d\'eveloppement de l'\'el\'ement $L'\in
\overline{G}$, nous allons  utiliser l'expression
$\overline{\Psi}(\alpha)\circ\overline{\Psi}(\beta)= i_{_{F_x(\alpha,
\beta)}}\circ\overline{\Psi}(\alpha\beta)$ (cf.  lemme 4) et la formule (17)
de la proposition 6 qui contr\^ole la d\'eviation de la bijection
$\overline{\Psi}(\alpha) : \overline{G}\rightarrow \overline{G}$ \`a \^etre un
homomorphisme.

\begin{eqnarray}
L'&=&h_0(\alpha)\overline{\Psi}(\alpha)[h_0(\beta)
\overline{\Psi}(\beta)(h_0(\gamma))F_x(\beta, \gamma)]F_x(\alpha, \beta\gamma)\nonumber\\
&=& h_0(\alpha)\overline{\Psi}(\alpha)(h_0(\beta))
[\overline{\Psi}(\alpha)(\overline{\Psi}(\beta)(h_0(\gamma))
F_x(\beta, \gamma))]F_x(\alpha, \beta\gamma)\nonumber\\
&&[c_x(h(\alpha), \Psi(\beta)(h(\gamma))f(\beta, \gamma))
(\Psi(\alpha))^*(c_x)(h(\beta), \Psi(\beta)(h(\gamma))f(\beta, \gamma))^{-1}]
\nonumber\\
&=& h_0(\alpha)\overline{\Psi}(\alpha)(h_0(\beta))[
\overline{\Psi}(\alpha)\circ\overline{\Psi}(\beta)(h_0(\gamma))
\overline{\Psi}(\alpha)(F_x(\beta, \gamma))]F_x(\alpha, \beta\gamma)
\nonumber\\
&&[c_x(\Psi(\beta)(h(\gamma)), f(\beta, \gamma))(\Psi(\alpha))^*(c_x)(
\Psi(\beta)(h(\gamma)), f(\beta, \gamma))^{-1}]\nonumber\\
&&[c_x(h(\alpha), \Psi(\beta)(h(\gamma))f(\beta, \gamma))
(\Psi(\alpha))^*(c_x)(h(\beta), \Psi(\beta)(h(\gamma))f(\beta,
\gamma))^{-1}]\nonumber\\
&=&[h_0(\alpha)\overline{\Psi}(\alpha)(h_0(\beta))F_x(\alpha, \beta)][
\overline{\Psi}(\alpha\beta)(h_0(\gamma))F_x(\alpha, \beta)^{-1}]\nonumber\\
&&[\overline{\Psi}(\alpha)(F_x(\beta, \gamma))]F_x(\alpha, \beta\gamma)]\nonumber\\
&&[c_x(\Psi(\beta)(h(\gamma)), f(\beta, \gamma))(\Psi(\alpha))^*(c_x)(
\Psi(\beta)(h(\gamma)), f(\beta, \gamma))^{-1}]\nonumber\\
&&[c_x(h(\alpha), \Psi(\beta)(h(\gamma))f(\beta, \gamma))
(\Psi(\alpha))^*(c_x)(h(\beta), \Psi(\beta)(h(\gamma))f(\beta,
\gamma))^{-1}]\nonumber
\end{eqnarray}

Pour poursuivre ce d\'eveloppement de l'\'el\'ement $L'\in\overline{G}$ nous
allons utiliser  l'expression de la cocha\^\i ne ab\'elienne $\overline{b} :
\Pi^2\rightarrow Z(\overline G)$ du lemme 5 et  l'expression (18) i.e.
$$K_{_{x, \overline{\Psi}}}(\alpha, \beta, \gamma) =
{\overline \Psi}(\alpha)(F_x(\beta, \gamma))F_x(\alpha, \beta\gamma)
F_x(\alpha\beta, \gamma)^{-1}F_x(\alpha, \beta)^{-1}$$ associ\'ee au noyau
abstrait $(\Psi, f)$.

\begin{eqnarray}
L'&=&[\overline{b}(\alpha, \beta)F_x'(\alpha, \beta)h_0(\alpha\beta)][
\overline{\Psi}(\alpha\beta)(h_0(\gamma))F_x(\alpha, \beta)^{-1}]\nonumber\\
&&[K_{_{x, \overline{\Psi}}}(\alpha, \beta, \gamma)
F_x(\alpha, \beta)F_x(\alpha\beta, \gamma)]
\nonumber\\
&&[c_x(\Psi(\beta)(h(\gamma)), f(\beta, \gamma))(\Psi(\alpha))^*(c_x)(
\Psi(\beta)(h(\gamma)), f(\beta, \gamma))^{-1}]\nonumber\\
&&[c_x(h(\alpha), \Psi(\beta)(h(\gamma))f(\beta, \gamma))
(\Psi(\alpha))^*(c_x)(h(\beta), \Psi(\beta)(h(\gamma))f(\beta,
\gamma))^{-1}]\nonumber\\
&=&[\overline{b}(\alpha, \beta)K_{_{x, \overline{\Psi}}}(\alpha, \beta, \gamma)
]F_x'(\alpha, \beta)[h_0(\alpha\beta)\overline{\Psi}(\alpha\beta)(h_0(\gamma))F_x(\alpha\beta,
\gamma)]\nonumber\\
&&[c_x(\Psi(\beta)(h(\gamma)), f(\beta, \gamma))(\Psi(\alpha))^*(c_x)(
\Psi(\beta)(h(\gamma)), f(\beta, \gamma))^{-1}]\nonumber\\
&&[c_x(h(\alpha), \Psi(\beta)(h(\gamma))f(\beta, \gamma))
(\Psi(\alpha))^*(c_x)(h(\beta), \Psi(\beta)(h(\gamma))f(\beta,
\gamma))^{-1}]\nonumber\\
&=&[\overline{b}(\alpha, \beta)\overline{b}(\alpha\beta, \gamma)
K_{_{x, \overline{\Psi}}}(\alpha, \beta, \gamma)]
[F_x'(\alpha, \beta)F_x'(\alpha\beta, \gamma)h_0(\alpha\beta\gamma)]
\nonumber\\
&&[c_x(\Psi(\beta)(h(\gamma)), f(\beta, \gamma))(\Psi(\alpha))^*(c_x)(
\Psi(\beta)(h(\gamma)), f(\beta, \gamma))^{-1}]\nonumber\\
&&[c_x(h(\alpha), \Psi(\beta)(h(\gamma))f(\beta, \gamma))
(\Psi(\alpha))^*(c_x)(h(\beta), \Psi(\beta)(h(\gamma))f(\beta,
\gamma))^{-1}]\nonumber
\end{eqnarray}

\noindent 2) Dans le deuxi\`eme d\'eveloppement de l'\'el\'ement
$L'\in\overline{G}$, nous allons utiliser la formule (18) qui donne
l'expression des cocha\^\i nes  $K_{_{x, \overline{\Psi}}}$ et $K_{_{x,
\overline{\Psi'}}} : \Pi^3\rightarrow Z(\overline{G})$ et le fait que la
bijection $\overline{\Psi}(\alpha) : \overline{G}\rightarrow \overline{G}$
v\'erifie la relation $\overline{\Psi}(\alpha)(\overline{g}\ \overline{z}) =
\overline{\Psi}(\alpha)(\overline{g})\overline{\Psi}(\alpha)(\overline{z})$
pour tous $\overline{g}\in\overline{G}$ et $\overline{z}\in Z(\overline{G})$
(cf. lemme 4). De m\^eme, nous allons utiliser l'existence de $h_0 :
\Pi\rightarrow \overline{G}$ tel que pour tout $\alpha\in \Pi$,
$\overline{\Psi}'(\alpha)=i_{_{h_0(\alpha)}} \circ\overline{\Psi}(\alpha)$
prouv\'ee dans le lemme 5.

\begin{eqnarray}
L'&=&h_0(\alpha)\overline{\Psi}(\alpha)[h_0(\beta)
\overline{\Psi}(\beta)(h_0(\gamma))F_x(\beta, \gamma)]
F_x(\alpha, \beta\gamma)\nonumber\\
&=&h_0(\alpha)\overline{\Psi}(\alpha)[\overline{b}(\beta, \gamma)F_x'(\beta,
\gamma)h_0(\beta\gamma)]F_x(\alpha, \beta\gamma)\nonumber\\
&=&\overline{\Psi}(\alpha)(\overline{b}(\beta, \gamma))
[h_0(\alpha)\overline{\Psi}(\alpha)(F_x'(\beta, \gamma)
h_0(\beta\gamma)]F_x(\alpha, \beta\gamma)\nonumber\\
&=&\overline{\Psi}(\alpha)(\overline{b}(\beta, \gamma))
[h_0(\alpha)\overline{\Psi}(\alpha)(F_x'(\beta, \gamma))
\overline{\Psi}(\alpha)(h_0(\beta\gamma))F_x(\alpha,
\beta\gamma)\nonumber\\
&&[c_x(f'(\beta, \gamma), h(\beta\gamma)))(\Psi(\alpha))^*(c_x)(
f'(\beta, \gamma), h(\beta\gamma)))^{-1}]\nonumber\\
&=&\overline{\Psi}(\alpha)(\overline{b}(\beta, \gamma))
[h_0(\alpha)\overline{\Psi}(\alpha)(F_x'(\beta, \gamma))h_0(\alpha)^{-1}]
[h_0(\alpha)\overline{\Psi}(\alpha)(h_0(\beta\gamma))F_x(\alpha, \beta\gamma)]
\nonumber\\
&&[c_x(f'(\beta, \gamma), h(\beta\gamma)))(\Psi(\alpha))^*(c_x)(
f'(\beta, \gamma), h(\beta\gamma)))^{-1}]\nonumber\\
&=&\overline{\Psi}(\alpha)(\overline{b}(\beta, \gamma))
[\overline{\Psi}'(\alpha)(F_x'(\beta, \gamma))]
[\overline{b}(\alpha, \beta\gamma)
F_x'(\alpha, \beta\gamma)h_0(\alpha\beta\gamma)]\nonumber\\
&&[c_x(f'(\beta, \gamma), h(\beta\gamma)))(\Psi(\alpha))^*(c_x)(
f'(\beta, \gamma), h(\beta\gamma)))^{-1}]\nonumber\\
&=&\overline{\Psi}(\alpha)(\overline{b}(\beta, \gamma))
\overline{b}(\alpha, \beta\gamma)
[K_{_{x, \overline{\Psi}'}}(\alpha, \beta, \gamma)
F_x'(\alpha, \beta)F_x'(\alpha\beta, \gamma)]h_0(\alpha\beta\gamma)\nonumber\\
&&[c_x(f'(\beta, \gamma), h(\beta\gamma)))(\Psi(\alpha))^*(c_x)(
f'(\beta, \gamma), h(\beta\gamma)))^{-1}]\nonumber
\end{eqnarray}

\noindent 3) Si on  simplifie par  $F_x'(\alpha, \beta)F_x'(\alpha\beta,
\gamma)h_0(\alpha\beta\gamma)\in\overline{G}$ qui appara\^\i t dans les deux
d\'eveloppements pr\'ec\'edents  de  $L' \in\overline{G}$ on pourra  \'ecrire
dans le sous-groupe ab\'elien additif $(Z(\overline{G}), +)$  que :
\begin{eqnarray}
K_{_{x, \overline{\Psi}}}(\alpha, \beta, \gamma)&=&
K_{_{x, \overline{\Psi}'}}(\alpha, \beta, \gamma) +
\overline{\Psi}(\alpha)(\overline{b}(\beta, \gamma))-
\overline{b}(\alpha\beta, \gamma) + \overline{b}(\alpha, \beta\gamma)
-\overline{b}(\alpha, \beta)\nonumber\\
&+& c_x(f'(\beta, \gamma), h(\beta\gamma))-
(\Psi(\alpha))_*(c_x)(f'(\beta, \gamma), h(\beta\gamma))\nonumber\\
&-& c_x(\Psi(\beta)(h(\gamma)), f(\beta, \gamma)) +
(\Psi(\alpha))_*(c_x)(\Psi(\beta)(h(\gamma)), f(\beta, \gamma))\nonumber\\
& -& c_x(h(\beta), \Psi(\beta)(h(\gamma))f(\beta, \gamma)) +
(\Psi(\alpha))_*(c_x)(h(\beta), \Psi(\beta)(h(\gamma))f(\beta,\gamma))
\nonumber \\
&=& K_{_{x, \overline{\Psi}'}}(\alpha, \beta, \gamma)
+ \overline{\Psi}(\alpha)(\overline{b}(\beta, \gamma))-
\overline{b}(\alpha\beta, \gamma) + \overline{b}(\alpha, \beta\gamma)
-\overline{b}(\alpha, \beta)\nonumber\\
&-&<c_x, \lambda(\beta, \gamma)> + <(\Psi(\alpha))^*(c_x), \lambda(\beta, \gamma)>
\nonumber
\end{eqnarray}
o\`u $\lambda : \Pi^2\rightarrow Z(\overline G)$ d\'esigne la cocha\^\i ne
d\'efinie dans le lemme 5 par, $\varphi_*(\overline b)=<c_x, \lambda>$.

Enfin,  puisque le quasi-morphisme homog\`ene $\varphi :
\overline{G}\rightarrow \mathbb R$  est $\overline{\Psi}$-invariant (cf. lemme
4),
$$\varphi(\overline{\Psi}(\alpha)(\overline{g}))=\varphi(\overline{g}), \quad \forall \alpha\in\Pi, \overline{g}\in
 \overline{G}$$
 et que pour tous $\alpha, \beta$ et $\gamma\in\Pi$ on a $\varphi(K_{_{x, \overline{\Psi}}}(\alpha,
\beta,\gamma))=<c_x,\theta_{_{\psi, f}}(\alpha, \beta, \gamma)>$ on obtient,
\begin{eqnarray}
<c_x,\theta_{_{\psi, f}}(\alpha, \beta, \gamma) - \theta_{_{\Psi',
f'}}(\alpha, \beta, \gamma)>&=&
\varphi(K_{_{x, \overline{\Psi}}}(\alpha,\beta,\gamma))-\varphi(K_{_{x, \overline{\Psi}'}}(\alpha,
\beta,\gamma)) \nonumber \\
&=&<c_x, \lambda(\beta, \gamma)>-
<c_x, \lambda(\alpha\beta, \gamma)> \nonumber \\
&+& <c_x, \lambda(\alpha, \beta\gamma)>
-<c_x, \lambda(\alpha, \beta)>\nonumber\\
& -&   <c_x, \lambda(\beta, \gamma)> +
<c_x, (\Psi(\alpha))_*(\lambda)(\beta, \gamma)>
\nonumber\\
&=& -<c_x, \lambda(\alpha\beta, \gamma)> +
<c_x, \lambda(\alpha, \beta\gamma)> \nonumber\\
&-&<c_x, \lambda(\alpha, \beta)>
 + <c_x, (\Psi(\alpha))_*(\lambda)(\beta, \gamma)>\nonumber\\
&=& <c_x, d\lambda(\alpha, \beta, \gamma)>.\nonumber
\end{eqnarray}

Ainsi,  en passant  dans l'espace d'homologie $\ell_1$-r\'eduite
$\overline{H}_2^{\ell_1}(G, \mathbb R)$ on d\'eduit de ce qui pr\'ec\`ede qu'on
a finalement $\overline{\theta}_{_{\Psi, f}}- \overline{\theta}_{_{\Psi', f'}}
= d\overline{\lambda}$.\end{proof}

\section{Expression de la diff\'erentielle $d_3$}

Consid\'erons  une extension de groupes discrets $1 \longrightarrow
G\stackrel{i}{\longrightarrow}\Gamma\stackrel{\sigma}{\longrightarrow}
\Pi\longrightarrow 1$ et fixons  une section ensembliste $s : \Pi\rightarrow
\Gamma$ pour l'homomorphisme surjectif $\sigma$. Rappelons qu'un  noyau
abstrait $(\Psi_s, f)$ peut \^etre associ\'e \`a la section $s : \Pi\rightarrow
\Gamma$  par les expressions suivantes,
$$f(\alpha, \beta)=s(\alpha)s(\beta)[s(\alpha\beta)]^{-1}
\quad \textrm{ et } \quad
\Psi_s(\alpha)(g)=s(\alpha)gs(\alpha)^{-1}, \quad  \forall\alpha, \beta\in\Pi, g\in G$$

Ainsi, puisque pour tous les \'el\'ements  $\alpha$ et $\beta\in \Pi$ on a
$\Psi_s(\alpha)\circ \Psi_s(\beta)=i_{_{f(\alpha, \beta)}}\circ
\Psi_s(\alpha\beta)$ on d\'eduit que l'application $\Psi_s : \Pi\rightarrow
Aut(G)$ induit un homomorphisme au niveau des automorphismes ext\'erieurs qu'on
notera $\theta : \Pi\rightarrow Out(G)$.

Dans le reste de cette section on pose $\Psi :=\Psi_s$.

Rappelons aussi que dans la section 4, en munissant l'espace d'homologie
$\ell_1$-r\'eduite $\overline{H}_2^{\ell_1}(G, \mathbb R)$ par la structure de
$\Pi$-module de Banach qui est induite par $\theta : \Pi\rightarrow Out(G)$,
nous avons d\'emontr\'e que pour tout noyau abstrait $(\Psi, f)$  l'expression
(23),
$$\overline{\theta}_{_{\Psi, f}}(\alpha, \beta, \gamma) =
\overline{\mathbf{m}}_{_{2}}(\Psi(\alpha)(f(\beta, \gamma), f(\alpha, \beta\gamma))-
\overline{\mathbf{m}}_{_{2}}(f(\alpha, \beta), f(\alpha\beta, \gamma)), \qquad \forall\alpha, \beta, \gamma\in \Pi
$$
d\'efinit  un  $3$-cocycle born\'e $\overline{\theta}_{_{\Psi, f}} : \Pi^3
\longrightarrow\overline{H}_2^{\ell_1}(G, \mathbb R)$ dont la classe de
cohomologie born\'ee associ\'ee $[\overline{\theta}_{_{\Psi, f}}]\in H_b^3(\Pi,
\overline{H}_2^{\ell_1}(G, \mathbb R))$   d\'epend seulement de  $\theta :
\Pi\rightarrow Out(G)$.

De m\^eme, rappelons que dans le  th\'eor\`eme principal A de la section 3 nous
avons d\'emontr\'e que l'expression  (13)
$$ \mathbf{m}_{_{2}}(g, h) = (g, h) - \mathbf{m}(g) + \mathbf{m}(gh) - \mathbf{m}(h)
= (g, h) - \mathbf{m}\circ\partial_2(g, h), \qquad \forall g, h\in G  $$
induit un  $2$-cocycle born\'e homog\`ene $\overline{\mathbf{m}}_{_{2}} :
G^2\rightarrow \overline{H}_2^{\ell_1}(G, \mathbb R)$ qui repr\'esente  une
classe de cohomologie born\'ee $\mathbf{g}_{_{2}}\in H_b^2(G,
\overline{H}_2^{\ell_1}(G, \mathbb R))$  unique pour la relation suivante,
$$x\cup \mathbf{g}_{_{2}}=x, \qquad \forall x\in H_b^2(G, \mathbb R).$$

\subsection{La diff\'erentielle $d_{_{3, \ell_1}}$ envoie   $\mathbf{g}_{_{2}}$ sur    $[\theta]$}
Soient $1 \longrightarrow
G\stackrel{i}{\longrightarrow}\Gamma\stackrel{\sigma}{\longrightarrow}
\Pi\longrightarrow 1$ une extension de groupes discrets et  $s : \Pi\rightarrow
\Gamma$ une section ensembliste de l'homomorphisme surjectif $\sigma$. Puisque
pour tout \'el\'ement $\gamma\in\Gamma$ on a $\sigma(\gamma)=
\sigma(s\circ\sigma(\gamma))$ on peut  d\'efinir une application $h :
\Gamma\rightarrow G$ en posant pour tout $\gamma\in \Gamma$,
$$ \gamma=h(\gamma).s\circ\sigma(\gamma)$$

Il r\'esulte de la d\'efinition de l'application  $h : \Gamma\rightarrow G$ que
$$h(g)=g, \ \forall g\in G$$

\begin{proposition}
L'image r\'eciproque $\sigma^*(\overline{\theta}_{_{\Psi, f}}) :
\Gamma^3\rightarrow\overline{H}_2^{\ell_1}(G, \mathbb R)$ est \'egale au cobord
de la concha\^\i ne born\'ee $\overline{T} : \Gamma^2\rightarrow
\overline{H}_2^{\ell_1}(G, \mathbb R)$  qui est d\'efinie par l'expression,
\begin{eqnarray}
\quad \overline{T}(\gamma_{_{1}}, \gamma_{_{2}}) =
\overline{\mathbf{m}}_{_{2}}(f(\sigma(\gamma_{_{1}}), \sigma(\gamma_{_{2}})),
h(\gamma_{_{1}}\gamma_{_{2}})^{-1}) - \overline{\mathbf{m}}_{_{2}}(\Psi(\sigma(\gamma_{_{1}}))(
h(\gamma_{_{2}})^{-1}), h(\gamma_{_{1}})^{-1})
\end{eqnarray}
En cons\'equence, la classe de cohomologie born\'ee $\sigma_b([\theta_{_{\Psi,
f}}])\in H_b^3(\Gamma, \overline{H}_2^{\ell_1}(G, \mathbb R))$ est nulle.
\end{proposition}

Tout le reste de la section sera consacr\'e \`a la d\'emonstration de la
proposition 11.

Soit $x\in H_b^2(G, \mathbb R)$ une classe de cohomologie born\'ee et $0
\longrightarrow \mathbb R\stackrel{j}{\longrightarrow}\overline
G{\rfl{\hfl{s_x}} {p}}G \longrightarrow 1$ l'extension centrale qui lui est
associ\'ee (cf. 3.4.2). Posons, $$h_{{x}} = s_x\circ h :
\Gamma{\longrightarrow}\overline{G} \qquad \textrm{  et } \qquad F_x = s_x\circ
f : \Pi^2\longrightarrow \overline{G}$$ les rel\`evements respectifs  sur
$\overline G$ des applications $h : \Gamma\longrightarrow G$ et $f : \Pi\times
\Pi\longrightarrow G$.

\begin{affirmation} Avec les notations ci-dessus la cocha\^\i ne non ab\'elienne $F_x(\sigma, \sigma) : \Gamma^2\longrightarrow\overline{G}$ s'\'ecrit sous la forme,
\begin{eqnarray}F_x(\sigma(\gamma_{_{1}}), \sigma(\gamma_{_{2}}))=
\overline{\Psi}(\sigma(\gamma_{_{1}}))(h_{{x}}(\gamma_{_{2}})^{-1})h_{{x}}(\gamma_{_{1}})^{-1}
h_{{x}}(\gamma_{_{1}}\gamma_{_{2}})<c_x, T(\gamma_{_{1}}, \gamma_{_{2}})>
\end{eqnarray}
o\`u $T:\Gamma^2\longrightarrow Z_2^{\ell_1}(G, \mathbb R)$ est un
repr\'esentant de la $2$-cocha\^\i ne  born\'ee $\overline{T} : \Gamma
\longrightarrow \overline{H}_2^{\ell_1}(G, \mathbb R)$ d\'efinie par
l'expression (27) et o\`u $\overline{\Psi} : \Pi\longrightarrow
Sym(\overline{G}, \mathbb R)$ d\'esigne la famille de bijections d\'efinies
dans le lemme 4 par l'expression (16).
\end{affirmation}

\begin{proof}[D\'emonstration] Avant d'\'etablir l'expression (28) nous allons d'abord
exprimer l'application $f(\sigma, \sigma) : \Gamma^2\longrightarrow G$ en
fonction des applications  $\Psi : \Pi\longrightarrow \textrm{Aut}(G)$ et $h :
\Gamma\longrightarrow G$.

D'abord observons que pour tous les \'el\'ements $\gamma_{_{1}}$ et
$\gamma_{_{2}}\in \Gamma$ on peut \'ecrire
\begin{eqnarray}
f(\sigma(\gamma_{_{1}}), \sigma(\gamma_{_{2}}))&=&
s\circ\sigma(\gamma_{_{1}})s\circ\sigma(\gamma_{_{2}})
s\circ\sigma(\gamma_{_{1}}\gamma_{_{2}})^{-1}\nonumber\\
&=& [s(\sigma(\gamma_{_{1}}))h(\gamma_2)^{-1} s(\sigma(\gamma_{_{1}}))^{-1}]
s\circ\sigma(\gamma_{_{1}})  \gamma_2 s\circ\sigma(\gamma_{_{1}}\gamma_{_{2}})^{-1}
\nonumber\\
&=&\Psi(\sigma(\gamma_{_{1}}))(h(\gamma_{_{2}})^{-1})h(\gamma_{_{1}})^{-1}
h(\gamma_{_{1}}\gamma_{_{2}}).
\nonumber
\end{eqnarray}

Ensuite, d\'eveloppons l'expression $s_x\circ  f(\sigma(\gamma_{_{1}}),
\sigma(\gamma_{_{2}}))\in\overline{G}$  tout en utilisant le fait que le
d\'efaut de la section $s_x : G\rightarrow \overline{G}$ est \'egal \`a
l'unique $2$-cocycle born\'e homog\`ene $c_x : G^2\longrightarrow \mathbb R$
qui repr\'esente la classe de cohomologie $x\in H_b^2(G, \mathbb R)$ :
\begin{eqnarray}
F_x(\sigma(\gamma_{_{1}}), \sigma(\gamma_{_{2}}))&=&
s_x\circ f(\sigma(\gamma_{_{1}}), \sigma(\gamma_{_{2}}))\nonumber \\
&=&s_x[\Psi(\sigma(\gamma_{_{1}}))
(h(\gamma_{_{2}})^{-1})h(\gamma_{_{1}})^{-1}h(\gamma_{_{1}}\gamma_{_{2}})]
\nonumber\\
&=&s_x[\Psi(\sigma(\gamma_{_{1}}))(h(\gamma_{_{2}})^{-1}]
s_x[h(\gamma_{_{1}})^{-1}h(\gamma_{_{1}}\gamma_{_{2}})]\nonumber\\
&&c_x(\Psi(\sigma(\gamma_{_{1}}))(h(\gamma_{_{2}})^{-1}),
h(\gamma_{_{1}})^{-1}h(\gamma_{_{1}}\gamma_{_{2}}))^{-1}\nonumber\\
&=& s_x[\Psi(\sigma(\gamma_{_{1}}))(h(\gamma_{_{2}})^{-1}]
s_x(h(\gamma_{_{1}})^{-1})s_x(h(\gamma_{_{1}}\gamma_{_{2}}))\nonumber
\\
&&c_x(\Psi(\sigma(\gamma_{_{1}}))(h(\gamma_{_{2}})^{-1}),
h(\gamma_{_{1}})^{-1}h(\gamma_{_{1}}\gamma_{_{2}}))^{-1}
c_x(h(\gamma_{_{1}})^{-1}, h(\gamma_{_{1}}\gamma_{_{2}}))^{-1}.\nonumber
\end{eqnarray}

Par cons\'equent, si on se rappelle que la bijection $\overline{\Psi}(\alpha) :
\overline{G} \longrightarrow \overline{G}$ est d\'efinie par la formule (16),
$$\overline{\Psi}(\alpha)(\overline g)= s_x(\Psi(\alpha)(p(\overline
g))\varphi(\overline g)$$ et que $\varphi\circ s_x=0$, on pourra r\'e\'ecrire
l'expression pr\'ec\'edente sous la forme  :
\begin{eqnarray}
F_x(\sigma(\gamma_{_{1}}), \sigma(\gamma_{_{2}}))
&=&\overline{\Psi}(\sigma(\gamma_{_{1}}))(h_{_{x}}(\gamma_{_{2}})^{-1})
s_x(h(\gamma_{_{1}})^{-1})s_x(h(\gamma_{_{1}}\gamma_{_{2}}))
k(\gamma_{_{1}}, \gamma_{_{2}})\nonumber\\
&=&
\overline{\Psi}(\sigma(\gamma_{_{1}}))(h_{_{x}}(\gamma_{_{2}})^{-1})h_{_{x}}
(\gamma_{_{1}})^{-1}h_{_{x}}(\gamma_{_{1}}\gamma_{_{2}})k(\gamma_{_{1}},
\gamma_{_{2}})\nonumber
\end{eqnarray}
o\`u $k : \Gamma^2\longrightarrow \mathbb R$ d\'esigne la cocha\^\i ne born\'ee
d\'efinie par,
\begin{eqnarray}
k(\gamma_{_{1}}, \gamma_{_{2}})=
-c_x(h(\gamma_{_{1}})^{-1}, h(\gamma_{_{1}}\gamma_{_{2}}))-
c_x(\Psi(\sigma(\gamma_{_{1}}))(h(\gamma_{_{2}})^{-1}),
h(\gamma_{_{1}})^{-1}h(\gamma_{_{1}}\gamma_{_{2}}))
\nonumber
\end{eqnarray}

Il est clair que pour aboutir \`a la formule (28), il suffit qu'on d\'emontre
que la $2$-cocha\^\i ne born\'ee $k : \Gamma^2\longrightarrow \mathbb R$ est
\'egale au crochet de dualit\'e, $k(\gamma_{_{1}}, \gamma_{_{2}})= <c_x,
{T}(\gamma_{_{1}}, \gamma_{_{2}})>$, o\`u $T:\Gamma^2\longrightarrow
Z_2^{\ell_1}(G, \mathbb R)$ d\'esigne  un repr\'esentant de la cocha\^\i ne
born\'ee $\overline{T}$ qui est d\'efinie par l'expression (27).

En effet, puisque $c_x : G^2\longrightarrow \mathbb R$ est un $2$-cocycle
born\'e homog\`ene ceci nous permet  d'\'ecrire  pour tous
$\alpha=\Psi(\sigma(\gamma_{_{1}}))(h(\gamma_{_{2}})^{-1}),
\beta=h(\gamma_{_{1}})^{-1}$ et $\gamma=h(\gamma_{_{1}}\gamma_{_{2}})\in G$ qu'on a :
\begin{eqnarray}
0&=&dc_x(\alpha, \beta, \gamma)=dc_x(\Psi(\sigma(\gamma_{_{1}}))(h(\gamma_{_{2}})^{-1}),
h(\gamma_{_{1}})^{-1},
h(\gamma_{_{1}}\gamma_{_{2}}))\nonumber\\
&=&c_x(h(\gamma_{_{1}})^{-1}, h(\gamma_{_{1}}\gamma_{_{2}}))-
c_x(\Psi(\sigma(\gamma_{_{1}}))(h(\gamma_{_{2}})^{-1})h(\gamma_{_{1}})^{-1},
h(\gamma_{_{1}}\gamma_{_{2}})) \nonumber\\
&+&
c_x(\Psi(\sigma(\gamma_{_{1}}))(h(\gamma_{_{2}})^{-1}),
h(\gamma_{_{1}})^{-1}
h(\gamma_{_{1}}\gamma_{_{2}})) -
c_x(\Psi(\sigma(\gamma_{_{1}}))(h(\gamma_{_{2}})^{-1}),
h(\gamma_{_{1}})^{-1})
\nonumber\\
&=&[c_x(h(\gamma_{_{1}})^{-1}, h(\gamma_{_{1}}\gamma_{_{2}})) +
c_x(\Psi(\sigma(\gamma_{_{1}}))(h(\gamma_{_{2}})^{-1}), h(\gamma_{_{1}})^{-1}
h(\gamma_{_{1}}\gamma_{_{2}})]\nonumber\\
&-&[c_x(\Psi(\sigma(\gamma_{_{1}}))(h(\gamma_{_{2}})^{-1})
h(\gamma_{_{1}})^{-1}, h(\gamma_{_{1}}\gamma_{_{2}})) +
c_x(\Psi(\sigma(\gamma_{_{1}}))(h(\gamma_{_{2}})^{-1}),
h(\gamma_{_{1}})^{-1})] \nonumber \\
&=&-k(\gamma_{_{1}}, \gamma_{_{2}}) \nonumber\\
&-&[c_x(\Psi(\sigma(\gamma_{_{1}}))(h(\gamma_{_{2}})^{-1})
h(\gamma_{_{1}})^{-1}, h(\gamma_{_{1}}\gamma_{_{2}})) +
c_x(\Psi(\sigma(\gamma_{_{1}}))(h(\gamma_{_{2}})^{-1}),
h(\gamma_{_{1}})^{-1})]. \nonumber\end{eqnarray}

Enfin, en remarquant que  le $2$-cocycle born\'e $c_x$ est  homog\`ene il en
r\'esulte que  pour tous $g$ et $h\in G$ on a $c_x(gh^{-1}, h)=- c_x(g,
h^{-1})$. De m\^eme, en remarquant que nous avons \'etabli au d\'ebut de cette
d\'emonstration que
$$\Psi(\sigma(\gamma_{_{1}}))(h(\gamma_{_{2}})^{-1})h(\gamma_{_{1}})^{-1}=f(\sigma(\gamma_{_{1}}),
\sigma(\gamma_{_{2}}))h(\gamma_1\gamma_2)^{-1}$$   ceci entra\^\i ne :
\begin{eqnarray}
k(\gamma_{_{1}}, \gamma_{_{2}})&=&
-[c_x(f(\sigma(\gamma_{_{1}}), \sigma(\gamma_{_{2}}))h(\gamma_{_{1}}\gamma_{_{2}})^{-1},
h(\gamma_{_{1}}\gamma_{_{2}})) +
c_x(\Psi(\sigma(\gamma_{_{1}}))(h(\gamma_{_{2}})^{-1}),
h(\gamma_{_{1}})^{-1})] \nonumber \\
&=&
- [-c_x(f(\sigma(\gamma_{_{1}}), \sigma(\gamma_{_{2}})),
h(\gamma_{_{1}}\gamma_{_{2}})^{-1}) +
c_x(\Psi(\sigma(\gamma_{_{1}}))(h(\gamma_{_{2}})^{-1}),
h(\gamma_{_{1}})^{-1})]. \nonumber
\end{eqnarray}

D'o\`u,  pour tous les \'el\'ements $\gamma_{_{1}}$ et $
\gamma_{_{2}}\in\Gamma$ on a $k(\gamma_{_{1}}, \gamma_{_{2}})= <c_x,
T(\gamma_{_{1}}, \gamma_{_{2}})>$.\end{proof}

\begin{proof}[D\'emonstration de la proposition 11]
Pour \'etablir  l'\'egalit\'e : $\sigma^*(\overline{\theta}_{_{\Psi, f}})
=d\overline{T}$, nous allons d\'evelopper l'expression de la cocha\^\i ne
born\'ee r\'eelle $<c_x, \sigma^*({\theta}_{_{\Psi, f}})>$.

Pour cela, nous allons utiliser la formule (28) \'etablie dans l'affirmation
pr\'ec\'edente et le fait que $F_x=s_x\circ f : \Pi^2\rightarrow \overline{G}$
contr\^ole la d\'eviation de l'application $\overline{\Psi} : \Pi\rightarrow
\mathrm{Sym}(\overline{G}, \mathbb R)$ \`a \^etre un homomorphisme (cf.  lemme 4 formule
(4)).

D'abord, pour tous $\gamma_{_{1}}$, $\gamma_{_{2}}$ et  $\gamma_{_{3}}\in
\Gamma$ appliquons l'expression (28) au couple $\gamma_{_{1}}\gamma_{_{2}}$
et $\gamma_{_{3}}$  :
\begin{eqnarray}
F_x(\sigma(\gamma_{_{1}}\gamma_{_{2}}), \sigma(\gamma_{_{3}}))&=&
\overline{\Psi}(\sigma(\gamma_{_{1}})\sigma(\gamma_{_{2}}))(h_{{x}}(\gamma_{_{3}})^{-1})
h_{{x}}(\gamma_{_{1}}\gamma_{_{2}})^{-1}
h_{{x}}(\gamma_{_{1}}\gamma_{_{2}}\gamma_{_{3}})
<c_x, T(\gamma_{_{1}}\gamma_{_{2}},\gamma_{_{3}})>
\nonumber\\
&=&F_x(\sigma(\gamma_{_{1}}), \sigma(\gamma_{_{2}}))^{-1}
\overline{\Psi}(\sigma(\gamma_{_{1}}))\circ\overline{\Psi}(\sigma(\gamma_{_{2}}))(
h_{_{x}}(\gamma_{_{3}})^{-1})
[F_x(\sigma(\gamma_{_{1}}), \sigma(\gamma_{_{2}})) \nonumber\\
&&h_{_{x}}(\gamma_{_{1}}\gamma_{_{2}})^{-1}]
h_{_{x}}(\gamma_{_{1}}\gamma_{_{2}}\gamma_{_{3}})
<c_x, T(\gamma_{_{1}}\gamma_{_{2}}, \gamma_{_{3}})>\nonumber\\
&=&F_x(\sigma(\gamma_{_{1}}), \sigma(\gamma_{_{2}}))^{-1}
\overline{\Psi}(\sigma(\gamma_{_{1}}))\circ\overline{\Psi}(\sigma(\gamma_{_{2}}))(
h_{_{x}}(\gamma_{_{3}})^{-1}) \nonumber\\
&&[\overline{\Psi}(\sigma(\gamma_{_{1}}))(h_{_{x}}(\gamma_{_{2}})^{-1})h_{_{x}}(
\gamma_{_{1}})^{-1}]
h_{{x}}(\gamma_{_{1}}\gamma_{_{2}}\gamma_{_{3}}) \nonumber\\
&&<c_x, T(\gamma_{_{1}}, \gamma_{_{2}})>
<c_x, T(\gamma_{_{1}}\gamma_{_{2}}, \gamma_{_{3}})>\nonumber
\end{eqnarray}

Observons que si   on multiplie  la derni\`ere ligne ci-dessus par
l'\'el\'ement $F_x(\sigma(\gamma_{_{1}}), \sigma(\gamma_{_{2}}))$ (depuis la
gauche) on obtient l'expression,
\begin{eqnarray}
F_x(\sigma(\gamma_{_{1}}), \sigma(\gamma_{_{2}}))F_x(\sigma(\gamma_{_{1}}
\gamma_{_{2}}), \sigma(\gamma_{_{3}}))&=&
\overline{\Psi}(\sigma(\gamma_{_{1}}))[
\overline{\Psi}(\sigma(\gamma_{_{2}}))(h_{{x}}(\gamma_{_{3}})^{-1})]\nonumber \\
&&\overline{\Psi}(\sigma(\gamma_{_{1}}))(h_{{x}}(\gamma_{_{2}})^{-1})
h_{{x}}(\gamma_{_{1}})^{-1}
h_{{x}}(\gamma_{_{1}}\gamma_{_{2}}\gamma_{_{3}})\nonumber\\
&&<c_x, T(\gamma_{_{1}}, \gamma_{_{2}})
<c_x, T(\gamma_{_{1}}\gamma_{_{2}}, \gamma_{_{3}})>\nonumber
\end{eqnarray}

Ainsi, en utilisant le fait que la d\'eviation de la bijection
$\overline{\Psi}(\sigma(\gamma_{_{1}})) : \overline{G}\rightarrow \overline{G}$
\`a \^etre un homomorphisme est contr\^ol\'ee par l'expression (11), $\forall
\bar g, \bar h\in \overline{G},$
$$\overline{\Psi}(\sigma(\gamma_{_{1}}))(\bar g)\overline{\Psi}(\sigma(\gamma_{_{1}}))(\bar h) =
\overline{\Psi}(\sigma(\gamma_{_{1}}))(\bar{g}\bar{h})[(\Psi_s(\sigma(\gamma_{_{1}})))^*(c_x)(p(\bar g), p(\bar h))][
c_x(p(\bar g), p(\bar h))^{-1}]
$$
on pourra \'ecrire, \begin{eqnarray} F_x(\sigma(\gamma_{_{1}}),
\sigma(\gamma_{_{2}}))F_x(\sigma(\gamma_{_{1}} \gamma_{_{2}}),
\sigma(\gamma_{_{3}})) &=&\overline{\Psi}(\sigma(\gamma_{_{1}}))[
\overline{\Psi}(\sigma(\gamma_{_{2}}))(h_{_{x}}(\gamma_{_{3}})^{-1})h_{_{x}}(
\gamma_{_{2}})^{-1}]\nonumber\\
&&h_{_{x}}(\gamma_{_{1}})^{-1}
h_{_{x}}(\gamma_{_{1}}\gamma_{_{2}}\gamma_{_{3}})<c_x, T(\gamma_{_{1}},
\gamma_{_{2}})> \nonumber\\
&&<c_x, T(\gamma_{_{1}}\gamma_{_{2}}, \gamma_{_{3}})>X(\gamma_{_{1}},
\gamma_{_{2}}, \gamma_{_{3}})
\end{eqnarray}
o\`u $X : \Gamma^3\rightarrow \mathbb R$ d\'esigne la cocha\^\i ne born\'ee
r\'eelle d\'efinie par l'expression suivante :
\begin{eqnarray}
X(\gamma_{_{1}}, \gamma_{_{2}}, \gamma_{_{3}})&=&
(\Psi_s(\sigma(\gamma_{_{1}})))^*(c_x)(\Psi_s(\sigma(\gamma_{_{2}}))(h(\gamma_{_{3}})^{-1}),
h(\gamma_{_{2}})^{-1})\nonumber \\
&& -c_x(\Psi_s(\sigma(\gamma_{_{2}}))(h(\gamma_{_{3}})^{-1}),
h(\gamma_{_{2}})^{-1}) \nonumber
\end{eqnarray}

Remarquons que dans l'expression  (29) si on applique la formule (28) on
pourra  remplacer l'\'el\'ement
$\overline{\Psi}(\sigma(\gamma_{_{2}}))(h_{_{x}}(\gamma_{_{3}})^{-1})h_{_{x}}(
\gamma_{_{2}})^{-1}$ (mis entre crochets dans (29)) par l'\'el\'ement
$F_x(\sigma(\gamma_{_{2}}),
\sigma(\gamma_{_{3}}))h_{_{x}}(\gamma_{_{2}}\gamma_{_{3}})^{-1}<c_x,
T(\gamma_{_{2}}, \gamma_{_{3}})>^{-1} $ on obtient :
\begin{eqnarray}
F_x(\sigma(\gamma_{_{1}}), \sigma(\gamma_{_{2}}))F_x(\sigma(\gamma_{_{1}}
\gamma_{_{2}}),
\sigma(\gamma_{_{3}}))&=&
\overline{\Psi}(\sigma(\gamma_{_{1}}))[F_x(\sigma(\gamma_{_{2}}),
\sigma(\gamma_{_{3}}))h_{_{x}}(\gamma_{_{2}}\gamma_{_{3}})^{-1}<c_x,
T(\gamma_{_{2}}, \gamma_{_{3}})>^{-1}]
\nonumber\\
&&h_{_{x}}(\gamma_{_{1}})^{-1}
h_{_{x}}(\gamma_{_{1}}\gamma_{_{2}}\gamma_{_{3}})
<c_x, T(\gamma_{_{1}},
\gamma_{_{2}})>
\nonumber\\
&&<c_x, T(\gamma_{_{1}}\gamma_{_{2}}, \gamma_{_{3}})>
X(\gamma_{_{1}}, \gamma_{_{2}}, \gamma_{_{3}})\nonumber\\
&=&
\overline{\Psi}(\sigma(\gamma_{_{1}}))[F_x(\sigma(\gamma_{_{2}}),
\sigma(\gamma_{_{3}}))h_{_{x}}(\gamma_{_{2}}\gamma_{_{3}})^{-1}]
h_{_{x}}(\gamma_{_{1}})^{-1}
\nonumber\\
&&
h_{_{x}}(\gamma_{_{1}}\gamma_{_{2}}\gamma_{_{3}})
<c_x, T(\gamma_{_{2}}, \gamma_{_{3}})>^{-1}
<c_x, T(\gamma_{_{1}}, \gamma_{_{2}})>\nonumber\\
&&<c_x, T(\gamma_{_{1}}\gamma_{_{2}}, \gamma_{_{3}})>
X(\gamma_{_{1}}, \gamma_{_{2}}, \gamma_{_{3}})
\nonumber\\
&=&\overline{\Psi}(\sigma(\gamma_{_{1}}))(F_x(\sigma(\gamma_{_{2}}),
\sigma(\gamma_{_{3}}))
[\overline{\Psi}(\sigma(\gamma_{_{1}}))(h_{_{x}}(\gamma_{_{2}}\gamma_{_{3}})^{-1})
\\
&& h_{_{x}}(\gamma_{_{1}})^{-1}
h_{_{x}}(\gamma_{_{1}}\gamma_{_{2}}\gamma_{_{3}})]Y(\gamma_{_{1}}, \gamma_{_{2}},
\gamma_{_{3}})<c_x, T(\gamma_{_{2}}, \gamma_{_{3}})>^{-1}\nonumber\\
&&<c_x, T(\gamma_{_{1}}, \gamma_{_{2}})>
<c_x, T(\gamma_{_{1}}\gamma_{_{2}}, \gamma_{_{3}})>
X(\gamma_{_{1}}, \gamma_{_{2}}, \gamma_{_{3}})
\nonumber
\end{eqnarray}
o\`u $Y : \Gamma^3\rightarrow \mathbb R$ d\'esigne la  cocha\^\i ne born\'ee
r\'eelle d\'efinie par l'expression,
\begin{eqnarray}
Y(\gamma_{_{1}}, \gamma_{_{2}}, \gamma_{_{3}})&=&
c_x(f(\sigma(\gamma_{_{2}}), \sigma(\gamma_{_{3}})), h(\gamma_{_{2}}
\gamma_{_{3}})^{-1}) \nonumber\\
&&-
(\Psi_s(\sigma(\gamma_{_{1}})))^*(c_x)(f(\sigma(\gamma_{_{2}}), \sigma(\gamma_{_{3}})),
h(\gamma_{_{2}}\gamma_{_{3}})^{-1}) \nonumber
\end{eqnarray}
d\'eduite de la formule (17) qui contr\^ole la d\'eviation de la bijection
$\overline{\Psi}(\sigma(\gamma_{_{1}})) : \overline{G} \rightarrow
\overline{G}$ \`a \^etre un homomorphisme.

Ci-dessous, pour d\'evelopper l'expression de  la cocha\^\i ne ab\'elienne
$\sigma^*(K_{_{x, \overline{\Psi}}}) : \Gamma^3\rightarrow Z(\overline G)$ qui
est \'egale \`a,
$$\overline{\Psi}(\sigma(\gamma_{_{1}})(
F_x(\sigma(\gamma_{_{2}}), \sigma(\gamma_{_{3}}))) F_x(\sigma(\gamma_{_{1}}),
\sigma(\gamma_{_{2}}\gamma_{_{3}})) F_x(\sigma(\gamma_{_{1}}\gamma_{_{2}}),
\sigma(\gamma_{_{3}}))^{-1} F_x(\sigma(\gamma_{_{1}}),
\sigma(\gamma_{_{2}}))^{-1}$$ nous allons remplacer  l'\'el\'ement
$\overline{\Psi}(\sigma(\gamma_{_{1}}))(h_{_{x}}(\gamma_{_{2}}\gamma_{_{3}})^{-1})
h_{_{x}}(\gamma_{_{1}})^{-1} h_{_{x}}(\gamma_{_{1}}\gamma_{_{2}
}\gamma_{_{3}})$ (mis entre crochets dans (30)) par l'\'el\'ement
$F_x(\sigma(\gamma_{_{1}}), \sigma(\gamma_{_{2}}\gamma_{_{3}}))) <c_x,
T(\gamma_{_{1}}, \gamma_{_{2}}\gamma_{_{3}})>^{-1}$ pour obtenir  :
\begin{eqnarray}
F_x(\sigma(\gamma_{_{1}}), \sigma(\gamma_{_{2}}))F_x(\sigma(\gamma_{_{1}}
\gamma_{_{2}}),
\sigma(\gamma_{_{3}}))&=&
\overline{\Psi}(\sigma(\gamma_{_{1}}))(F_x(\sigma(\gamma_{_{2}}),
\sigma(\gamma_{_{3}}))
F_x(\sigma(\gamma_{_{1}}), \sigma(\gamma_{_{2}}\gamma_{_{3}})))\nonumber\\
&&<c_x,  T(\gamma_{_{1}}, \gamma_{_{2}}\gamma_{_{3}})>^{-1}
<c_x, T(\gamma_{_{1}},
\gamma_{_{2}})> \nonumber \\
&& <c_x, T(\gamma_{_{2}}, \gamma_{_{3}})>^{-1}
<c_x, T(\gamma_{_{1}}\gamma_{_{2}}, \gamma_{_{3}})>\nonumber\\
&& X(\gamma_{_{1}}, \gamma_{_{2}}, \gamma_{_{3}})
Y(\gamma_{_{1}}, \gamma_{_{2}}, \gamma_{_{3}})
\end{eqnarray}

Mais comme la repr\'esentation ext\'erieure $\theta : \Pi\rightarrow Out(G)$
est d\'efinie  \`a partir d'une extension de groupes,  son $3$-cocycle
d'obstruction $K_{_{\Psi, f}}$ est nul. Ainsi, si on applique la formule (19)
qui donne
$$K_{_{x, \overline{\Psi}}} = \varphi_*(K_{_{x, \overline{\Psi}}}) +
(s_x)_*(K_{_{\Psi, f}}) \in Z(\overline G)\simeq \mathbb R\oplus Z(G)$$ on en
d\'eduit qu'en fait la cocha\^\i ne  $K_{_{x,
\overline{\Psi}}}=\varphi_*(K_{_{x, \overline{\Psi}}}) : \Pi^3\rightarrow
\mathbb R$. D'autre part, en utilisant simultan\'ement l'expression (31) et
l'expression  (20) de la proposition 7 on peut \'ecrire dans le groupe
commutatif additif $(Z(\overline{G}), +)$  que :
\begin{eqnarray}
<c_x, \sigma^*(\theta_{_{\Psi, f}})(\gamma_{_{1}}, \gamma_{_{2}}, \gamma_{_{3}})>
&=&\varphi(K_{_{x, \overline{\Psi}}}(\sigma(\gamma_{_{1}}), \sigma(\gamma_{_{2}}),
\sigma(\gamma_{_{3}}))) \nonumber \\
&=& K_{_{x, \overline{\Psi}}}(\sigma(\gamma_{_{1}}), \sigma(\gamma_{_{2}}),
\sigma(\gamma_{_{3}}))   \nonumber \\
&=& <c_x,  T(\gamma_{_{1}}, \gamma_{_{2}}\gamma_{_{3}})> -
<c_x, T(\gamma_{_{1}}, \gamma_{_{2}})> +
<c_x, T(\gamma_{_{2}}, \gamma_{_{3}})>
\nonumber \\
&-& <c_x, T(\gamma_{_{1}}\gamma_{_{2}}, \gamma_{_{3}})>
- X(\gamma_{_{1}}, \gamma_{_{2}}, \gamma_{_{3}}) -
Y(\gamma_{_{1}}, \gamma_{_{2}}, \gamma_{_{3}})
\end{eqnarray}

Il devient maintenant clair que pour achever  la preuve de la proposition 11 il
suffit qu'on d\'emontre que le second membre de l'expression  (32) est \'egal
au crochet de dualit\'e  $$<c_x, dT(\gamma_{_{1}}, \gamma_{_{2}},
\gamma_{_{3}})>$$

En effet, si on  fait la somme des deux cocha\^\i nes $X :
\Gamma^3\longrightarrow \mathbb R$ et $Y : \Gamma^3\longrightarrow \mathbb R$
tout en appliquant l'expression (27) qui d\'efinit la cocha\^\i ne $\overline
T$ on voit facilement que pour tous les \'el\'ements $\gamma_{_{1}}$,
$\gamma_{_{2}}$ et $\gamma_{_{3}}\in \Gamma$ on a,
\begin{eqnarray}
X(\gamma_{_{1}}, \gamma_{_{2}}, \gamma_{_{3}})+
Y(\gamma_{_{1}}, \gamma_{_{2}}, \gamma_{_{3}})
&=& <c_x, T(\gamma_{_{2}}, \gamma_{_{3}})-
 \Big(\Psi_s(\sigma(\gamma_{_{1}}))\Big)_*(T(\gamma_{_{2}},
\gamma_{_{3}}))>\nonumber
\end{eqnarray}

Maintenant, gr\^ace \`a cette remarque on peut r\'e\'ecrire le second membre de
l'expression (32) sous la forme :

\begin{eqnarray}
<c_x, \sigma^*(\theta_{_{\Psi, f}})(\gamma_{_{1}}, \gamma_{_{2}}, \gamma_{_{3}})>&=&
<c_x,  T(\gamma_{_{1}}, \gamma_{_{2}}\gamma_{_{3}})> -
<c_x, T(\gamma_{_{1}}, \gamma_{_{2}})> \nonumber \\
&+& <c_x, T(\gamma_{_{2}}, \gamma_{_{3}})>
- <c_x, T(\gamma_{_{1}}\gamma_{_{2}}, \gamma_{_{3}})> \nonumber \\
&-& <c_x, T(\gamma_{_{2}}, \gamma_{_{3}}) -
\Big(\Psi_s(\sigma(\gamma_{_{1}}))\Big)_*(T)(\gamma_{_{2}},
\gamma_{_{3}})>\nonumber \\
&=& <c_x,  T(\gamma_{_{1}}, \gamma_{_{2}}\gamma_{_{3}})> -
<c_x, T(\gamma_{_{1}}, \gamma_{_{2}})>
\nonumber \\
&-& <c_x, T(\gamma_{_{1}}\gamma_{_{2}}, \gamma_{_{3}})>
+
<c_x, \Big(\Psi_s(\sigma(\gamma_{_{1}}))\Big)_*(T)(\gamma_{_{2}},
\gamma_{_{3}})>\nonumber \\
 &=& <c_x, dT(\gamma_{_{1}}, \gamma_{_{2}}, \gamma_{_{3}})>\nonumber
 \end{eqnarray}

Finalement, en passant  dans l'espace d'homologie $\ell_1$-r\'eduite
$\overline{H}_2^{\ell_1}(G, \mathbb R)$  nous  obtenons l'expression
recherch\'ee
 $\sigma^*(\overline{\theta}_{_{\Psi, f}}) =
 d\overline{T}$.\end{proof}

Le r\'esultat de la proposition 11 nous permet maintenant de d\'eduire les deux
corollaires importants suivants.

\begin{corollary}
La restriction  de la cocha\^\i ne born\'ee $\overline{T} : \Gamma^2\rightarrow
\overline{H}_2^{\ell_1}(G, \mathbb R)$  sur le sous-groupe normal $i(G)\subset
\Gamma$ repr\'esente la classe de cohomologie  born\'ee $\mathbf{g}_{_{2}}$,
$$[i^*(\overline{T})]=
[\overline{\mathbf{m}}_{_{2}}]=\mathbf{g}_{_{2}}\in H_b^2(G, \overline{H}_2^{\ell_1}(G, \mathbb R)).$$
\end{corollary}

\begin{proof}[D\'emonstration] En effet, puisque pour tout $g\in G$ on a $h(g)=g$ ;
donc en \'evaluant la cocha\^\i ne  $\overline{T} : \Gamma^2\rightarrow
\overline{H}_2^{\ell_1}(G, \mathbb R)$ sur tous les couples d'\'el\'ements
$(g_{_{1}}, g_{_{2}})\in i(G)\times i(G)\subset \Gamma\times \Gamma$ on voit
ais\'ement que $ i^*(\overline{T})(g_{_{1}}, g_{_{2}}) = -
\overline{\mathbf{m}}_{_{2}}(g_{_{2}}^{-1},  g_{_{1}}^{-1}).$

D'autre part, puisque pour toute classe de cohomologie born\'ee $x=[c_x]\in
H_b^2(G, \mathbb R)$ et pour tout couple d'\'el\'ements $g_{_{1}}$ et
$g_{_{2}}\in G$ on a les deux relations,
$$c_x(g_{_{1}}^{-1},  g_{_{2}}^{-1})=-c_x(g_{_{1}},  g_{_{2}}) \quad \textrm{ et } \quad
<c_x, \mathbf{m}_{_{2}}(g_{_{1}},  g_{_{2}})>=c_x(g_{_{1}},  g_{_{2}}) $$ on en
d\'eduit l'\'egalit\'e  $i^*(\overline{T})(g_{_{1}}, g_{_{2}}) =-
\overline{\mathbf{m}}_{_{2}}(g_{_{2}}^{-1},
g_{_{1}}^{-1})=\overline{\mathbf{m}}_{_{2}}(g_{_{1}}, g_{_{2}})$ qui entra\^\i
ne l'\'egalit\'e en cohomologie :
$[i^*(\overline{T})]=[\overline{\mathbf{m}}_{_{2}}]=\mathbf{g}_{_{2}}$.
\end{proof}

\begin{corollary}
La classe de cohomologie born\'ee $\mathbf{g}_{_{2}}\in H_b^2(G,
\overline{H}_2^{\ell_1}(G, \mathbb R))$ est invariante par l'action d\'efinie
par la repr\'esentation ext\'erieure $\theta : \Pi\rightarrow Out(G)$
associ\'ee \`a l'extension de groupes $1\rightarrow G\stackrel{i}{\rightarrow}
\Gamma\stackrel{\sigma}{\rightarrow}\Pi\rightarrow1$.
\end{corollary}

\begin{proof}[D\'emonstration] On proc\`ede comme dans la d\'emonstration de la
proposition 4 rappel\'ee ci-dessus (cf. section 3) et qui correspond au
corollaire 2 de \cite{Bou3}.

Plus pr\'ecis\'ement, il suffit qu'on regarde la $2$-cocha\^\i ne born\'ee
$\overline{T} : \Gamma^2 \rightarrow \overline{H}_2^{\ell_1}(G, \mathbb R)$ qui
est d\'efinie par l'expression (27) comme \'etant une $3$-cocha\^\i ne
homog\`ene $\Gamma$-invariante~:
$$
\overline{T}\in(\mathcal{L}(\mathbb R[\Gamma^3],
\overline{H}_2^{\ell_1}(G, \mathbb R)))^\Gamma \simeq \mathcal{L}(\mathbb R[\Gamma^2],
\overline{H}_2^{\ell_1}(G, \mathbb R))$$

Ensuite, en utilisant l'homomorphisme compos\'e
$$
\begin{array}
{ccccccccc}
(\mathcal{L}(\mathbb R[\Gamma^3], \overline{H}_2^{\ell_1}(G, \mathbb R)))^\Gamma
&\hookrightarrow&(\mathcal{L}(\mathbb R[\Gamma^3], \overline{H}_2^{\ell_1}(G, \mathbb
R)))^G&  \\
&\stackrel{i^*}{\rightarrow}& (\mathcal{L}(\mathbb R[G^3],
\overline{H}_2^{\ell_1}(G, \mathbb R)))^G&\simeq \mathcal{L}(\mathbb R[G^2],
\overline{H}_2^{\ell_1}(G, \mathbb R))
\end{array}
$$
on  d\'eduit que l'image  de la cocha\^\i ne $\overline{T}$ {\it via } la
composition de ces deux morphismes induit  une $2$-cocha\^\i ne
$\Pi$-invariante sur le groupe $G$. Ainsi, comme d'apr\`es le corollaire 5 on
sait que $i^*(\overline{T})= \overline{\mathbf{m}}_{_{2}}$ est un $2$-cocycle
ceci implique  que finalement la classe de cohomologie born\'ee
$\mathbf{g}_{_{2}}=[\overline{\mathbf{m}}_{_{2}}]\in H_b^2(G,
\overline{H}_2^{\ell_1}(G, \mathbb R))^\Pi$ est $\Pi$-invariante.
\end{proof}

Si on d\'esigne par $(E_{_{r, \ell_1}}^{p, q}, d_{_{r, \ell_1}})$ la suite
sp\'ectrale de Hochschild-Serre associ\'ee \`a l'extension $1 \longrightarrow
G\stackrel{i}{\longrightarrow}\Gamma\stackrel{\sigma}{\longrightarrow}
\Pi\longrightarrow 1$  en cohomologie born\'ee \`a coefficients dans le
$\Pi$-module de Banach $\overline{H}_2^{\ell_1}(G, \mathbb R)$, on voit  que la
classe  $\mathbf{g}_{_{2}}$ induit un \'el\'ement du terme $E_{_{3,
\ell_1}}^{0, 2}$, et que la classe  $[\theta]$ induit aussi un \'el\'ement du
terme $E_{_{3, \ell_1}}^{3, 0}$.

Pour justifier  ces deux faits  observons que puisque le terme $E_{_{2, \ell_1}}^{0,2}=E_{_{3, \ell_1}}^{0,2}$ (cf. lemme 4) et puisque la classe de cohomologie
$\mathbf{g}_{_{2}}\in E_{_{2, \ell_1}}^{0,2}\stackrel{\sim}{\rightarrow}H_b^2(G,\overline{H}_2^{\ell_1}(G,
\mathbb R))^\Pi$, donc $\mathbf{g}_{_{2}}$ repr\'esente une classe de cohomologie qu'on notera aussi
$\mathbf{g}_{_{2}}\in E_{_{3, \ell_1}}^{0, 2}$. De m\^eme,
puisque le terme $E_{_{2, \ell_1}}^{3, 0}=E_{_{3, \ell_1}}^{3, 0}$ (cf. lemme 4) avec $E_{_{2, \ell_1}}^{3,
0}\stackrel{\sim}{\rightarrow}H_b^3(\Pi, \overline{H}_2^{\ell_1}(G, \mathbb R))$ donc la classe de cohomologie $[\theta]\in H_b^3(\Pi, \overline{H}_2^{\ell_1}(G, \mathbb R))$  repr\'esente une classe de cohomologie qu'on notera aussi
$[\theta]\in E_{_{3, \ell_1}}^{3, 0}$.

Rappelons aussi que d'apr\`es la d\'efinition  des termes $E_3^{n, 0}$ et $E_3^{n, 2}$ (cf. 3.2.1) on a~:
$$ E_3^{n, 0} = \displaystyle\frac{Z_3^{n, 0}}{Z_2^{n+1, -1} + B_2^{n, 0}}  \quad \textrm{ et } \quad E_3^{n, 2} =
  \displaystyle\frac{Z_3^{n, 2}}{Z_2^{n+1, 1} + B_2^{n,2}}$$

Ainsi, puisque la diff\'erentielle totale $d_* = d_\Pi + (-1)^pd_U$ du complexe diff\'erentiel double  filt\'e verticalement,  $$ \forall p, q\in\mathbb N, \quad K^{p, q}:=C_b^p(\Pi, U^q) \qquad \textrm{ o\`u }\qquad U^q:= {\mathcal L}_{G}
(\mathbb R[\Gamma^{q+1}], V)$$
 envoie  le sous-espace  $Z_3^{n, 2}$ dans le sous-espace  $Z_3^{n+3, 0}$ (i.e. $d_{_{n+2}}(Z_3^{n, 2}) \subseteq Z_3^{n+3, 0}$) nous avons d\'efini
la diff\'erentielle $d_3^{n,2} : E_3^{n, 2}\longrightarrow E_3^{n+3, 0}$  en  passant  aux espaces  quotients par l'expression (cf. 3.2.1), $$\forall x\in Z_3^{n, 2}, \qquad d_3^{n,2}([x]) := [d_{_{n+2}}(x)]$$

Donc, si en particulier on prend $V=\overline{H}_2^{\ell_1}(G, \mathbb R)$ et $n=0$ on voit  que la diff\'erentielle   $d_{_{3, \ell_1}}^{0,2} : E_{_3, \ell_1}^{0,2}\longrightarrow E_{_{3, \ell_1}}^{3, 0}$ est induite par la diff\'erenielle totale $d_3 : Z_3^{0, 2} \longrightarrow Z_3^{3, 0}$ o\`u
\begin{eqnarray*}Z_3^{0, 2} &:=& \{x\in F_v^0\textrm{Tot}(K^{*,*})^{2} \ ; d_{_{2}}(x)\in F_v^{3}\textrm{Tot}(K^{*,*})^{3}\}\\
&=&\{x\in \textrm{Tot}(K^{*,*})^{2}\ ; d_{_{2}}(x)\in C_b^3(\Pi, U^0)\}\end{eqnarray*}
et
\begin{eqnarray*}Z_3^{3, 0}&:=& \{x\in F_v^3\textrm{Tot}(K^{*,*})^{3} \ ; d_{_{3}}(x)\in F_v^{6}\textrm{Tot}(K^{*,*})^{4}\} \\
&=& \{x\in C_b^3(\Pi, U^0)\ ; d_{_{3}}(x)=0\}\end{eqnarray*}

Avec les discussions pr\'ec\'edentes on peut maintenant d\'emontrer le th\'eor\`eme principal B.

\newtheorem*{prB}{Th\'eor\`eme principal B}\begin{prB}
La diff\'erentielle $d_{_{3, \ell_1}}^{0,2} : E_{_3, \ell_1}^{0,2}\longrightarrow E_{_{3, \ell_1}}^{3, 0}$ de la
suite spectrale de Hochschild-Serre associ\'ee \`a l'extension de groupes
discrets $1 \longrightarrow
G\stackrel{i}{\longrightarrow}\Gamma\stackrel{\sigma}{\longrightarrow}
\Pi\longrightarrow 1$ en cohomologie born\'ee \`a coefficients dans le
$\Pi$-module de Banach $\overline{H}_2^{\ell_1}(G, \mathbb R)$,  envoie la
classe $\mathbf{g}_{_{2}}\in E_{_{3, \ell_1}}^{0,2}$  sur la classe $[\theta]\in
E_{_{3, \ell_1}}^{3, 0}$.
\end{prB}

\begin{proof}[D\'emonstration] D'abord,  notons que la cocha\^\i ne
    born\'ee $\overline{T} : \Gamma^2\rightarrow \overline{H}_2^{\ell_1}(G,
    \mathbb R)$ (cf. pr. 11) peut \^etre vue comme \'el\'ement de l'espace vectoriel $C_b^0(\Pi, U^2)$  parce que on a~:
$$
\overline{T}\in\mathcal{L}(\mathbb R[\Gamma^2],
\overline{H}_2^{\ell_1}(G, \mathbb R))\simeq (\mathcal{L}(\mathbb R[\Gamma^3],
\overline{H}_2^{\ell_1}(G, \mathbb R)))^\Gamma \Longrightarrow \overline{T}\in  (\mathcal{L}(\mathbb R[\Gamma^3], \overline{H}_2^{\ell_1}(G, \mathbb
R)))^G=U^2$$

Ainsi,  comme le cobord   $d_{_{2}}\overline{T}=\sigma^*(\overline{\theta}_{_{\Psi, f}})$ (cf. pr. 11) induit un \'el\'ement de l'espace
vectoriel $C_b^3(\Pi, U^0)$ il s'ensuit  que  la cocha\^\i ne born\'ee  $\overline{T}\in Z_3^{0, 2}$, son cobord $d_{_{2}}\overline{T}\in Z_3^{3, 0}$  et que par cons\'equent
$$d_{_{3, \ell_1}}^{0, 2}([\overline{T}]) = [d_{_{2}}\overline{T}]\in E_{_{3, \ell_1}}^{3, 0}$$

D'autre part, puisque  d'apr\`es les corollaires 5 et 6,  la restriction de la cocha\^\i ne
    born\'ee $\overline{T} : \Gamma^2\rightarrow \overline{H}_2^{\ell_1}(G,
    \mathbb R)$ au sous-groupe normal $i(G)\subset \Gamma$ repr\'esente la
    classe de cohomologie born\'ee $\Pi$-invariante $\mathbf{g}_{_{2}}\in E_{3, \ell_1}^{0,
    2}=E_{2, \ell_1}^{0, 2}\stackrel{\sim}{\rightarrow}
    H_b^2(G, \overline{H}_2^{\ell_1}(G, \mathbb R))^\Pi$, et comme d'apr\`es la proposition 10 le cobord de la cocha\^\i ne
    born\'ee $\overline{T} : \Gamma^2\rightarrow \overline{H}_2^{\ell_1}(G,
    \mathbb R)$  repr\'esente la classe de cohomologie born\'ee $[\theta]\in E_{_{3, \ell_1}}^{3, 0}\stackrel{\sim}{\rightarrow}
    H_b^3(\Pi, \overline{H}_2^{\ell_1}(G, \mathbb R))$ on conclut finalement que la diff\'erentielle  $d_{3,
\ell_1}(\mathbf{g}_{_{2}}) = [\theta]\in E_{3, \ell_1}^{3,
0}$.
 \end{proof}

\subsection{Preuve du th\'eor\`eme principal C}
Soit $1 \longrightarrow
G\stackrel{i}{\longrightarrow}\Gamma\stackrel{\sigma}{\longrightarrow}
\Pi\longrightarrow 1$ une extension de groupes discrets et $\theta :
\Pi\rightarrow Out(G)$ sa repr\'esentation ext\'erieure. D'apr\`es \cite{Bou3},
il existe  trois suites spectrales de Hochschild-Serre en cohomologie born\'ee
\`a coefficients dans les  $\Pi$-modules de Banach $H_b^2(G, \mathbb R)$,
$\overline{H}_2^{\ell_1}(G, \mathbb R)$ et $\mathbb R$ (qui est trivial) dont
les termes sont d\'esign\'es respectivement par $(E_{_{r, \infty}}^{p, q},
d_{_{r, \infty}}^{p, q})$, $(E_{_{r, \ell_1}}^{p, q}, d_{_{r, \ell_1}}^{p, q})$ et
$(E_{_{r}}^{p, q}, d_{_{r}}^{p, q})$. Rappelons aussi qu'au paragraphe 3.3 nous avons
fait remarquer que les diff\'erentielles de ces trois suites spectrales
commutent avec le cup-produit dans la relation suivante,
\begin{eqnarray}
d_r^{p+p', q+q'}(x_{_{\infty}}^{p, q}\cup x_{_{\ell_1}}^{p', q'}) =
d_{_{r, \infty}}^{p, q}(x_{_{\infty}}^{p, q})\cup x_{_{\ell_1}}^{p', q'} +
(-1)^{p+q}x_{_{\infty}}^{p, q}\cup d_{_{r, \ell_1}}^{p', q'}(x_{_{\ell_1}}^{p',
q'}).\nonumber
\end{eqnarray}

De m\^eme, notons que l'identit\'e \'etablie dans le th\'eor\`eme principal A,
$$x = x\cup \mathbf{g}_{_{2}}, \qquad \forall x\in H_b^2(G, \mathbb R) $$ permet  d'obtenir un isomorphisme canonique
\begin{displaymath}
\begin{array}
{cccccc}
\cup\mathbf{g}_{_{2}} : &E_{_{2, \infty}}^{n, 0}&\rightarrow& E_{_{2}}^{n,
2}\\
 &x_{_{\infty}}& \rightarrow & x_{_{\infty}}\cup \mathbf{g}_{_{2}}
\end{array}
\end{displaymath}
dont le morphisme inverse associe \`a tout  vecteur $x\in E_{_{2}}^{n, 2}$  un
unique vecteur $x_{_{\infty}}\in E_{_{2, \infty}}^{n, 0} $ tel que, $x =
x_{_{\infty}}\cup \mathbf{g}_{_{2}}$.

Finalement, observons que puisque en cohomologie born\'ee r\'eelle le terme
$E_{_{2}}^{n, 2}$ se surjecte sur le terme $E_{_{3}}^{n, 2}$ (cf. lemme 3) et comme on a
aussi $E_{_{2}}^{n+3, 0}=E_{_{3}}^{n+3, 0}$ (cf. lemme 3),  on en d\'eduit que la
diff\'erentielle $d_3^{n, 2} : E_3^{n, 2}\rightarrow E_3^{n+3, 0}$ transforme
l'expression $x = x_{_{\infty}}\cup \mathbf{g}_{_{2}}$ comme suit,
\begin{eqnarray}
d_{_{3}}^{n, 2}(x) &=& d_{_{3}}^{n, 2}(x_{_{\infty}}\cup \mathbf{g}_{_{2}}) \nonumber \\
&=& d_{_{3, \infty}}^{n, 0}(x_{_{\infty}})\cup\mathbf{g}_{_{2}} +
(-1)^nx_{_{\infty}}\cup d_{_{3, \ell_1}}^{0, 2}(\mathbf{g}_{_{2}}) \nonumber \\
&=&0\cup\mathbf{g}_{_{2}}
+(-1)^nx_{_{\infty}}\cup [\theta] = (-1)^nx\cup[\theta].\nonumber
\end{eqnarray}

\newtheorem*{CA}{Corollaire A}\begin{CA}
L'op\'erateur de transgression $\delta : H_b^2(G, \mathbb R)^\Pi\rightarrow
H_b^3(\Pi, \mathbb R)$ associ\'e \`a la repr\'esentation ext\'erieure $\theta :
\Pi\rightarrow Out(G)$ de l'extension $1 \longrightarrow
G\stackrel{i}{\longrightarrow}\Gamma\stackrel{\sigma}{\longrightarrow}
\Pi\longrightarrow 1$ est \'egal \`a la diff\'erentielle $d_{_{3}}^{0,2} :
E_{_{3}}^{0, 2}\rightarrow E_{_{3}}^{3, 0}$.
\end{CA}

\begin{proof}[D\'emonstration] Il suffit de remarquer que d'apr\`es la proposition 7 et l'expression
(22) (cf. cor. 4), si  $c_x  : G^2\rightarrow \mathbb R$ d\'esigne un $2$-cocycle born\'e
homog\`ene  invariant par
 $\theta : \Pi\rightarrow Out(G)$ il en r\'esulte
que pour tout noyau abstrait $(\Psi, f)$ de $\theta$ l'expression suivante (cf.
(20)),
$$\varphi_*(K_{_{x, \overline{\Psi}}})= <c_x, \theta_{_{\Psi, f}}(\alpha, \beta, \gamma)>=
c_x(\Psi(\alpha)(f(\beta, \gamma), f(\alpha, \beta\gamma)) -
c_x(f(\alpha, \beta), f(\alpha\beta, \gamma))$$
d\'efinit un $3$-cocycle r\'eel born\'e sur le groupe $\Pi$ qui repr\'esente
\`a la fois les classes de cohomologie born\'ee :
$\delta([c_x])=[\varphi_*(K_{_{x, \overline{\Psi}}})]$ (cf. \cite{Bou} et
\cite{Bou1}) et $d_3([c_x])=(-1)^0[c_x]\cup [\theta]$.  \end{proof}

\noindent{\bf Remerciement-. } Je remercie le r\'ef\'er\'e anonyme de l'article pour ses pr\'ecieuses remarques qui m'ont permis de refaire l'article et de le completer par des paragraphes visant \`a clarifier  les passages de certaines d\'emonstrations,  et par cons\'equent  rendent  le contenu de l'article ind\'ependant et auto suffisant. Je tiens aussi \`a le remercier pour ses quesions int\'eressantes qui seront d\'evelopp\'ees dans de futures papiers.

\end{document}